\documentclass[11pt,reqno]{article}

\usepackage[text={152mm,230mm},centering]{geometry}
\usepackage{graphicx}
\usepackage{amssymb}
\usepackage{epstopdf}
\usepackage{bm}

\usepackage{amsmath,amsthm}
\usepackage{appendix}
\usepackage[mathscr]{eucal}
\usepackage[all]{xy}
\usepackage{mathrsfs}
\usepackage{authblk}
\usepackage{hyperref}
\usepackage{color}
\usepackage{etoolbox}
\usepackage[hang,flushmargin]{footmisc}

\usepackage{marginnote}
\let\oldmarginpar\marginpar
\renewcommand\marginpar[1]{\-\oldmarginpar[\raggedleft\footnotesize\color{red} #1]%
{\raggedright\footnotesize\color{red} #1}}
 \newcommand{\R}{{\mathbb R}}
 \newcommand{\Z}{{\mathbb Z}}
 
 \newcommand{\N}{{\mathbb  N}}
 \newcommand{\g}{{\mathfrak g}}
 \newcommand{\A}{{\mathscr  A}}
\newcommand{\h}{{\mathfrak h}}
\newcommand{\<}{\langle}
  \renewcommand{\>}{\rangle}

\newtheorem{theorem}{Theorem}[section]
\newtheorem{lemma}[theorem]{Lemma}
\newtheorem{conjecture}[theorem]{Conjecture}
\newtheorem{corollary}[theorem]{Corollary}
\newtheorem{proposition}[theorem]{Proposition}

\newtheorem{fact}[theorem]{Fact}

\newtheorem{definition-lemma}[theorem]{Definition-Lemma}
\newtheorem{definition-theorem}[theorem]{Definition-Theorem}
\newtheorem{proposition-definition}[theorem]{Proposition-Definition}

\theoremstyle{definition}
\newtheorem{definition}[theorem]{Definition}
\newtheorem{example}[theorem]{Example}
\newtheorem{remark}[theorem]{Remark}
\newtheorem{notation}[theorem]{Notation}
\newtheorem{convention}[theorem]{Convention}

\newtheorem*{ack}{Acknowledgements}

\allowdisplaybreaks

\begin{document}

\title{Quantization of the minimal nilpotent orbits and\\ the quantum Hikita conjecture
\footnotetext{Email:
xjchen@scu.edu.cn (Chen), hewq@mail2.sysu.edu.cn (He), banaenoptera@163.com (Yu)
}}

\author[1,2]{Xiaojun Chen}
\author[3]{Weiqiang He}
\author[3]{Sirui Yu}

\renewcommand\Affilfont{\small}

\affil[1]{School of Mathematics, Sichuan University, Chengdu, 610064 P.R. China}
\affil[2]{Department of Applied Mathematics, New Uzbekistan University,
Tashkent, 100007 Uzbekistan}
\affil[3]{School of Mathematics, Sun Yat-sen University, Guangzhou, 510275 P.R. China}

\date{}

\maketitle
\begin{abstract}
We show that the specialized quantum D-module of the
equivariant quantum cohomology ring of the minimal resolution
of an ADE singularity
is isomorphic to the D-module of graded traces on the
minimal nilpotent orbit in the Lie algebra of the same type. This generalizes
a recent result of Shlykov [Hikita conjecture for the minimal nilpotent orbit,
to appear in Proc. AMS, \href{https://doi.org/10.1090/proc/15281}{https://doi.org/10.1090/proc/15281}]
and hence verifies in this case the quantum version
of Hikita's conjecture,
proposed by Kamnitzer, McBreen and Proudfoot [The quantum Hikita conjecture,
Advances in Mathematics 390 (2021) 107947].
We also show analogous isomorphisms for singularities of BCFG type.


\end{abstract}

\setcounter{tocdepth}{2}\tableofcontents

%=================================================

\section{Introduction}
Over the past two decades, 3d $\mathcal N=4$ mirror symmetry has attracted a lot
of attentions from both physicists and mathematicians (see, for example,
\cite{BFN,BDG,IS,N} and references therein). It is also equivalent
to the theory of \textit{symplectic duality} of Braden et al. \cite{BPW,BLPW} (see also
\cite{Kam} for a survey).
For two (possibly singular) manifolds that are symplectic dual to each
other, there are some highly nontrivial identities between
the geometry and topology of them. One of the properties predicted
by 3d $\mathcal N=4$ mirror symmetry and
symplectic duality is Hikita's conjecture.
Suppose we are given a pair of symplectic dual conical symplectic singularities,
then Hikita's conjecture relates the coordinate ring of one
symplectic variety to the cohomology ring of the symplectic
resolution of the other, which is stated as follows.

\begin{conjecture}[Hikita {\cite[Conjecture 1.3]{Hi}}]\label{conj:Hikita}
Let $X$ and $X^{!}$ be a pair of symplectic dual conical
symplectic singularities over $\mathbb{C}$.
Suppose $X^{!}$ admits
a conical symplectic resolution $\tilde{X}^{!}\rightarrow X^{!}$,
and suppose $T$ is a maximal torus of the Hamiltonian action
on $X$. Then there is an isomorphism of graded algebras
\begin{eqnarray*}
\mathrm{H}^\bullet(\tilde{X}^{!},\mathbb{C})\cong\mathbb{C}[X^T].
\end{eqnarray*}	
\end{conjecture}

In loc. cit. Hikita proved this conjecture in several cases, such as
hypertoric varieties, Spaltenstein varieties and the Hilbert schemes of points in
the plane. He then asked whether this phenomenon holds for
other examples of symplectic duality.
In \cite{KTW}, Kamnitzer et. al. proved Hikita's conjecture
for the case of Nakajima quiver varieties of type A and affine Grassmannian slices,
which are symplectic dual to each other. In loc. cit., they also
stated a conjecture of Nakajima, which generalized Hikita's conjecture
to the {\it equivariant} case (see \cite[\S8]{KTW}).
In his Ph.D. thesis \cite{Weekes}, Weekes proved this
conjecture for symplectic dual pairs studied in \cite{KTW}.
In a recent paper \cite{KS}, Krylov and Shlykov
called this conjecture the {\it Hikita-Nakajima conjecture}
and proved it for Gieseker varieties (the ADHM spaces).

In \cite{KMP},
Kamnitzer, McBreen and
Proudfoot further generalized the Hikita-Nakajima
conjecture
to the {\it quantum} case, and proved it for nilpotent cones
in ADE type Lie algebras
and hypertoric varieties.
A bit more precisely, in loc. cit., they introduced,
for a symplectic dual pair $X$ and $X^!$, two concepts:
one is the so-called {\it specialized quantum D-module},
which is induced by the equivariant quantum cohomology
of $X^!$, and the other is the {\it D-module of graded
traces} on $X$, which may be understood as the
``graded functions" (the universal source of graded traces)
of the quantized coordinate ring of $X$.
Kamnitzer et. al. conjectured that these
two types of D-modules are isomorphic, and called it
the {\it quantum Hikita conjecture}.

According to 3d mirror symmetry, the  minimal
nilpotent orbit closure $\overline{\mathcal O}_{min}$
in a simple Lie algebra
$\mathfrak{g}$ of ADE type is mirror to (or equivalently symplectic dual to)
the intersection of a Slodowy slice to the
subregular nilpotent orbit with the nilpotent cone in the same Lie algebra;
 see, for example, \cite{Webster} and \cite[Remark 10.6]{BLPW}.
This is highly related to the duality discovered by Spaltenstein
\cite{Spa} and Lusztig \cite{Lus} (see also \cite{CM} for more details).
Recall that by Brieskorn \cite{Br} and Slodowy \cite{S}, the latter
is isomorphic to the Kleinian singularity
$\mathbb C^2/\Gamma$ of the same type.
If we denote by $\widetilde{\mathbb C^2/\Gamma}$
the minimal resolution of $\mathbb C^2/\Gamma$,
then in a recent paper \cite{Sh},
Shlykov showed that
$$\mathrm H^\bullet(\widetilde{\mathbb C^2/\Gamma})
\cong\mathbb C[\overline{\mathcal O}_{min}^{\mathbb C^\times}]
$$
as graded algebras, and hence verified Hikita's conjecture in these cases.
The purpose of this paper is to
generalize his work to the quantum case.

\begin{theorem}\label{maintheorem0}
Let $\mathfrak g$ be a complex semisimple Lie algebra of ADE type,
and let $\overline{\mathcal O}_{min}$
be the closure of the minimal nilpotent orbit in $\mathfrak g$.
Let $\widetilde{{\mathbb C^2}/\Gamma}$ be the minimal resolution
of the singularity of the same type. Then the quantum Hikita
conjecture
holds for the pair $\widetilde{{\mathbb C^2}/\Gamma}$ and
$\overline{\mathcal O}_{min}$; more precisely,
\begin{enumerate}
\item[$(1)$] for $\mathbb C^2/\Gamma$ being the $A_n$ singularity,
there is an isomorphism
$$
\mathrm{QH}^\bullet_{(\mathbb C^{\times})^2}
(\widetilde{{\mathbb C^2}/\Gamma})
\cong Q(\mathscr A[\overline{\mathcal O}_{min}])$$
of graded modules over
$F_{\mathrm{reg}}\otimes \mathrm{Sym} \mathrm{H}^2_{(\mathbb C^{\times})^2}
(\widetilde{{\mathbb C^2}/\Gamma})$;

\item[$(2)$] for other types of singularity, there is an isomorphism
$$\mathrm{QH}^\bullet_{\mathbb C^{\times}}
(\widetilde{{\mathbb C^2}/\Gamma})
\cong Q(\mathscr A[\overline{\mathcal O}_{min}]),
$$
of graded modules over
$
 F_{\mathrm{reg}}\otimes \mathrm{Sym} \mathrm{H}^2_{\mathbb C^{\times}}
(\widetilde{{\mathbb C^2}/\Gamma})$,
\end{enumerate}
where $\mathrm{QH}^\bullet(-)$
is the specialized quantum D-module,
and $Q(\mathscr A(-))$ is the D-module of graded traces (see
\S\ref{sect:quatum D-module} and \ref{sect:D-module of graded traces}
for the definitions of these two concepts as well as the base rings).
\end{theorem}

In the above theorem, if we let the quantum parameter
$q=0$, then both the specialized quantum D-module
and the D-module of graded traces are algebras,
and therefore
we get the corresponding Hikita-Nakajima conjecture:

\begin{theorem}\label{maintheorem}
With the notations in Theorem \ref{maintheorem0}, the
Hikita-Nakajima
conjecture
holds for the pair $\widetilde{{\mathbb C^2}/\Gamma}$ and
$\overline{\mathcal O}_{min}$; that is,
there are isomorphisms of graded algebras:
$$
\begin{array}{ll}
\mathrm{H}^\bullet_{(\mathbb C^{\times})^2}
(\widetilde{{\mathbb C^2}/\Gamma})
\cong B(\mathscr A[\overline{\mathcal O}_{min}]),
&\mbox{if $\mathbb C^2/\Gamma$ is an $A_n$ singularity,}\\[2mm]
\mathrm{H}^\bullet_{\mathbb C^{\times}}
(\widetilde{{\mathbb C^2}/\Gamma})
\cong B(\mathscr A[\overline{\mathcal O}_{min}]),
&\mbox{otherwise,}
\end{array}
$$
where
$B(-)$ is the associated $B$-algebra (see \S\ref{sect:B-algebra} for the
definition).
\end{theorem}

The notion of the $B$-algebra of a graded associative algebra
is introduced by Braden et. al. in \cite{BLPW},
which is the quantization of the fixed point scheme
of a scheme with a torus action.
It plays an essential
role in the Hikita-Nakajima conjecture (see \cite{KMP,KTW,KS,Weekes}).

Let us also say some words about the torus actions.
In the above two theorems, the $A_n$ singularities
are toric varieties,
and hence we naturally consider the
$(\mathbb C^\times)^2$-equivariant quantum cohomology for them.
We also expect that the isomorphism in Theorem \ref{maintheorem} in the $A_n$ case should
be identical to the one proved by Weekes \cite{Weekes} by a different method.
For singularities of DE type,  there is only a natural $\mathbb C^\times$-action
on them, and we can only consider
their $\mathbb C^\times$-equivariant quantum cohomology, which
has been studied by Bryan and Gholampour in \cite{BG}.

On the other side, Joseph gave in \cite{Jo}
the quantizations of the minimal nilpotent orbit closures
in simple Lie algebras.
They are the quotients of the corresponding universal enveloping
algebras by some two-sided ideals, which
are nowadays called the Joseph ideals.
Later, Garfinkle in her thesis \cite{Ga} constructed
explicitly the Josephs ideals.
Interestingly enough, the Joseph ideals in the type A case
are not unique, but are parameterized by the complex numbers $\mathbb C$.
Thus in the type A case, if we view the number that parameterizes
the Joseph ideals as a formal variable, then the quantizations
of the minimal orbits in this case are over the ring of polynomials of two variables,
which exactly matches the base ring of the $(\mathbb C^\times)^2$-equivariant
cohomology of the dual side. For the other types of Lie algebras,
the Joseph ideals are uniquely determined, and hence
the quantizations are over the polynomials of one variable.

If we take the usual $\mathbb C^\times$-action on an $A_n$ singularity
similar to that
on the DE singularities
and take a specific Joseph ideal in its symplectic dual side (see Remark \ref{rem:JosephidealintypeA}),
then all the isomorphisms in the above two theorems become
\begin{equation}\label{iso:uniformform}
\mathrm{QH}^\bullet_{\mathbb C^{\times}}
(\widetilde{{\mathbb C^2}/\Gamma})
\cong Q(\mathscr A[\overline{\mathcal O}_{min}])
\quad\mbox{and}\quad
\mathrm{H}^\bullet_{\mathbb C^{\times}}
(\widetilde{{\mathbb C^2}/\Gamma})
\cong B(\mathscr A[\overline{\mathcal O}_{min}])
\end{equation}
respectively.

Moreover,
Theorems \ref{maintheorem0} and \ref{maintheorem} can be generalized
to the BCFG type singularities as follows. Recall that the Lusztig-Spaltenstein
duality says
the subregular nilpotent orbit
in a Lie algebra of BCFG type
is dual to
the minimal {\it special} nilpotent orbit
in its Langlands dual.
A theorem of Brylinski and Kostant \cite{BK}
says that the minimal special nilpotent orbits
in these cases are covered
by the minimal nilpotent orbits of $D_{n+1}$,
$A_{2n-1}$, $E_6$ and $D_4$ respectively,
with the deck transformations $\mathbb Z_2$
or $\mathfrak S_3$. On the other hand,
Slodowy showed in \cite{S} that the intersections of
Slodowy slices to the subregular nilpotent orbit
with the nilpotent cone
in these Lie algebras, which are also called
the {\it simple singularities of BCFG type}, is isomorphic to those in $A_{2n-1}$, $D_{n+1}$, $E_6$
and $D_4$, together with some extra symmetry also
given by either $\mathbb Z_2$ or $\mathfrak S_3$.
For these types of singularities,
Bryan and Gholampour constructed a version
of equivariant quantum algebra,
denoted by $\mathrm{QH}_R^\bullet(-)$, according to the associated
root systems/Dynkin diagrams. They also showed these algebras admit a
Frobenius algebra structure (see \cite{BG}).

Considering the associated D-modules of these varieties, we obtain the
following result, which is a corollary of Theorem
\ref{maintheorem0}:

\begin{theorem}\label{maintheorem2}
Let $\mathcal B_n$,
$\mathcal C_n$, $\mathcal F_4$ and
$\mathcal G_2$ be the minimal resolutions of
singularities of $B_n$, $C_n$, $F_4$ and $G_2$ respectively,
and let $\tilde{\mathcal O}_{ms}[B_n]$, $\tilde{\mathcal O}_{ms}[C_n]$,
$\tilde{\mathcal O}_{ms}[F_4]$ and
$\tilde{\mathcal O}_{ms}[G_2]$ be the normalizations
of the closures of the minimal special nilpotent orbits
in Lie algebras of BCFG type respectively.
Then
$$
\begin{array}{ll}
\mathrm{QH}^\bullet_{R}
(\mathcal B_n)\cong
Q(\mathscr A[\tilde{\mathcal{O}}_{ms}(C_n)]),&
\mathrm{QH}^\bullet_{R}
(\mathcal C_n)\cong
Q(\mathscr A[\tilde{\mathcal{O}}_{ms}(B_n)]),\\[2mm]
\mathrm{QH}^\bullet_{R}
(\mathcal F_4)\cong
Q(\mathscr A[\tilde{\mathcal{O}}_{ms}(F_4)]),&
\mathrm{QH}^\bullet_{R}
(\mathcal G_2)\cong
Q(\mathscr A[\tilde{\mathcal{O}}_{ms}(G_2)])
\end{array}
$$
as D-modules over the corresponding base rings.
\end{theorem}

Again, letting the quantum parameter $q=0$, we get the following
(see \S\ref{sect:BCFG} for more details about the notions):

\begin{theorem}\label{maintheorem3}
Let $\mathcal B_n$,
$\mathcal C_n$, $\mathcal F_4$ and
$\mathcal G_2$ be the minimal resolutions of
singularities of $B_n$, $C_n$, $F_4$ and $G_2$ respectively,
and let $\tilde{\mathcal O}_{ms}[B_n]$, $\tilde{\mathcal O}_{ms}[C_n]$,
$\tilde{\mathcal O}_{ms}[F_4]$ and
$\tilde{\mathcal O}_{ms}[G_2]$ be the normalizations
of the closures of the minimal special nilpotent orbits
in Lie algebras of BCFG type respectively.
Then
$$
\begin{array}{ll}
\mathrm H^\bullet_{\mathbb Z_2\times\mathbb C^\times}
(\mathcal B_n)\cong
B(\mathscr A[\tilde{\mathcal{O}}_{ms}(C_n)]),&
\mathrm H^\bullet_{\mathbb Z_2\times\mathbb C^\times}
(\mathcal C_n)\cong
B(\mathscr A[\tilde{\mathcal{O}}_{ms}(B_n)]),\\[2mm]
\mathrm H^\bullet_{\mathbb Z_2\times\mathbb C^\times}
(\mathcal F_4)\cong
B(\mathscr A[\tilde{\mathcal{O}}_{ms}(F_4)]),&
\mathrm H^\bullet_{\mathfrak S_3\times\mathbb C^\times}
(\mathcal G_2)\cong
B(\mathscr A[\tilde{\mathcal{O}}_{ms}(G_2)])
\end{array}
$$
as algebras over $\mathbb C[\hbar]$.
\end{theorem}

The rest of this paper is devoted to the proofs of the above
two theorems. It is organized as follows.
In \S\ref{sect:cohomologyofADE} we first recall some
basic facts on Kleinian singularities, and
then compute the equivariant quantum cohomology of
the minimal resolutions of these singularities.
In \S\ref{sect:quantization}
we go over the quantizations of the minimal nilpotent orbit closures
in Lie algebras of ADE type, which is due to Joseph \cite{Jo} and Garfinkle \cite{Ga}.
After that, in \S\ref{sect:Q(A)} we study with some details
the corresponding $B$-algebra of these quantizations.
In \S\ref{sect:thequantumHikitaconj}, we first recall Kamnitzer-McBreen-Proudfoot's
version of the quantum Hikita conjecture, and then
prove Theorems \ref{maintheorem0} and \ref{maintheorem}.
In \S\ref{sect:BCFG}, we study the quantizations
of the minimal special orbits
and the equivariant cohomologies of the minimal resolutions of BCFG type
singularities, and prove Theorems \ref{maintheorem2} and \ref{maintheorem3}.

\begin{ack}
In the spring of 2021, Professor Yongbin Ruan gave a series of lectures
at Zhejiang University
on his project on the mirror symmetry of nilpotent orbits
of semi-simple Lie algebras.
This paper is also motivated
by our study of his lectures. We are extremely grateful to him as well as
IASM, Zhejiang University
for inviting us to attend the lectures and for offering excellent working conditions.
We also thank Xiaowen Hu, Huazhong Ke and Yaoxiong Wen
for some very valuable discussions.
This work is supported by NSFC Nos. 11890663, 12271377 and 12261131498.
\end{ack}

\section{Equivariant quantum cohomology of ADE resolutions}\label{sect:cohomologyofADE}

In this section, we study the equivariant quantum
cohomology of the minimal resolutions of Kleinian singularities.
In \S\ref{subsect:Kleiniansing}
we briefly recall the definition of ADE singularities.
In \S\ref{subsect:Cstarequivcoho} we go over Bryan and Gholampour's result
in \cite{BG}
on the $\mathbb C^\times$-equivariant quantum cohomology of resolution of ADE singularities, and then \S\ref{subsect:typeA}
we further study the $(\mathbb C^\times)^2$-equivariant
quantum cohomology of the minimal resolution of $A_n$ singularities.
For $A_n$ singularities, we shall use both of them in later sections.

\subsection{Kleinian singularities}\label{subsect:Kleiniansing}

Let $\Gamma$ be a finite subgroup of
$\mathrm{SL}_2(\mathbb C)$.
It naturally acts on $\mathbb C^2$ via the canonical
action of $\mathrm{SL}_2(\mathbb C)$.
The singularity
$\mathbb C^2/\Gamma$ is called a Kleinian singularity,
and has been widely studied. The following table summarizes
the classification of Kleinian singularities:

\begin{center}
\begin{tabular}{p{2cm}p{4cm}p{2cm}p{5cm}}
\hline
Type & $\Gamma$ &  $|\Gamma|$ &Defining equation\\
\hline\hline
$A_n$& Cyclic Group $\mathbb{Z}_{n+1}$ & $n+1$ & $x^{n+1}-yz=0$\\
$D_n$& Binary Dihedral & $4(n-2)$&$x(y^2+x^{n-2})+z^2=0$\\
$E_6$& Binary Tetrahedral & $24$ &$x^4+y^3+z^2=0$\\
$E_7$& Binary Octahedral & $48$ &$x^3+xy^3+z^2=0$\\
$E_8$& Binary Icosahedral & $120$ &$x^5+y^3+z^2=0$\\
\hline
\end{tabular}
\end{center}

The singularity $\mathbb C^2/\Gamma$ has a unique minimal
resolution, denoted by $\widetilde{\mathbb C^2/\Gamma}$,
whose exceptional fiber is given by a tree of $\mathbb{CP}^1$'s.
The corresponding tree, whose vertices are the $\mathbb{CP}^1$'s
and whose edges between two given vertices are identified with the intersection
points of the corresponding $\mathbb{CP}^1$'s.
It turns out that the trees such constructed are exactly the Dynkin diagrams
of the Lie algebra of the same type.

There is another direct relationship between the Kleinian singularities
and the Lie algebras; namely, the Kleinian singularities
are exactly the Slodowy slices to the subregular nilpotent orbits
in the Lie algebra of the same type (see Slodowy \cite{S}
for more details).

Let $\mathfrak g$ be a Lie algebra. Recall that
the nilpotent cone of $\mathfrak g$, usually denoted
by $\mathcal N$, is the set
$$\mathcal N:=\left\{x\in\mathfrak g: (\mathrm{ad}_x)^n=0\,\,
\mbox{for some $n\in\mathbb N$}\right\}.$$

\begin{definition}[Slodowy slice \cite{S}]\label{def:Slodowyslice}
Let $x \in\mathfrak{g}$ be a nilpotent element, and extend it
to be an $\mathfrak{sl}_2(\mathbb{C})$
triple $\{ x, h, y\} \subseteq \mathfrak{g}$. The
Slodowy slice associated to $(\mathfrak{g}, x)$ is the
affine sub-variety $S = x + \ker[y, -] \subseteq \mathfrak{g}$.
\end{definition}

It is a transverse slice to the nilpotent orbit of the point $x$.

\begin{theorem}[Brieskorn \cite{Br} and Slodowy \cite{S}]
\label{Grothendieck-Brieskorn}
Let $\mathfrak{g}$ be simply-laced, $\mathcal{N} \subseteq \mathfrak{g}$
denote the nilpotent cone, and $S_x$ be a Slodowy slice to a subregular nilpotent
element $x\in\mathcal{O}_{sub}$. The intersection $S_x \cap \mathcal{N}$
is a Kleinian surface singularity with the same Dynkin diagram as
$\mathfrak{g}$. Moreover, the symplectic resolution
$\widetilde{S_x \cap \mathcal{N}}\rightarrow S_x \cap \mathcal{N}$
is the same as the minimal resolution of the Kleinian singularity
$\widetilde{\mathbb{C}^2/\Gamma}\rightarrow\mathbb{C}^2/\Gamma$.
\end{theorem}

\subsection{The $\mathbb C^\times$-equivariant
quantum cohomology}\label{subsect:Cstarequivcoho}

Suppose $T$ is a torus, $X$ is a variety with a $T$-action on it. The
$T$-equaivariant quantum cohomology ring $(\mathrm{QH}^{\bullet}_T(X), \star)$ of $X$ is a
deformation of $(\mathrm{H}^{\bullet}_T(X), \cup)$, which is induced by the genus 0
$T$-equivariant Gromov-Witten invariants of $X$.
More explicitly, the quantum product $\star$ is defined as follows (see e.g. \cite{KM}).
For any $e_\alpha, e_{\alpha'}, e_\gamma\in \mathrm{H}^{\bullet}_T(X)$
\begin{equation}\label{qproduct}
  (e_\alpha\star e_{\alpha'}, e_\gamma)= (e_\alpha\cup e_{\alpha'}, e_\gamma)+
  \sum_{\beta\in H_2(X, \Z)-\{0\}}\< e_\alpha,  e_{\alpha'}, e_\gamma\>^{X,T}_{0, \beta}\cdot q^{\beta}.
\end{equation}
Here $(-,-)$ is the Poincare pairing, $q$ is the quantum parameter and
$\< e_\alpha,  e_{\alpha'}, e_\gamma\>^{X,T}_{0, \beta}$ is some genus 0 $T$-equivariant Gromov-Witten invariant of $X$.

Now let $\widetilde{\mathbb C^2/\Gamma}$ be the minimal resolution
of an ADE singularity. Observe that the scalar $\mathbb C^\times$-action
on $\mathbb C^2$ commutes with the action of $\Gamma$,
and thus $\mathbb C^\times$ acts on $\mathbb C^2/\Gamma$.
It lifts to an action on $\widetilde{\mathbb C^2/\Gamma}$.
Let $\{E_1,E_2,\cdots, E_n\}$ be the set of irreducible components
in the exceptional fiber in
$\widetilde{\mathbb C^2/\Gamma}$, which gives a basis
of $\mathrm{H}_2(\widetilde{\mathbb C^2/\Gamma},\mathbb Z)$.
It is direct to see that they are invariant under the $\mathbb C^\times$-action,
and hence lifts to a basis $e_1, e_2, \cdots, e_n$ of the $\mathbb C^\times$-equivariant cohomology.
The intersection matrix $E_i\cap E_j$ defines a perfect pairing on
$\mathrm{H}_2(\widetilde{\mathbb C^2/\Gamma},\mathbb Z)$,
which coincides with the $\mathbb C^\times$-equivariant Poincare pairing $(e_i, e_j)$.

Let $\Delta$ be the root system associated
to the Dynkin diagram given by this pairing. Following \cite{BG}, we can identify both $E_1, \cdots, E_n$ and $e_1,\cdots,e_n$
with the simple roots $\alpha_1,\cdots,\alpha_n$ of ADE Lie algebra,
and the intersection matrix with minus of the Cartan matrix
$$E_i\cap E_j=(e_i, e_j)=-\langle \alpha_i,\alpha_j\rangle.$$
where  $\langle \alpha,\alpha'\rangle:=K(h_\alpha,h_{\alpha'})$ for
roots $\alpha,\alpha'\in\Delta$ and corresponding Cartan elements
$h_\alpha,h_{\alpha'}$ in $\mathfrak h$, $K(-,-)$ is the Killing form.
Bryan and Gholampour computed the $\mathbb C^\times$-equivariant
quantum cohomology ring
of $\widetilde{\mathbb C^2/\Gamma}$, which is given as follows.

\begin{theorem}[{\cite[Theorem 1]{BG}}]\label{thm:BryanGholam2}
For $e_\alpha$, $e_{\alpha'} \in \mathrm{QH}^\bullet(\widetilde{\mathbb C^2/\Gamma})$, the quantum product is given by
\begin{eqnarray}\label{eq: quantum coh DE}
e_\alpha\star e_{\alpha'}=-t^2|\Gamma|\langle \alpha ,\alpha'\rangle
+\sum_{\gamma\in\Delta^+}t\langle\alpha,\gamma\rangle\langle\alpha',
\gamma\rangle\dfrac{1+q^\gamma}{1-q^\gamma}  e_\gamma,
\end{eqnarray}
where $\Gamma$ is the subgroup of $\mathrm{SL}_2(\mathbb{C})$,
$e_\alpha=c_1e_1+\cdots+c_n e_n$, if $\alpha=c_1\alpha_1+\cdots c_n\alpha_n$, $c_1,\cdots, c_n\in\mathbb{N}$.
\end{theorem}

By the root data of Lie algebras of ADE type (see, for example,
Bourbaki \cite[PLATE I-VII]{Bo}), we may explicitly write down the cup product
in all the cases.

\subsection{The $(\mathbb C^\times)^2$-equivariant quantum cohomology of $A_n$ resolutions}\label{subsect:typeA}

In this subsection
we calculate the $(\mathbb C^\times)^2$-equivariant quantum cohomology of
the minimal resolution of the $A_n$ singularity.

\subsubsection{Equivariant cohomology}

We first calculate the equivariant cohomology.
The main reference we use here is \cite[Chapter 8]{AF}.

Let $\Gamma=\mathbb{Z}_{n+1}$ with
generator $\xi$. The finite group $\Gamma$ acts on $\mathbb C^2$ as
$$\xi \cdot (z_1, z_2)=\left(e^{\frac{2\pi i}{n+1}}z_1, e^{-\frac{2\pi i}{n+1}}z_2\right).$$
The $A_n$ singularity is given by $\mathbb C^2/\Gamma$.

Let $\mathcal{A}_n$ be the minimal resolution of $A_n$,
which is also called the
Hirzebruch-Jung resolution. By \cite{Ne}, $\mathcal{A}_n$ is a toric variety,
corresponding to a 2-dimensional fan $\Sigma$.
$\Sigma$ contains $n+1$ cones generated by the following $n+2$ rays
$\{\rho_i=\mathbb{C}v_i|0\leq i \leq n+1\}$ in $\R^2$, where $v_i$ is as follows:
\begin{equation}\label{fan}
	v_0=(1, 0), v_1=(0, 1), v_2=(-1,2), \cdots, v_{n+1}=(-n, n+1).
\end{equation}

Denote by $\langle\ ,\ \rangle_{\R^2}$ the canonical paring of $\R^2$.
Let $M\cong \Z^2$ be the character group of $(\mathbb{C}^\times)^2$.
Suppose $u_1=(1, 0)$ and $u_2=(0, 1)$ are basis of $M$.
By \cite{AF}, $u_1$ and $u_2$ can be treated as the equivariant
parameters of the torus action, which  corresponds to the regular
$(\mathbb{C}^\times)^2$-embedding in $\mathcal{A}_n$ given by the fan \eqref{fan}.

Set $\Lambda:=\mathrm H^{\bullet}_{(\mathbb{C}^\times)^2}(pt)=\mathbb{C}[u_1, u_2]$.
Let $X_0, X_1, \cdots, X_{n+1}$ be formal
variables, one for each ray of $\Sigma$.
The following $R_\Sigma$ is called the \textit{Stanley-Reisner ring}:
\begin{equation}\label{SR ring}
	R_\Sigma:=\Lambda[X_0, X_1, \cdots, X_{n+1}]/(I_\Sigma+J_\Sigma),
\end{equation}
where
\begin{itemize}
\item[$-$] the ideal $I_\Sigma$ is generated by all monomials $X_{i}\cdot X_{j}$
such that the corresponding rays $\rho_{i}, \rho_{j}$ do not span a cone,
that is, $|i-j|\geq 2$;

\item[$-$] the ideal $J_\Sigma$ is generated by the following two elements:
\begin{align}\label{SR ring-J}
	u_1-\sum \langle u_1, v_i \rangle_{\R^2} X_i, \quad u_2-\sum \langle u_2, v_i \rangle_{\R^2} X_i.
\end{align}
\end{itemize}
Denote the equivariant cohomology of
$\mathcal{A}_n$ by $\mathrm H^{\bullet}_{(\mathbb{C}^\times)^2}(\mathcal{A}_n)$.
There is a ring structure on it induced by the cup product $\cup$.
Define a homomorphism
\begin{equation}
	R_\Sigma\rightarrow \mathrm H^{\bullet}_{(\mathbb{C}^\times)^2}(\mathcal{A}_n)
\end{equation}
by $X_i\mapsto e_i$, where $e_i$ is the equivariant class of the
$(\mathbb{C}^\times)^2$-invariant divisor corresponding to the ray $\rho_i$.
The following proposition is proved in \cite[Theorem 8]{BCP} (see also \cite[Theorem 3.1]{AF}).

\begin{proposition}\label{An ring iso}
The homomorphism
$(R_\Sigma, \cdot)\rightarrow
\mathrm (\mathrm H^{\bullet}_{(\mathbb{C}^\times)^2}(\mathcal{A}_n), \cup)$
is a ring isomorphism.
\end{proposition}

According to \cite{Cox}, $\mathcal{A}_n$ can be viewed as a GIT quotient, namely,
\begin{equation*}
  \mathcal{A}_n\cong(\mathbb{C}^{n+2}-Z(\Sigma))/(\mathbb{C}^\times)^n,
\end{equation*}
where $Z(\Sigma)=\bigcap_{1\leq i \leq n+1}\{z_0\cdots \hat{z_i}\cdots z_{n+1}=0\}$, and
the $(\mathbb{C}^\times)^{n}$-action on $\mathbb{C}^{n+2}$ is as follows:
for any $(\lambda_1, \lambda_2, \cdots, \lambda_{n})\in (\mathbb{C}^\times)^{n}$,
$$(\lambda_1,  \cdots, \lambda_{n})\cdot (z_0, \cdots, z_{n+1})=(\lambda_1 z_0, \lambda_1^{-2}\lambda_2 z_1,
\lambda_1\lambda_2^{-2}\lambda_3z_2,  \cdots, \lambda_{n-1} \lambda_n^{-2} z_{n-1}, \lambda_n z_{n+1}).$$
Here we use the homogeneous coordinate $[z_0: z_1: \cdots : z_{n+1}]$ to
parametrize the $(\mathbb{C}^\times)^{n}$-orbit of $(z_0, \cdots, z_{n+1})$. Then the projection $\mathcal{A}_n\rightarrow \mathbb{C}^2/\Gamma$ can be written as
\begin{equation}\label{proj}
  [z_0: z_1: \cdots : z_{n+1}]\mapsto \left[\big(\prod_{i=0}^{n+1}z_i^{n+1-i}\big)^{1\over n+1}, 
  \big(\prod_{i=0}^{n+1}z_i^{i}\big)^{1\over n+1}\right].
\end{equation}
There is another natural $(\mathbb{C}^\times)^2$-action on $\mathcal{A}_n$: for $\eta\in \mathbb{C}^\times$,
\begin{equation}\label{C2action}
  \eta \cdot [z_0: z_1: \cdots z_n: z_{n+1}]=[\eta^{t_2} z_0: z_1: \cdots z_n: \eta^{t_1}z_{n+1}].
\end{equation}
This torus action is used to calculate equivariant Gromov-Witten invariants of
$\mathcal{A}_n$ in \cite{Mau}. We treat $t_1$, $t_2$ as equivariant parameters
of the above torus action, then $(u_1, u_2)$ is determined by $(t_1, t_2)$ as
following (see e.g \cite[Section 4.4]{Liu}):
\begin{equation}\label{coordinate change}
\begin{split}
u_1&=\<u_1, v_{n+1}\>_{\R^2}t_1+\<u_1, v_{0}\>_{\R^2}t_2=-nt_1+t_2,\\
 u_2&=\<u_2, v_{n+1}\>_{\R^2}t_1+\<u_2, v_{0}\>_{\R^2}t_2=(n+1)t_1.
\end{split}
\end{equation}
By Proposition \ref{An ring iso}, $\mathrm H^{\bullet}_{(\mathbb{C}^\times)^2}(\mathcal{A}_n)$ is a free
$\Lambda$-module generated by $1, e_1, e_2, \cdots, e_n$.

Recall that, for complex semisimple Lie algebra, there is an isomorphism
$\mathfrak h \rightarrow \mathfrak h^*$. In this section, for convenience to further discussion,
we identify $e_i$ with the $i$-th simple root $\alpha_i$ of $\mathfrak{sl}(n+1, \mathbb{C})$
and the cartan element $h_{\alpha_i}$, and identify the fundamental weights $\omega_i\in\mathfrak h^*$
with its linear dual $\omega_i^*\in\mathfrak h$. 	
Then the Cartan subalgebra $\mathfrak{h}={\rm Span}_\mathbb{C}\{e_1, e_2, \cdots, e_n\}$.
It is well know that $\mathfrak{h}$ can be embedded into $\mathbb{C}^{n+1}$.
And $e_i= \varepsilon_i-\varepsilon_{i+1}$, where $\{\varepsilon_j\}_{j=1, \cdots, n+1}$ is the canonical basis of
$\mathbb{C}^{n+1}$. Denote the euclidean pairing on $\mathbb{C}^{n+1}$ by $\<- ,- \>$,
of which restriction on $\mathfrak{h}$ is the Killing form.
\begin{equation}\label{euclidean}
  \<\varepsilon_i, \varepsilon_j\>=\delta_{ij}
\end{equation}

\begin{definition}\label{def: bilinear c}
Define a bilinear map
$(-,-)_c: \mathfrak{h}\otimes \mathfrak{h} \rightarrow \mathfrak{h}$ as follows: for
$\alpha=\sum_{i=1}^{n+1}x_i\varepsilon_i$, $\alpha'=\sum_{i=1}^{n+1}y_i\varepsilon_i$,
\begin{equation}\label{classical}
     (\alpha, \alpha')_c:=\sum_{i=1}^{n+1}x_iy_i(\omega_{i-1}-\omega_i),
\end{equation}
where $\omega_i$, $1\le i\le n$, is the $i$-th fundamental
weight satisfying $\<\omega_i , e_j\>=\delta_{ij}$,
and $\omega_0=\omega_{n+1}=0$.
\end{definition}

From the isomorphism of linear space: $\mathfrak h\cong \mathrm H^2_{(\mathbb{C}^\times)^2}(\mathcal{A}_n)$,
$(\alpha,\alpha')_c$ is in $H^2_{(\mathbb{C}^\times)^2}(\mathcal{A}_n)$, we rephrase the ring structure of
$\mathrm H^{\bullet}_{(\mathbb{C}^\times)^2}(\mathcal{A}_n)$ as follows.

\begin{theorem}\label{equ coh An}
$(\mathrm H^{\bullet}_{(\mathbb{C}^\times)^2}(\mathcal{A}_n), \cup)
\cong (\Lambda\otimes \mathrm{Sym}^\bullet\mathfrak{h})/
\widetilde{I}$, where $\widetilde{I}$ is generated by the following relation:
for any $e_\alpha, e_{\alpha'}\in \mathfrak{h}$,
\begin{align}\label{eqcup}
e_{\alpha} \cup e_{\alpha'}=-(n+1)\langle\alpha,\alpha'\rangle
t_1t_2-(n+1)\dfrac{t_1-t_2}{2}(\alpha,\alpha')_c+
\sum_{\gamma\in\Delta^+}\dfrac{t_1+t_2}{2}\langle\alpha,
\gamma\rangle\langle\alpha',\gamma\rangle e_\gamma,
\end{align}
where $\Delta^+=\{\varepsilon_i-\varepsilon_j|1\leq i< j \leq n+1\}$ is the set of positive roots.
\end{theorem}

\begin{proof}
Plugging \eqref{coordinate change} into \eqref{SR ring-J},
we have in $\mathrm H^{\bullet}_{(\mathbb{C}^\times)^2}(\mathcal{A}_n)$,
\begin{equation}\label{e0en+1}
	e_{n+1}=t_1-\frac{1}{n+1}\sum_{i=1}^n(ie_i), \quad e_0=t_2-\frac{1}{n+1}\sum_{i=1}^n(ie_{n+1-i}).
\end{equation}
Notice that the two handsides of \eqref{eqcup} are both commutative
$\Lambda$-bilinear map on $\mathfrak{h}$. It suffice to verify \eqref{eqcup}
on a basis of $\mathrm{Sym}^2 \mathfrak{h}$. And we choose a basis of
$\mathrm{Sym}^2 \mathfrak{h}$ which consists of the following four types of vectors:
\begin{enumerate}
\item[(a)] $e_ie_j, \quad 1\leq i<j\leq n, j-i\geq 2$;
\item[(b)] $(t_2-e_0)e_k,\quad 2\leq k\leq n$;
\item[(c)] $ (t_1-e_{n+1})e_l, \quad 1\leq l\leq n-1$;
\item[(d)] $(t_2-e_0)(t_1-e_{n+1})$.
\end{enumerate}
Now we check \eqref{eqcup} holds on these vectors. For type (a),
when $e_\alpha=e_i=\varepsilon_i-\varepsilon_{i+1}$,
$e_{\alpha'}=e_j=\varepsilon_j-\varepsilon_{j+1}$ in
\eqref{eqcup}, LHS vanishes by Proposition
\ref{An ring iso} and \eqref{SR ring}.
Notice that $\langle \alpha,\alpha'\rangle$ and $(\alpha,\alpha')_c$
vanish by $j-i>2$, so the first two terms in RHS vanish.
By \eqref{euclidean}, $\langle \alpha,\gamma\rangle\langle\alpha',
\gamma\rangle$ does not vanish only when $\gamma
=\varepsilon_i-\varepsilon_j$, $\varepsilon_i-\varepsilon_{j+1}$,
$\varepsilon_{i+1}-\varepsilon_j$ or $\varepsilon_{i+1}-\varepsilon_{j+1}$,
one can easily check the last term in RHS vanishes.
For type (d), set $\alpha=t_2-e_0=\frac{1}{n+1}(n\varepsilon_1-
\varepsilon_2-\cdots-\varepsilon_{n+1})=\omega_1$, $\alpha'=t_1-e_{n+1}
=\frac{1}{n+1}(\varepsilon_1+\cdots+\varepsilon_{n}-n\varepsilon_{n+1})=
\omega_{n}$,
then by Proposition \ref{An ring iso} and \eqref{SR ring},
$$\mathrm{LHS} =t_1t_2-t_1e_0-t_2e_{n+1}=-t_1t_2+t_1\omega_1+t_2\omega_2.$$
By \eqref{euclidean} and \eqref{classical}, we have
\begin{equation*}
\mbox{RHS}=-t_1t_2+\frac{t_1-t_2}{2}(\omega_1-\omega_n)+
\frac{t_1+t_2}{2}(\varepsilon_1-\varepsilon_{n+1})=-t_1t_2+t_1\omega_1+t_2\omega_2.
\end{equation*}
The verification for types (b) and (c) is left to readers.
\end{proof}

\subsubsection{Equivariant quantum cohomology}

In this subsection, we calculate the equivariant
quantum cohomology of $\mathcal{A}_n$.
The exceptional locus of $\mathcal{A}_n$ consists of a chain of $n$ rational curves
$E_1, \cdots, E_n$ with intersection matrix given by the minus
Cartan matrix for the $A_n$ root lattice.
Notice that $E_i$ is $(\mathbb{C}^\times)^2$-invariant, and $e_i$
is its equivariant lift in
$\mathrm{QH}^{2}_{(\mathbb{C}^\times)^2}(\mathcal{A}_n)$.
In \cite{Mau}, Maulik calculated all genus
$(\mathbb{C}^\times)^2$-equivariant Gromov-Witten invariant of
$\mathcal{A}_n$.
The following lemma is a special case of \cite[Theorem 1.1]{Mau}.

\begin{lemma}
The Gromov-Witten invariant does not vanish only when
$\beta=d(E_i+E_{i+1}+\cdots+E_{j-1})$, for some $1\leq i<j\leq n+1$,
$d\geq 0$. Furthermore, if  $\beta=d(E_i+E_{i+1}+\cdots+E_{j-1})$,
\begin{equation}\label{GWAn}
    \< e_\alpha,  e_{\alpha'}, e_\gamma\>^{\mathcal{A}_n,
    (\mathbb{C}^\times)^2}_{0, \beta}=(t_1+t_2)(e_\alpha, e_\eta)(e_{\alpha'}, e_\eta)(e_\gamma, e_\eta),
  \end{equation}
  where $e_\eta=e_i+e_{i+1}+\cdots+e_{j-1}$, $(-,-)$ is the
  Poincar\'e pairing satisfying $(e_k, e_l)=-\<\alpha_k, \alpha_l\>$.
\end{lemma}

\begin{remark}
In \cite[Theorem 1.1]{Mau},
Maulik calculated the reduced Gromov-Witten invariants of
$\mathcal{A}_n$. For $\beta\neq 0$, the reduced Gromov-Witten invariants coincides with
$(\mathbb{C}^\times)^2$-equivariant Gromov-Witten invariant up to a factor
$(t_1+t_2)$; see \cite[Section 2.2]{Mau}.
\end{remark}

Following \cite{BG}, we identify both $e_i$ and $E_i$ with the $i$-th
simple root $\alpha_i$ of $A_n$, then the Gromov-Witten
invariant $\< e_\alpha,  e_{\alpha'}, e_\gamma\>^{\mathcal{A}_n,
(\mathbb{C}^\times)^2}_{0, \beta}$ does not vanish
only when $\beta$ is a multiple of some positive root.

\begin{theorem}\label{thm:quantcohofA_n}
For any $e_\alpha, e_{\alpha'}\in
\mathrm{QH}^{2}_{(\mathbb{C}^\times)^2}(\mathcal{A}_n)$, we have
  \begin{equation}\label{qcohAn}
    e_\alpha\star e_{\alpha'}=-(n+1)\langle\alpha,
    \alpha'\rangle t_1t_2-(n+1)\dfrac{t_1-t_2}{2}(\alpha,\alpha')_c
    +\sum_{\gamma\in\Delta^+}\frac{(t_1+t_2)}{2}\langle\alpha,
    \gamma\rangle\langle\alpha',\gamma\rangle \frac{1+q^\gamma}{1-q^\gamma} e_\gamma.
  \end{equation}
\end{theorem}

\begin{proof}
By \eqref{qproduct} and \eqref{GWAn}, we have
  \begin{align}\label{eq14}
    e_\alpha\star e_\alpha'&=e_\alpha\cup e_\alpha'+
     \sum_{d>0}\sum_{\gamma\in\Delta^+}(t_1+t_2)\langle\alpha,\gamma\rangle\langle\alpha',
     \gamma\rangle e_\gamma \cdot q^{d\gamma}\nonumber\\
     &=e_\alpha\cup e_\alpha'+\sum_{\gamma\in\Delta^+}(t_1+t_2)
     \langle\alpha,\gamma\rangle\langle\alpha',\gamma\rangle \frac{q^\gamma}{1-q^\gamma}e_\gamma.
  \end{align}
Now combining \eqref{eqcup} with \eqref{eq14}, we obtain \eqref{qcohAn}.
\end{proof}

\begin{remark}
  In Theorem \ref{thm:BryanGholam2}, we consider the conical action on
  $\widetilde{\mathbb C^2/\Gamma}$ induced by scalar $\mathbb C^\times$-action
on $\mathbb C^2$. And in Theorem \ref{thm:quantcohofA_n},
$t_1, t_2$ corresponds to the following $(\mathbb C^\times)^2$-action on
$\mathbb C^2$ (see \eqref{proj} and \eqref{C2action})
$$\eta\cdot (z_1, z_2)=(\eta^{t_2}z_1, \eta^{t_1}z_2),\quad \eta\in \mathbb C^\times.$$
So the transformation between two equivariant parameter is $t_1=t, t_2=t$.
It is straightforward
to check that in \eqref{qcohAn} if we set $t_1=t_2$, then the quantum products
for $A_n$ singularities are identical to the ones given by \eqref{eq: quantum coh DE} .
\end{remark}

\begin{remark}\label{rem:redofequivquantcoh}
In both Theorems \ref{thm:BryanGholam2} and
\ref{thm:quantcohofA_n},
if we let $q=0$, then the equivariant quantum
cohomology reduces to the usual equivariant cohomology.
\end{remark}

\section{Quantization of the minimal nilpotent orbits}\label{sect:quantization}

In this section, we study the quantization of the minimal nilpotent
orbits of Lie algebras of ADE type.
In \cite{Jo}, Joseph studied
the quantizations of these orbits,
which are given by the quotients of the universal enveloping
algebras by the two-sided ideals called the {\it Joseph ideals}.
In her Ph.D. thesis \cite{Ga}, Garfinkle gave a new construction of the Joseph ideals,
by explicitly writing down the generators and relations.

In \S\ref{subsect:coordofminimal} we briefly go over Shlykov's result on
the minimal nilpotent orbits. In \S\ref{subsect:Joseph} we recall Joseph's result on the quantization
of the minimal nilpotent orbits and then in \S\ref{subsect:Garfinkle} we go over Garfinkle's construction
of Joseph's ideals.
In \S\ref{sect:B-algebra} we briefly
recall the $B$-algebra of the quantization
of the minimal nilpotent orbits.

\subsection{The coordinate ring of minimal orbits}\label{subsect:coordofminimal}

In this subsection, we assume $\mathfrak g$ is a complex semisimple Lie algebra,
and $\mathcal O_{min}$ is the minimal nilpotent orbit of $\mathfrak g$.
Let us first recall the following.

\begin{proposition}[{c.f. \cite[\S 8.3]{Jan}}]
Let $\mathfrak g$ be a complex semisimple Lie algebra
and $\mathcal{O}$ be a nilpotent orbit of $\mathfrak g$. Then
\[
\mathbb{C}[\mathcal{O}]=\mathbb{C}[\overline{\mathcal{O}}]
\]
if and only if $\overline{\mathcal{O}}$ is normal.
\end{proposition}
In particular, $\overline{\mathcal{O}}_{min}$ is normal with isolated singularity (see \cite{VP}), and hence
\[
\mathbb{C}[\mathcal{O}_{min}]=\mathbb{C}[\overline{\mathcal{O}}_{min}].
\]
Due to this proposition, in what follows we shall not distinguish
$\mathbb{C}[\mathcal{O}_{min}]$ and $\mathbb{C}[\overline{\mathcal{O}}_{min}]$.
The following result is proved by Shlykov in \cite{Sh}.

\begin{theorem}[{\cite[Theorem 2.2]{Sh}}] \label{Sh2}
Let $I$ be the defining ideal of $\overline{\mathcal{O}}_{min}$ in $\mathrm{Sym}(\mathfrak{g})$, i.e.,
	\[
	I:=\{\mu\in \mathrm{Sym}(\mathfrak{g})|\mu(\overline{\mathcal{O}}_{min})=0\},
	\]
then its image of the projection
$$
f: \mathrm{Sym}(\mathfrak{g})\rightarrow  \mathrm{Sym}(\mathfrak{h})
$$
induced by the inclusion $\mathfrak{h}^*\hookrightarrow \mathfrak{g}^*$
is given by  $\mathrm{Sym}^{\geq 2}\mathfrak{h}$. 	
\end{theorem}

Let $G$ be the corresponding Lie group of $\mathfrak g$. Then the adjoint action
$Ad: G\times\overline{\mathcal{O}}_{min}\to \overline{\mathcal{O}}_{min}
$
is Hamiltonian. Let $T$ be the maximal torus of $G$.
If we choose a generic action of $\mathbb{C}^\times$ on
$\overline{\mathcal{O}}_{min}$ such that the fixed point schemes for it and
for $T$ are the same, that is
$\overline{\mathcal{O}}_{min}^{\mathbb{C}^\times}\cong\mathfrak{h}\cap\overline{\mathcal{O}}_{min}$
as a scheme, then
the main result of Shlykov \cite{Sh} says
\begin{equation*}
\mathbb{C}[\overline{\mathcal{O}}_{min}^{\ \mathbb{C}^\times}]
\cong\mathbb{C}[\mathfrak{h}\cap
\overline{\mathcal{O}}_{min}]=\mathrm{Sym}(\mathfrak{h})/f(I)
=\mathrm{Sym}(\mathfrak{h})/\mathrm{Sym}^{\geq 2}\mathfrak{h}
\end{equation*}
is isomorphic to $\mathrm H^\bullet(\widetilde{\mathbb C^2/\Gamma})$,
where $\mathbb C^2/\Gamma$ is the Kleinian singularity with
the same type of $\mathfrak g$.

\subsection{Quantization of the minimal nilpotent orbits}\label{subsect:Joseph}

We now study the quantization of the minimal nilpotent orbits in Lie algebra of ADE type.	
We start with some basic concepts on the quantization of Poisson algebras;
see, for example, Losev \cite{Lo} for more details.

\begin{definition}[Filtered and graded quantizations]
Suppose $A$ is a commutative $\mathbb Z_{\ge 0}$-graded $k$-algebra,
equipped with a Poisson
bracket whose degree is $-1$, where $k$ is a field of characteristic zero.
\begin{enumerate}
\item[$(1)$]
A {\it filtered quantization} of $A$ is a filtered $k$-algebra
$\mathcal A=\bigcup_{i\ge 0}\mathcal A_{i}$ such that
the associated graded algebra $\mathrm{gr}\,\mathcal A$
is isomorphic to $A$ as graded Poisson algebras.

\item[$(2)$]
A {\it graded quantization} of $A$ is a graded $k[\hbar]$-algebra
$A_\hbar$ ($\mathrm{deg}\,\hbar=1$)
which is free as a $k[\hbar]$-module, equipped with an isomorphism of $k$-algebras:
$f: A_\hbar/\hbar\cdot A_\hbar \to A$
such that for any
$a, b\in A_\hbar$, if we denote their images in
$A_\hbar/\hbar\cdot A_\hbar$ by $\overline a, \overline b$ respectively,
then
$$f\left(\overline{\frac{1}{\hbar}[a, b]}\right)=\{f(\overline a), f(\overline b)\}.$$
\end{enumerate}
\end{definition}

Let $A$ be an filtered associative algebra. Recall that the {\it Rees algebra} of $A$ is the
graded algebra $Rees(A):=\bigoplus_{i\in\mathbb{Z}}A_{\leq i}\cdot\hbar^i$,
equipped with the multiplication $(a\hbar^i)(b\hbar^j)=ab\hbar^{i+j}$ for $a,b\in A$.
Now, suppose $\mathcal A$ is a filtered quantization of $A$, then
the associated Rees algebra $Rees(\mathcal A)$ is a graded quantization of $A$.

\begin{example}\label{example: quantization}
The universal enveloping algebra $\mathcal{U}(\mathfrak{g})$ is the filtered quantization of
$\mathbb{C}[\mathfrak{g}^*]=\mathrm{Sym}(\mathfrak{g})$, and the Rees algebra of
$\mathcal{U}(\mathfrak{g})$,
$Rees(\mathcal{U}(\mathfrak{g})):=
\bigoplus_{i\in\mathbb{Z}}\mathcal{U}(\mathfrak{g})_{\leq i}\cdot\hbar^i$
is the graded quantization of  $\mathrm{Sym}(\mathfrak{g})$.
On the other hand, there is an isomorphism of
$\mathfrak{g}$-modules:
\begin{eqnarray*}
  \beta: \mathrm{Sym}(\mathfrak{g})&\rightarrow & \mathcal{U}(\mathfrak{g}),\\
  	x_1\cdots x_k &\mapsto & \dfrac{1}{k!}\sum_{\pi\in S_n} x_{\pi(1)}\cdots x_{\pi(k)},
 \end{eqnarray*}
which is called {\it symmetrization}.
\end{example}

Since the universal enveloping algebra $\mathcal{U}(\mathfrak{g})$ is
the quantization of the symmetric algebra $\mathrm{Sym}(\mathfrak{g})$,
we need to study the quantization of the ideal $I$ of $\mathrm{Sym}(\mathfrak{g})$.
Joseph in \cite{Jo} found a two-sided ideal of $\mathcal{U}(\mathfrak{g})$
which plays the role of the quantization of $I$.

\subsubsection{Joseph's quantization of the minimal nilpotent orbits}

Let us first recall the result of Joseph \cite{Jo},
which is stated as follows.

\begin{theorem}[Joseph \cite{Jo} and Garfinkle \cite{Ga}]
Let $\mathfrak{g}$ be a complex semisimple Lie algebra.
 \begin{enumerate}
 \item[$(1)$]
If $\mathfrak{g}$ is the type A Lie algebra, then there exists a family of
completely prime two-sided primitive ideals $J^z$, parametrized
by $z\in\mathbb{C}$, such that
$$
\mathrm{gr}J^z=I(\overline{\mathcal{O}}_{min}).
$$

\item[$(2)$]	
If $\mathfrak{g}$ is not of type A,
 then there exists a unique completely prime two-sided primitive ideal $J$ such that
 $$
\mathrm{gr}J=I(\overline{\mathcal{O}}_{min}).
 $$
 \end{enumerate}
 \end{theorem}

In the above theorem,
a two-sided ideal $J$ of $\mathcal{U}(\mathfrak{g})$ is called
{\it primitive} if it is
the kernel of an irreducible representation $(\pi, V)$ of $\mathcal{U}(\mathfrak{g})$, i.e.,
$J$ is the annihilator of $V$,
	\[
	J=Ann(V)=\{u\in\mathcal{U}(\mathfrak{g})|\pi(u)\cdot V=0\}.
	\]
An ideal $J$ of $\mathcal{U}(\mathfrak{g})$ is called
{\it completely prime} if for all $u,v\in\mathcal{U}(\mathfrak{g})$, $uv\in J$
implies $u\in J$ or $v\in J$.
In literature, the ideals $J^z$ and $J$ are usually
called the {\it Joseph ideals}.

In fact, in the original paper \cite{Jo}, Joseph proved that the Joseph ideals in type A Lie algebras
are not unique. It is Garfinkle who
 gave the explicit constructions of the Joseph ideals in Lie algebras of all types,
and in particular, formulated the Joseph ideals in type A Lie algebras in the form given in the above
theorem.

Since
\[
\mathrm{gr} \big(\mathcal{U}(\mathfrak{g})/J\big)=\mathrm{gr}\big(\mathcal{U}(\mathfrak{g})\big)/gr(J)
=\mathrm{Sym}(\mathfrak{g})/I(\overline{\mathcal{O}}_{min})=\mathbb{C}[\overline{\mathcal{O}}_{min}],
\]
we have that, for the symplectic singularity $\overline{\mathcal{O}}_{min}$, the algebra
$\mathcal{U}(\mathfrak{g})/J$ is its filtered quantization.

By the above theorem, $Rees(\mathcal{U}(\mathfrak{g})/J)$ is the graded quantization of
$\overline{\mathcal{O}}_{min}$, and we sometimes write it
as $\mathscr A[\overline{\mathcal O}_{min}]$; that is,
$\mathscr A[\overline{\mathcal O}_{min}]=Rees\big(\mathcal U(\mathfrak g)/J\big)$.

\subsection{Garfinkle's construction of the Joseph ideals}\label{subsect:Garfinkle}

Garfinkle in her thesis \cite{Ga} gave an explicit construction of the Joseph ideals.
In this subsection, we go over her results with some details.

 \begin{notation}\label{notation}
Let us fix some notations in representation theory of Lie algebras.

Let $\mathfrak{g}$ be a complex semisimple Lie algebra, $\mathfrak{h}$ be a
Cartan subalgebra of $\mathfrak{g}$, $\Delta$ be the set of roots of $\mathfrak{h}$ in
$\mathfrak{g}$ and $\Delta^+$ be a fixed choice of positive roots.
Let $\Pi\subset \Delta^+$ be the set of the simple roots of $\mathfrak g$ and $Q:=\mathbb{Z}\Pi$ is the root lattice of $\mathfrak g$.
The Lie algebra $\mathfrak{g}$ has the root space decomposition
$\mathfrak{g}=\oplus_{\alpha\in\Delta}\mathfrak{g}_\alpha$, and let
\[
\mathfrak{n}^+=\oplus_{\alpha\in\Delta^+}\mathfrak{g}_\alpha,
\mathfrak{n}^-=\oplus_{\alpha\in\Delta^+}\mathfrak{g}_{-\alpha},
\mathfrak{b}=\mathfrak{h}\oplus\mathfrak{n}^+.
\]
denote the associated subalgebras of $\mathfrak{g}$.

Let $(\pi, V)$ be a representation of $\mathfrak{g}$; for any weight
$\lambda\in\mathfrak{h}^*$, let
$V^\lambda=\{v\in V|\pi(h)(v)=\lambda(h) v\;\mbox{for any}\; h\in\mathfrak{h}\}$. Let
$V^{\mathfrak{n}^-}:= \{v\in V|\pi(x)v=0\;\mbox{for any}\; x\in\mathfrak{n}^-\}$.

For any $\alpha\in \Delta^+$, fix a root vector $X_{\alpha}$ in $\mathfrak{g}_{\alpha}$, and denote by $Y_\alpha\in \mathfrak{g}_{-\alpha}$ the dual basis
of $X_{\alpha}$ with respect to the Killing form $K(-,-)$.  Denote by $h_i$ the element in
$\mathfrak{h}$ corresponding to $\alpha_i\in \Pi$ such that $\alpha_i(H)=K(H,h_i)$
for all $H\in\mathfrak{h}$. By the construction of
the Chevalley basis,
$h_i=[X_{\alpha_i}, Y_{\alpha_i}]$. Denote by $h_i^{\vee}$ the dual element of $h_i$
via the Killing form, i.e., $K(h_i^{\vee}, h_j)=\delta_{ij}$.

Let $C:=\sum_{\alpha\in\Delta^+} (X_\alpha Y_\alpha+Y_\alpha X_\alpha)
+\sum_{i=1}^n h_ih_i^{\vee}$ be the Casimir element of $\mathcal{U}(\g)$,
$n=\mathrm{rank}(\mathfrak h)$.

Let $\alpha_1,\cdots,\alpha_n\in\Pi$ with the subscripts the same as
\cite[PLATE I-VII]{Bo}). Denote by $\theta$ the highest root in $\Delta$.

\end{notation}

\subsubsection{Joseph ideal for type A Lie algebras}\label{subsect:JosephintypeA}

In \cite{Ga}, Garfinkle gave the explicit construction
of the Joseph's ideals. Let us recall her results.

\begin{proposition}[{\cite[Proposition 3.2]{Ga} and
\cite[\S4.4]{BJ}}]\label{structure of J}

For type $A_n$ Lie algebras $\mathfrak{g}$, we have the following decomposition of
irreducible representations:
\begin{eqnarray*}
\mathrm{Sym}^2(\mathfrak{g})\cong
V(2\theta)\oplus V(\theta+\alpha_2+\cdots+\alpha_{n-1})\oplus V(\theta)\oplus V(0).
\end{eqnarray*}
The ideal $I(\overline{\mathcal{O}}_{min})$
is generated by the lowest weight
vectors in $V(\theta+\alpha_2+\cdots+\alpha_{n-1})$, $V(\theta)$ and $V(0)$, where $V(0)$ is spanned by the
Casimir element $C$ of $\mathcal U(\mathfrak g)$.
\end{proposition}

Garfinkle showed that the Joseph ideal $J$ in the type A case
is generated by elements corresponding
to the three types of lowest weight vectors in the above proposition.
We examine them one by one.	

First, for the subrepresentation $V(\theta+\alpha_2+\cdots+\alpha_{n-1})$,
we have the following:

\begin{lemma}[\cite{Ga} \S IV.3 Theorem 2 and \S 5]\label{Garfinkle lemma1}
Let $v_0$ be the lowest weight of the representation
$V(\theta+\alpha_2+\cdots+\alpha_{n-1})$. Then
$\beta(v_0)$ is an
element of Joseph ideal $J$ of $\mathcal{U}(\mathfrak{g})$.
\end{lemma}

Next, we consider the lowest weight vector in $V(\theta)$.
For convenience, we fix a special choice of root vectors $X_\alpha$ and $Y_\alpha$ via elements in $gl(n+1, \mathbb{C})$. More explicitly, we set $X_{i\cdots j}=E_{i,j+1}$
and $Y_{i\cdots j}=E_{j+1,i}$ the root vectors in $\mathfrak{g}$ corresponding to the root
$\alpha_i+\alpha_{i+1}+\cdots+\alpha_j\in \Delta^+$ and the root
$-(\alpha_i+\alpha_{i+1}+\cdots+\alpha_j)$ respectively, where $E_{ij}\in\mathrm{M}_{(n+1)\times(n+1)}(\mathbb C)$.

\begin{lemma}\label{generalterm1}
The lowest weight vector of the subrepresentation $V(\theta)$ in Proposition \ref{structure of J} is
\begin{eqnarray}\label{vtheta}
v=-(n+1)(Y_1Y_{2\cdots n}+Y_{12}Y_{3\cdots n}+\cdots+Y_{1\cdots n-1}Y_n)+
\sum_{k=1}^n Y_\theta(2k-1-n)h_k.
\end{eqnarray}
\end{lemma}

\begin{proof}
It is straightforward to verify $[Y_{\alpha_i},v]=0$ for all $\alpha_i\in\Pi$,
and thus $v$ is the lowest weight vector.
\end{proof}

We next find the generator of $J$ corresponding to \eqref{vtheta}.
Recall that a subalgebra $\mathfrak{p}\subseteq\mathfrak{g}$ such that
$\mathfrak{p}\supseteq \mathfrak{b}$ is called a parabolic subalgebra.
Let $\Pi'\subset \Pi$, we define a parabolic subalgebra as follows:
Let $\Delta_\mathfrak{l}=\{\gamma\in\Delta|\gamma=\sum_{\alpha\in\Pi'}n_\alpha\alpha,
n_\alpha\in\mathbb{Z}\}$, $\Delta_{\mathfrak{u}^+}=\{\alpha\in\Delta^+|\alpha\notin\Delta_l \}$.
Then, let
$\mathfrak{l}=\mathfrak{h}\oplus\oplus_{\alpha\in\Delta_l}\mathfrak{g}_\alpha$, $\mathfrak{u}
=\oplus_{\alpha\in\Delta_{\mathfrak{u}^+}}\mathfrak{g}_\alpha$. We call $\mathfrak{p}
=\mathfrak{l}\oplus\mathfrak{u}$
the parabolic subalgebra defined by $\Pi'$.
The following lemma is straightforward.

\begin{lemma}\label{prop:lambda}
Let $\mathfrak g$ be a complex semisimple Lie algebra, and
$\mathfrak{p}$ be a parabolic subalgebra defined by $\Pi-\{\alpha_n\}$.
Suppose
$\lambda\in\mathfrak{h}^*$. Then the following two conditions are equivalent:
\begin{enumerate}
\item[$(1)$]
$\lambda$ can be extended to a character on $\mathfrak{p}$, i.e.,
$\lambda|_{[\mathfrak{p},\mathfrak{p}]}=0, \lambda|_\mathfrak{h}=\lambda$;
\item[$(2)$]
there exists a complex number
$z\in\mathbb C$ such that
$\lambda (h_n)=z$, while $\lambda (h_1)=\cdots=\lambda (h_{n-1})=0.$
\end{enumerate}
\end{lemma}

Based on this lemma, Garfinkle showed the following.

\begin{lemma}[{\cite[\S IV.3 Proposition 3, \S IV.6 Theorem 1 and \S V Theorem 1]{Ga}}]\label{Vtheta}
Let $v\in V(\theta)^{\mathfrak{n}^-}$, $\mathfrak{p}$ be the parabolic subalgebra of
$\mathfrak{g}$ defined by $\Pi-\{\alpha_n\}$, and $\lambda\in\mathfrak{h}^*$
satisfy the conditions in
Lemma \ref{prop:lambda}. Then there exists an element
$y\in\mathcal{U}_1(\mathfrak{g})^{\mathfrak{n}^-}$ depending on
$\lambda$ such that $\beta(v)-y\in I_{\mathfrak{p},\lambda}$, where $I_{\mathfrak{p},\lambda}$ be the left ideal of the universal enveloping algebra
$\mathcal{U}(\mathfrak{g})$ generated by $\{x-\lambda(x)|x\in\mathfrak{p}\}$. In this case, $\beta(v)-y\in J$.	
\end{lemma}

More explicitly, we have that
\begin{align}\label{vinU}
 \beta(v)-y=&-(n+1)(Y_{2\cdots n}Y_1+Y_{3\cdots n}Y_{12}+\cdots +Y_nY_{1\cdots n-1})\notag \\
&+Y_\theta\left(\sum_{l=1}^n (2l-1-n)h_l-\lambda\Big(\sum_{l=1}^n (2l-1-n)h_l\Big)\right)\notag\\
=&\dfrac{-(n+1)}{2}\sum_{k=2}^n(Y_{k\cdots n}Y_{1\cdots,k-1}
 +Y_{1\cdots k-1}Y_{k\cdots n})+ Y_\theta\left(\sum_{l=1}^n (2l-1-n)h_l\right)\notag\\
 &-\dfrac{(n-1)(n+1+2z)}{2}Y_\theta
 \end{align}
is an element in the Joseph ideal $J$.

Third, we find the generator of the Joseph ideal that corresponds to
the Casimir element of $\mathfrak g$. Let us denote by $C$ the Casimir element.
We have the following.

\begin{lemma}[{\cite[\S IV.3]{Ga}}]\label{prop:vtheta0An}
Let $\mathfrak g$ be the $A_n$ Lie algebra. Then
\begin{align}\label{vtheta0An}
 C-c_\lambda =&\sum_{\alpha\in\Pi}(X_\alpha Y_\alpha+ Y_\alpha X_\alpha)+
 \sum_{i=1}^n h_i\cdot\dfrac{1}{n+1}\Big((n-i+1)\big(h_1+2h_2+\cdots+(i-1)h_{i-1}\big)\notag\\
 &+i\big((n-i+1)h_i+(n-i)h_{i+1}+\cdots+h_n\big)\Big)-n\left(\dfrac{z}{n+1}+1\right)z
\end{align}
is a generator of $J$, where $c_\lambda=\langle\lambda,\lambda\rangle+\langle\lambda,2\delta\rangle$ and
$\delta$ is the half of the sum of positive roots. 	
\end{lemma}

\begin{proof}
The Casimir element is $C=\sum_{\alpha\in\Pi}X_\alpha Y_\alpha+
Y_\alpha X_\alpha+\sum_{i=1}^n h_ih_i^{\vee}$,
where $n$ is the rank of the corresponding Lie algebra.
	
For Lie algebra of $A_n$,
$2\delta=n\alpha_1+2(n-1)\alpha_2+\cdots+i(n-i+1)\alpha_i+\cdots+n\alpha_n$.
By Lemma \ref{prop:lambda}, we have $\lambda=z\lambda_n$. Thus
\[
c_\lambda=n\left(\dfrac{z}{n+1}+1\right)z.
\]
By \cite[\S IV.3 \S IV.6 Theorem 1 and \S V Theorem 1]{Ga}, $C-c_\lambda$
is an element of $J$.
\end{proof}

By Garfinkle\cite{Ga}, $J$ depends on an element $z\in\mathbb C$; to specify
its dependence on $z$, in what follows we shall write it as $J^z$.
Summarizing the above lemmas, we have the following:

\begin{theorem}[\cite{Ga}]\label{Prop:GaJosephAn}
Let $\mathfrak g$ be the type A Lie algebra. $v_0$ is the lowest weight vector in Lemma \ref{Garfinkle lemma1}.
For each $z\in\mathbb C$, there is a Joseph ideal in $\mathcal U(\mathfrak g)$, denoted by $J^z$, which is
generated by \eqref{vinU}, \eqref{vtheta0An} and $\beta(v_0)$,
where $v_0$ is given in Lemma \ref{Garfinkle lemma1}.
\end{theorem}

\subsubsection{Joseph ideal for type D and E Lie algebras}

Now we consider the Lie algebra $\mathfrak g$ of DE type.
Let $\alpha$ be the simple root not orthogonal to the highest root $\theta$;
in the case of type D and $E_6$, $E_7$, $E_8$, such an $\alpha$
is unique.

 \begin{proposition}[{see \cite{Ga}, \cite[\S 4.4]{BJ}  and \cite{GS}}]\label{structureofDE}
 Let $\mathfrak g$ be the complex semisimple Lie algebra of DE type.
 Let $\{\theta_i\}_i$ be the set of the highest roots of the complex semisimple Lie algebras obtained from
 $\mathfrak{g}$ by deleting $\alpha$ from the Dynkin diagram of $\mathfrak{g}$.
 Then we have the following decomposition of irreducible representations:
\begin{eqnarray*}
	\mathrm{Sym}^2(\mathfrak{g})= V(2\theta)
	\bigoplus\oplus_i V(\theta+\theta_i)\bigoplus V(0).
\end{eqnarray*}
\end{proposition}

For the type D Lie algebras, the unique simple root which is not
perpendicular to $\theta$ is precisely the simple root $\alpha_2$,
and thus we have the following:

\begin{fact}\label{fact 1}
For the $D_n (n> 4)$ Lie algebra $\mathfrak{g}$, we have the decomposition of
irreducible representations:
\begin{eqnarray*}
	\mathrm{Sym}^2(\mathfrak{g})\cong V(2\theta)\oplus V(\theta+\theta{'})\oplus V(\theta+\alpha_1)\oplus V(0),
\end{eqnarray*}
where $\theta{'}=\alpha_3+2\alpha_4+\cdots+2\alpha_{n-2}+\alpha_{n-1}+\alpha_n$ is the
highest root of the Lie algebra corresponding to the sub-Dynkin diagram $D_{n-2}$ of
$D_n$, which consists of the roots $\alpha_3,\cdots, \alpha_n$.

For the $D_4$ Lie algebra, we have the decomposition of
irreducible representations:
\begin{eqnarray*}
\mathrm{Sym}^2(\mathfrak{g})\cong V(2\theta)\oplus V(\theta+\alpha_1)
\oplus V(\theta+\alpha_3)\oplus V(\theta+\alpha_4)\oplus V(0).
\end{eqnarray*}
\end{fact}

For type E Lie algebras, we have the following.

\begin{fact}\label{fact 2}
\textup{(1)} For the $E_6$ Lie algebra $\mathfrak g$, $\alpha=\alpha_2$,
and
we have the following decomposition of representation:
 \begin{eqnarray*}
	\mathrm{Sym}^2(\mathfrak{g})\cong V(2\theta)\oplus V(\theta+\alpha_1+\alpha_3+
		\alpha_4+\alpha_5+\alpha_6)\oplus V(0),
		\end{eqnarray*}
where $\theta$ is the highest root of Lie algebra of type $E_6$.

\textup{(2)} For the $E_7$ Lie algebra $\mathfrak g$,
we have the following decomposition of representation:
\begin{eqnarray*}
		\mathrm{Sym}^2(\mathfrak{g})\cong V(2\theta)\oplus V(\theta+\alpha_2+
		\alpha_3+2\alpha_4+2\alpha_5+2\alpha_6+\alpha_7)\oplus V(0),
\end{eqnarray*}
where $\theta$ is the highest root of Lie algebra of type $E_7$, i.e.,
$\theta=2\alpha_1+2\alpha_2+3\alpha_3+4\alpha_4+3\alpha_5+2\alpha_6+\alpha_7$.

\textup{(3)} For the $E_8$ Lie algebra $\mathfrak g$
we have the following decomposition of representation:
 \begin{eqnarray*}
		\mathrm{Sym}^2(\mathfrak{g})= V(2\theta)\oplus
		V(\theta+2\alpha_1+2\alpha_2+3\alpha_3+
		4\alpha_4+3\alpha_5+2\alpha_6+\alpha_7)\oplus V(0),
		\end{eqnarray*}
where $\theta$ is the highest root of Lie algebra of type $E_8$, i.e.,
$\theta=2\alpha_1+3\alpha_2+4\alpha_3+6\alpha_4+5\alpha_5+4\alpha_6+
3\alpha_7+2\alpha_8$.
\end{fact}

By Kostant (see \cite{Ga} \S III.2), the ideal $I(\overline{\mathcal{O}}_{min})$ is generated by the lowest
weight vectors $v_i$ in each summand of $\oplus_i V(\theta+\theta_i)$ and $C$ in $V(0)$.
We have the following:

\begin{theorem}[{\cite[\S IV.3 Theorem 2, \S IV.6 Theorem 1 and \S V]{Ga}}]\label{Garfinkle lemma1 Dn}
Let $\g$ be the Lie algebra of type
D or E, let $v_i$ be a lowest weight vector of the irreducible
representation in $V(\theta+\theta_i)$ in
Proposition \ref{structureofDE}.
Then the Joseph ideal $J$ is generated by $\beta(v_i)$ and
$C-c_\lambda$, where in the $D_n$ case, $c_\lambda=2n-n^2$,
and in the $E_6$, $E_7$ and $E_8$ case,
$c_\lambda=-36,\ -84$ and $-240$ respectively.
\end{theorem}

\begin{remark}
According to {\cite[\S IV.4, \S IV.6 Theorem 1 and \S V]{Ga}}, in the $D_n$ case,
$\lambda(h_1)=-(n-2)$, $\lambda(h_2)=\cdots=\lambda(h_n)=0$.
In the $E_6$ case, $\lambda(h_6)=-3$, $\lambda(h_1)=\cdots=\lambda(h_5)=0$.
In the $E_7$ case,
$\lambda(h_7)=-4$, $\lambda(h_1)=\cdots=\lambda(h_6)=0$. And in the $E_8$ case,
 $\lambda(h_8)=-5$, $\lambda(h_1)=\cdots=\lambda(h_7)=0$.
 Recall that $c_\lambda=\langle\lambda,\lambda\rangle+\langle\lambda,2\delta\rangle$, where
$\delta$ is the half of the sum of positive roots, we get the values of $c_\lambda$ in the above
theorem.	
\end{remark}

\subsection{The $B$-algebras}\label{sect:B-algebra}

Suppose $\g$ is a simple Lie algebra, and $Q$ is the root lattice.
Let $\mathcal{U}(\g)$ be the universal enveloping algebra of $\g$, and
$J$ be the corresponding Joseph ideal. Recall that there is the PBW filtration of $\mathcal{U}(\g)$:
$$\mathcal{U}^0\subseteq \mathcal{U}^1 \subseteq \mathcal{U}^2\subseteq \cdots $$
 On the other hand, $\mathcal{U}(\g)$ have the following weight decomposition
$$\mathcal{U}(\g)=\bigoplus_{\mu\in Q}\mathcal{U}_\mu.$$
Furthermore, the Joseph ideal $J$ can be split as
\begin{equation}\label{J decomp}
	J=\bigoplus_{\mu\in Q}J_\mu=\bigoplus_{\mu\in Q}J\cap \mathcal{U}_\mu.
\end{equation}
Denote $\A=\A[\overline{\mathcal O}_{min}]:=Rees(\mathcal{U}(\g)/J)$,
and then there is a weight decomposition induced by that of $\mathcal{U}(\g)$,
$$\A=\bigoplus_{\mu\in Q}\A_\mu,$$
where $\A_\mu=\mathcal{U}_\mu/J_\mu$.

\begin{definition}\label{def:B-algebra}
The {\it B-algebra} of $\A[\overline{\mathcal O}_{min}]$ is defined to be
	\begin{equation*}
B(\A)=B(\A[\overline{\mathcal O}_{min}]):=\A_0\Big/\sum_{\mu\in\Delta^{+}}\{ab|a\in \A_\mu, b\in \A_{-\mu}\}.
	\end{equation*}
\end{definition}

\section{The quantum Hikita conjecture}\label{sect:thequantumHikitaconj}

As we have mentioned before, the quantum Hikita conjecture was proposed
by Kamnitzer, McBreen and Proudfoot in \cite{KMP}. Under some conditions,
the quantum Hikita conjecture implies the Hikita-Nakajima equivariant conjecture.
In this section, we first recall the two main objects in the quantum Hikita conjecture,
and then prove Theorems \ref{maintheorem0} and \ref{maintheorem}.

\subsection{Specialized quantum D-module}\label{sect:quatum D-module}
Let $X$ be a conical symplectic variety. Assume that $X$ admits
a $T\times\mathbb G_m$-equivariant projective symplectic resolution
$\tilde X$. There is a natural short exact sequence (see \cite[(6)]{KMP})
$$
0\to \mathrm H^2_{T\times\mathbb C^\times}(pt; \mathbb C)
\to \mathrm H^2_{T\times\mathbb C^\times}(\tilde X; \mathbb C)
\to\mathrm H^2(\tilde X; \mathbb C)
\to 0,$$
which is called the {\it cohomology exact sequence}.
The image of an element $u\in\mathrm H^2_{T\times\mathbb C^\times}(\tilde X; \mathbb C)$
in $\mathrm H^2(\tilde X; \mathbb C)$
is henceforth denoted by $\bar u$.

Now let $\mathrm H_2(\tilde X; \mathbb Z)_{\mathrm{free}}$ be the quotient
of $\mathrm H_2(\tilde X; \mathbb Z)$ by its torsion subgroup.
Okounkov conjectured that there is a finite set
$\Delta_{+}\subset\mathrm H_2(\tilde X; \mathbb Z)_{\mathrm{free}}$ and
an element
$L_\alpha\in\mathrm H^{2\dim X}(\tilde X\times_X\tilde X; \mathbb C)$
for each $\alpha\in\Delta_{+}$ such that
$$
u\star(-)=u\cup(-)+\hbar\sum_{\alpha\in\Delta_{+}}\langle\alpha,\bar u\rangle
\frac{q^\alpha}{1-q^\alpha}L_\alpha(-),
$$
for any $u\in\mathrm H^2_{T\times\mathbb C^\times}(\tilde X; \mathbb C)$,
where $\star$ is the quantum product. The minimal such subset $\Delta_{+}$ is called the set of positive
\textit{K\"ahler roots}.

\begin{remark}
Acording to \eqref{eq: quantum coh DE} and \eqref{qcohAn}, in our case,
the set $\Delta_{+}$ above is identified with the set of positive roots $\Delta^+$ of the
corresponding Lie algebra, after identifying $E_i$ with the simple roots $\alpha_i$.
And Okounkov's conjecture holds in our case, by setting
$L_\alpha=-(pr_1)^*\mathrm{PD}(\alpha)\cup (pr_2)^*\mathrm{PD}(\alpha)$ for each $\alpha\in \Delta_+$,
where $pr_i: \tilde X\times_X\tilde X\rightarrow \tilde X$ is the $i$-th projection,
$i=1,2$, and $\mathrm{PD}(\alpha)$ is the Poincar\'e dual of $\alpha$. Also,
$\hbar$ is $\dfrac{t_1+t_2}{2}$ in the type A case, and is $t$ in the type DE case.
\end{remark}

We next introduce several algebraic structures.
Let
$
F:=\mathbb C\{q^\alpha\,|\,\alpha\in\mathbb Z\Delta^{+}\}
$
and
$F_{\mathrm{reg}}:=F\left[\frac{1}{1-q^\alpha}\Big|\,\alpha\in\Delta^{+}\right]$.
Let
$E:=F\otimes \mathrm{Sym}\mathrm
H^2_{T\times \mathbb C^\times}(\tilde X; \mathbb C)
$ equipped with the multiplication satisfying
$uq^\alpha=q^\alpha(u+\hbar\langle\alpha,\bar u\rangle)$
for all $\alpha\in\mathbb N\Delta_+$ and $u\in\mathrm{Sym}
\mathrm H^2_{T\times\mathbb C^\times}(\tilde X; \mathbb C)$.
Let $E_{\mathrm{reg}}$ be the Ore localization with respect to the multiplication
set generated by $1-q^\alpha$ for $\alpha\in\Delta_+$
(it is showed in \cite[\S4.2]{KMP} that the multiplicative set
satisfies the Ore condition).

Let
$Q_{\mathrm{reg}}(\tilde X):=F_{\mathrm{reg}}\otimes
\mathrm H^\bullet_{T\times\mathbb C^\times}(\tilde X;\mathbb C)$.
Then $E_{\mathrm{reg}}$ acts on $Q_{\mathrm{reg}}$ as follows:
elements in $F_{\mathrm{reg}}\subset E_{\mathrm{reg}}$
acts by multiplication on the first tensor factor, while an element
$u\in \mathrm H^2_{T\times\mathbb C^\times}(\tilde X;\mathbb C)$
acts by the operator
$\hbar\partial_{\bar u}+u\star$, where $\partial_{\bar u}(q^\alpha)=
\langle \alpha,\bar u\rangle q^\alpha$.

\begin{definition}
The {\it specialized quantum D-module}
of $\tilde X$ is the $E_{\mathrm{reg}}$-module $Q_{\mathrm{reg}}(\tilde X)$.
\end{definition}

\begin{remark}\label{rem:quantumDmoduleinADEcase}
Let us move to the case of the minimal resolutions
of ADE singularities
$\widetilde{{\mathbb C^2}/\Gamma}$. In this case,
the quantum cohomology is generated by the exceptional
divisors, whose product contains no terms of degree higher than 2,
and therefore $Q_{\mathrm{reg}}(\widetilde{{\mathbb C^2}/\Gamma})$
is nothing but the equivariant quantum cohomology
algebra $\mathrm{QH}^\bullet(\widetilde{\mathbb C^2/\Gamma})$;
see also \cite[Remark 4.1]{KMP}.
Let $q=0$, then $Q_{\mathrm{reg}}(\widetilde{{\mathbb C^2}/\Gamma})$
becomes the equivariant cohomology ring $\mathrm{H}^\bullet(\widetilde{\mathbb C^2/\Gamma})$.
\end{remark}

\subsection{D-module of graded traces}\label{sect:D-module of graded traces}

We now introduce the notion of D-module of graded traces.
The general construction is quite complicated (see \cite[\S3]{KMP} for details),
and in
this section, we only focus on the case that $X$ is the minimal
nilpotent orbits in ADE type Lie algebras.

Let $X$ be a conical symplectic variety, suppose $T$ is a maximal torus of the
Hamiltonian action on $X$. Then there is an exact sequence (see \cite[\S 2.1]{KMP}):
\begin{eqnarray*}
0\rightarrow \mathrm{H}_2(\widehat{X}; \mathbb C)\oplus\mathbb C \hbar
\rightarrow\A^1_0\rightarrow\mathfrak t\rightarrow 0,
\end{eqnarray*}
where $\A^1_0$ denote the weight $0$ degree $1$ part of $\A$.
When $X=\mathcal{O}_{min}$ of type $A_n$, this exact sequence has the form:
\begin{eqnarray*}
0\rightarrow \mathbb C\oplus\mathbb C \hbar
\rightarrow\A^1_0\rightarrow\mathfrak{h}\rightarrow 0,
\end{eqnarray*}
and when $X=\mathcal{O}_{min}$ of DE type, this exact sequence has the form:
\begin{eqnarray*}
0\rightarrow \mathbb C \hbar\rightarrow\A^1_0\rightarrow\mathfrak{h}\rightarrow 0.
\end{eqnarray*}
All the $\mathfrak h$ above are the Cartan subalgebra
corresponding to the Lie algebra type.

\begin{remark}\label{element in A^0_1}
Notice that choosing a splitting of the exact sequence above is
equivalent
to choosing a quantum comoment map
$\mathrm{Sym}\mathfrak h\rightarrow\mathrm{Sym}\A_{0}^1$, and for simplicity,
the image of $h_\alpha\in\mathfrak h$ under this quantum comoment map
is also denote by $h_\alpha$, which is an element in $\A^1_0$.
\end{remark}

For
$X=\overline{\mathcal O}_{min}\subset\mathfrak g$,
let $\Delta^+$ be the set of positive roots of $\mathfrak g$,
which is called the {\it equivariant roots} of $X$.
Let $S:=\mathbb C\{q^\mu|\mu\in \mathbb N\Delta^{+}\}$ and
$S_{\mathrm{reg}}:=S\bigg[\dfrac{1}{1-q^\mu}\bigg|\mu\in  \Delta^+ \bigg]$ be the
localization of $S$.
Let $R:=S\otimes\mathrm{Sym}\mathscr A_0^1$ be the $\mathbb C[\hbar]$-algebra
with the multiplication satisfying
$xq^\mu=q^\mu(x+\hbar\langle \lambda,\bar x\rangle)$,
for all $\mu\in \mathbb N\Delta^+$ and $x\in \mathscr A_0^1$; here $\bar x\in\mathfrak h$
by the above exact sequence.
Let $\mathfrak S\subset R$ be the multiplicative set
generated by $1-q^\mu$ for all $\mu\in\Delta^+$; it
is shown in \cite[Lemma 3.4]{KMP} that $\mathfrak S$ satisfies
the Ore condition, and hence we may define
the Ore localization $R_{\mathrm{reg}}:=\mathfrak S^{-1}R$,
which, as a vector space, is isomorphic to $S_{\mathrm{reg}}\otimes\mathrm{Sym}\mathscr
A_0^1$.

Now we endow $S\otimes \A_0$ with the structure of an
$\mathbb N$-graded left $R$-module by
putting
\begin{eqnarray}\label{eq:h-action}
h_\alpha\cdot (q^\gamma\otimes a):=
q^\gamma\otimes(h_\alpha+\hbar\langle
\gamma,\alpha\rangle) a \quad\mbox{ and  }\quad
q^\mu(q^\gamma\otimes a):=q^{\mu+\gamma}\otimes a,
\end{eqnarray}
for all $h_\alpha\in\A_0^1$, $a\in\A_0$ and $\gamma,\mu\in\Delta^+$.
Let
\begin{eqnarray*}
\mathcal{I}_q:=\sum_{\mu\in\Delta^{+}}
S\cdot \left\{1\otimes ab-q^\mu\otimes ba | a\in\mathscr{A}_\mu, b
\in\mathscr{A}_{-\mu}\right\}\subset S\otimes\mathscr A_0,
\end{eqnarray*}
which turns out to be an $R$-submodule (see \cite[Proposition 3.5]{KMP}).

\begin{definition}[{\cite[\S 3.3]{KMP}}]\label{def:Dmodofgradedtraces}
Let $X=\overline O_{min}$ be the closure of the minimal nilpotent orbit
in $\mathfrak g$. Then the {\it D-module of graded traces}
of $X$
is $M:=(S\otimes\mathscr A_0)/\mathcal{I}_q$.
Its localization $M_{\mathrm{reg}}:=R_{\mathrm{reg}}\otimes_RM$ is denoted by
$Q(\mathscr{A}[\overline O_{min}])$ or simply
$Q(\mathscr{A})$.
\end{definition}

\begin{remark}\label{re:reductionofQ(A)}
(1) Since the Poisson bracket on $\mathbb C[\overline{\mathcal{O}}_{min}]$
has degree $-1$, the algebra $\A^1_0$ plays the role of $\A^2_0$ in \cite{KMP}.

(2) In fact, when $q=0$,
the submodule $\mathcal{I}_q$ reduces to the two-sided ideal
$$\sum_{\mu\in\Delta^{+}}\{ab|a\in \A_\mu, b\in \A_{-\mu}\}$$
 of $\A_0$ in Definition \ref{def:B-algebra}. By \cite[Proposition 3.8]{KMP},
$Q(\mathscr{A})|_{q=0}$ is isomorphic to degree 0 Hochshchild homology $\mathrm{HH}_0(B(\A))$ as $\mathrm{Sym} \mathscr A_0^1$-module. In our case, $B(\A)$ is commutative (see Lemma \ref{lemma:B-alg str} below), so $\mathrm{HH}_0(B(\A))\cong B(\A)$.
\end{remark}

\subsection{D-module of graded traces for minimal nilpotent orbits}\label{sect:Q(A)}

In  this subsection, we
study the D-module of graded traces of the minimal nilpotent
orbits in Lie algebras of ADE types.

Firstly we define a $S_{reg}$-linear map
\begin{align*}
\kappa: S_{reg}\otimes\mathcal{U}_0\rightarrow S_{reg}\otimes\mathcal{U}(\h).
\end{align*}
Fix a total order on the positive root set $\Delta^+$,
we obtain a PBW basis $\{\phi_i\}$ of $\mathcal{U}
=\mathcal{U}(\mathfrak{n}^{+})\otimes\mathcal{U}(\mathfrak{h})\otimes\mathcal{U}(\mathfrak{n}^{-})$.
Now, we define $\kappa$ inductively with respect to
the degree $k$ on the PBW filtration $\mathcal{U}^k_0=\mathcal{U}^k\cap \mathcal{U}_0$:

\begin{enumerate}
\item[(1)] For $k=0,1$, $\kappa$ is the identity map
(since $\mathcal{U}^1_0\cong\mathfrak h$).

\item[(2)] Assume that for $k-1$, $\kappa$ has been defined.
Now consider $\phi\in S_{reg}\otimes\mathcal{U}_0^k$. Suppose
$\phi=\sum_{i=1}^mk_i\phi_i$ under the PBW basis.
For $\phi_i\in\mathcal{U}(\h)$, set $\kappa(\phi_i)=\phi_i$;
for $\phi_i\notin\mathcal{U}(\h)$, since $\phi_i\in\mathcal{U}_0$ is a PBW basis,
$\phi_i\in \mathfrak n^{+}\mathcal{U}(\g)$, i.e., $\phi_i=X_\gamma\cdot a_{-\gamma}$
for $X_\gamma\in\g_\gamma$,
$a_{-\gamma}\in\mathcal{U}_{-\gamma}, \gamma\in\Sigma_{+}$,
set \begin{eqnarray*}
\kappa(\phi_i)=\kappa
\left(\dfrac{q^\gamma}{q^\gamma-1}[X_\gamma,a_{-\gamma}]\right).\end{eqnarray*}
 Since $[X_\gamma,a_{-\gamma}]\in\mathcal{U}_0^{k-2}$,
 this is well-defined by the induction assumption.
\end{enumerate}

\begin{lemma}\label{lem: kappa well-defined}
Set \begin{eqnarray*}
\widetilde{\mathcal{I}}_q:=\sum_{\mu\in\mathbb N\Delta^{+}}
S_{reg}\cdot \left\{1\otimes ab-q^\mu\otimes ba | a\in\mathcal{U}_\mu, b
\in\mathcal{U}_{-\mu}\right\}\subset S_{reg}\otimes\mathcal{U}_0.
\end{eqnarray*}
Then for $\phi\in\widetilde{\mathcal{I}}_q$, $\kappa(\phi)=0$.
\end{lemma}

\begin{proof}
Using the relation $a_{\mu}a_{-\mu}-a_{-\mu}a_{\mu}=[a_{\mu}, a_{-\mu}]$, 
one can deduce that $$\widetilde{\mathcal{I}}_q=\sum_{\mu\in\mathbb N\Delta^{+}}
S_{reg}\cdot \left\{1\otimes ab-\dfrac{q^\mu}{q^\mu-1}\otimes[a,b] | a\in\mathcal{U}_\mu, b
\in\mathcal{U}_{-\mu}\right\}.$$ It suffices to consider a monomial
$\phi=a_1\cdots a_n\in\mathcal{U}_0, a_i\in\g$. For any $a_\mu\in\mathcal{U}_\mu, a_{-\mu}
\in\mathcal{U}_{-\mu}$, such that $\phi=a_\mu a_{-\mu}$, we will prove
\begin{eqnarray}\label{kappa of module}
\kappa(\phi)=\kappa\left(\dfrac{q^\mu}{q^\mu-1}[a_\mu,a_{-\mu}]\right).
\end{eqnarray}
We show \eqref{kappa of module} by the induction on $k$ with $\phi\in\mathcal{U}^k_0$.
Since $\phi=a_\mu a_{-\mu}$, we have $k\geq 2$.

\noindent (1)
For $k=2$, $a_\mu\in\g_\mu$, $a_{-\mu}\in\g_{-\mu}$, $\phi$ is a PBW basis and
$\kappa(\phi)=\kappa\left(\dfrac{q^\mu}{q^\mu-1}[a_\mu,a_{-\mu}]\right)$.

\noindent (2)
Assume for $\phi=a_\mu a_{-\mu}\in\mathcal{U}^{k-1}_0$,
$\kappa(\phi)=\kappa\left(
\dfrac{q^\mu}{q^\mu-1}[a_\mu,a_{-\mu}]\right)$ .
\begin{itemize}
\item[(2a)]Consider $\phi=a_\mu a_{-\mu}\in\mathcal{U}^k_0$.
Then $a_\mu$ is of the form $a_\mu=b\cdot v\cdot c$,
where $v\in\mathfrak{n}^+$ and $b, c$ is some monomial in $\mathcal{U}$.
Set $\tilde{a}_\mu=v\cdot b\cdot c$,
$\tilde{\phi}=\tilde{a}_\mu a_{-\mu}$, then
$\phi-\tilde{\phi}=[b,v]\cdot c$, $a_{-\mu}\in \mathcal{U}_{-\mu}^{ k-1}$.
Furthermore,
\begin{align*}
\dfrac{q^\mu}{q^\mu-1}[a_\mu,a_{-\mu}]-\dfrac{q^\mu}{q^\mu-1}[\tilde{a}_\mu,a_{-\mu}]=\dfrac{q^\mu}{q^\mu-1}[[b,v],a_{-\mu}]
\end{align*}
By the induction assumption,
we have
$\kappa(\phi-\tilde{\phi})=\kappa\left(\dfrac{q^\mu}{q^\mu-1}[a_\mu-\tilde{a}_\mu,a_{-\mu}]\right)$.
Therefore, that \eqref{kappa of module} holds for $\phi$ is equivalent to that
\eqref{kappa of module} holds for $\tilde{\phi}$.
This argument reduces the proof of \eqref{kappa of module} to the case of $\phi=a_\mu a_{-\mu}$, where
both $a_\mu$ and $a_{-\mu}$ are PBW basis.

\item[(2b)] Suppose $a_\mu$ is a PBW basis,
then $a_\mu=a_{\mu+\gamma} a_{-\gamma}$,
where $\mu, \gamma\in\mathbb N\Delta^{+}$, and $a_{\mu+\gamma}\in \mathcal{U}(\mathfrak n^{+})$. We claim that
\begin{eqnarray}\label{kappa of module 2}
\kappa\left(\dfrac{q^\mu}{q^\mu-1}[a_\mu,a_{-\mu}]\right)
=\kappa\left(\dfrac{q^{\mu+\gamma}}{q^{\mu+\gamma}-1}
[a_{\mu+\gamma},a_{-\gamma}\cdot a_{-\mu}]\right).
\end{eqnarray}
In fact, we have
\begin{eqnarray*}
\kappa\left(\dfrac{q^\mu}{q^\mu-1}[a_{\mu+\gamma}\cdot a_{-\gamma},a_{-\mu}]\right)
=\dfrac{q^\mu}{q^\mu-1}\kappa\left([a_{\mu+\gamma},a_{-\mu}]a_{-\gamma}+a_{\mu+\gamma}[a_{-\gamma},a_{-\mu}]\right).	
\end{eqnarray*}
By the induction assumption, it equals
\begin{eqnarray*}
\dfrac{q^\mu}{q^\mu-1}\cdot\dfrac{q^\gamma}{q^\gamma-1}\kappa
\left(\big[[a_{\mu+\gamma},a_{-\mu}],a_{-\gamma}\big]\right)+	
\dfrac{q^\mu}{q^\mu-1}\cdot\dfrac{q^{\mu+\gamma}}{q^{\mu+\gamma}-1}\kappa
\left(\big[a_{\mu+\gamma},[a_{-\gamma},a_{-\mu}]\big]\right).
\end{eqnarray*}
Similarly, \begin{align*}
&\kappa\left(\dfrac{q^{\mu+\gamma}}{q^{\mu+\gamma}-1}
[a_{\mu+\gamma}, a_{-\gamma}\cdot a_{-\mu}]\right)\\
&=\dfrac{q^{\mu+\gamma}}{q^{\mu+\gamma}-1}
\cdot\dfrac{q^\mu}{q^\mu-1}\kappa\left(\big[[a_{\mu+\gamma},a_{-\gamma}],a_{-\mu}\big]\right)
+\dfrac{q^{\mu+\gamma}}{q^{\mu+\gamma}-1}
\cdot\dfrac{1}{q^{\gamma}-1}\kappa\left(\big[[a_{\mu+\gamma},a_{-\mu}],a_{-\gamma}]\big]\right).
\end{align*}
Then it is direct to check that \eqref{kappa of module 2} holds by the Jacobi identity.
Equation \eqref{kappa of module 2} reduces the proof of
\eqref{kappa of module} to the case of $\phi=a_\mu a_{-\mu}$, where
 $a_\mu\in\mathcal{U}(\mathfrak n^{+})$.

\item[(2c)] Now suppose  $\phi=a_\mu a_{-\mu}$, where
$a_\mu\in\mathcal{U}(\mathfrak n^{+})$. By a similar argument  in (2a), it suffices to prove \eqref{kappa of module}
in the case $a_{-\mu}$ is a PBW basis.
By argument in (2b), it suffices to consider the case  $\phi=a_\mu a_{-\mu}$, where
 $a_\mu\in \mathcal{U}(\mathfrak n^{+})$, $a_{\mu}\in \mathcal{U}(\h)\otimes \mathcal{U}(\mathfrak n^{-})$.
 By the argument in (2a) again, we can further assume  $a_{\mu}$ and $a_{\mu}$ are PBW basis, which makes $\phi$ be a PBW basis itself.

\item[(2d)]  Finally, perform the argument in (2b) on the $\phi$ with assumption in the end of (2c),
we reduce the proof of \eqref{kappa of module} to the case $\phi=a_{\mu}a_{-\mu}$,
where $\phi$ is a PBW basis and  $a_{\mu}\in \mathfrak n^{+}$. Then \eqref{kappa of module} holds by the definition of $\kappa$.
\end{itemize}
The proof of the lemma is now complete.
\end{proof}

\begin{remark}\label{rmk:kappa}
It is easy to see that $\kappa$ does not depend on the choice of the
PBW basis. In fact, one can calculate $\kappa(\phi)$ 
via the decomposition (cf. \cite{Ga})
\begin{eqnarray}\label{key decomp}
\mathcal{U}(\g)=(\mathfrak{n}^+\mathcal{U}(\g)+\mathcal{U}(\g)\mathfrak{n}^-)\oplus \mathcal{U}(\h).
\end{eqnarray}
For any $\phi \in\mathcal{U}(\g)_0$, we have $\phi =\tilde{\phi}+\phi_0$
where
$\phi_0\in\mathcal{U}(\mathfrak h)$, $\tilde{\phi}
\in(\mathfrak{n}^+\mathcal{U}(\g)+\mathcal{U}(\g)\mathfrak{n}^-)_0$.
Thus $\tilde{\phi}=\sum_{\mu\in\mathbb N\Delta^+}a_\mu a_{-\mu}$, and
\[\kappa(\phi)=\phi_0+\sum_{\mu\in\mathbb N\Delta^+}
\kappa\left(\dfrac{q^\mu}{q^\mu-1}[a_\mu, a_{-\mu}]\right).\]
The calculation of $\kappa([a_\mu, a_{-\mu}])$ is by repeating the above progress.
Furhtermore, by $a_{\mu}a_{-\mu}-a_{-\mu}a_{\mu}
=[a_{\mu}, a_{-\mu}]$, if $\phi=a_{-\mu} a_{\mu}$,
we have $\kappa(\phi)=\frac{1}{q^{\mu}-1}\kappa([a_\mu , a_{-\mu}])$.
\end{remark}

By Lemma \ref{lem: kappa well-defined}, $\kappa$ induces an $S_{reg}$-linear map
\begin{eqnarray*}
\overline{\kappa}: S_{reg}\otimes\mathcal{U}_0/\mathcal{I}_q
\rightarrow S_{reg}\otimes\mathcal{U}(\h).
\end{eqnarray*}

\begin{proposition}\label{prop:kappa iso}
$\overline{\kappa}$ is an isomorphism of
$R_{reg}:=S_{reg}\otimes\mathrm{Sym}\mathscr A_0^1$-modules.	
\end{proposition}

\begin{proof}
The injectivity of $\overline{\kappa}$ is by the definition of $\kappa$. The surjectivity $\overline{\kappa}$ is by the natural embedding $\mathcal{U}(\h)\hookrightarrow \mathcal{U}(\mathfrak{g})$. In what follows, we check
the compatibility between $\overline{\kappa}$ and the $h$-action, where $h\in\A_0^1$.

We perform an induction argument on the degree $k$ of PBW filtration of $\mathcal{U}(\mathfrak{g})$. For $k=0$ and $1$, the compatibility holds obviously.
Now it suffices to check that
$\overline{\kappa}(h\cdot a_\gamma a_{-\gamma})
=\overline{\kappa}\left(h\cdot\big(\dfrac{q^\gamma}{q^\gamma-1}
[a_\gamma,a_{-\gamma}]\big)\right)$. By \eqref{kappa of module},
\begin{align*}
\overline{\kappa}(h\cdot a_\gamma a_{-\gamma})
&=\overline{\kappa}\left(\dfrac{q^\gamma}{q^\gamma-1}[ha_\gamma,a_{-\gamma}]\right)\\
&=\overline{\kappa}\left(\dfrac{q^\gamma}
{q^\gamma-1}(h[a_\gamma,a_{-\gamma}]+[h,a_{-\gamma}]a_\gamma)\right)\\
&=\overline{\kappa}\left(\dfrac{q^\gamma}{q^\gamma-1}h[a_\gamma,a_{-\gamma}]
+\dfrac{q^\gamma}{q^\gamma-1}\cdot\dfrac{1}{q^\gamma-1}
\big[a_\gamma,[h,a_{-\gamma}]\big]\right)\\
&=\overline{\kappa}\left(\dfrac{q^\gamma}
{q^\gamma-1}h[a_\gamma,a_{-\gamma}]-\dfrac{q^\gamma}
{(q^\gamma-1)^2}\langle h,\gamma\rangle [a_\gamma,a_{-\gamma}]\right).
\end{align*}
On the other hand, by \eqref{eq:h-action},
\begin{align*}
h\cdot\left(\dfrac{q^\gamma}{q^\gamma-1}[a_\gamma,a_{-\gamma}]\right)
&=\dfrac{q^\gamma}{q^\gamma-1}h\cdot[a_\gamma,a_{-\gamma}]-\dfrac{\langle h,\gamma\rangle
q^\gamma(q^\gamma-1)+\langle h,\gamma\rangle q^\gamma\cdot q^\gamma}{(q^\gamma-1)^2}
[a_\gamma,a_{-\gamma}]\\
&=\dfrac{q^\gamma}{q^\gamma-1}h[a_\gamma,a_{-\gamma}]
-\dfrac{\langle h,\gamma\rangle q^\gamma}{(q^\gamma-1)^2}[a_\gamma,a_{-\gamma}].
\end{align*}
Applying $\overline{\kappa}$ to it and comparing with the previous formula,
we get the desired identity.
\end{proof}

\begin{lemma}\label{lem: Q(A) structure}
$
Q(\A)\cong (S_{reg}\otimes\mathcal{U}_0/\widetilde{\mathcal{I}}_q)/
\big(S_{reg}\otimes J/(S_{reg}\otimes J)\cap\widetilde{\mathcal{I}}_q\big).
$
\end{lemma}
	
\begin{proof}
One can check that, via the fact that $J$ is an ideal,
$Q(\A)\cong S_{reg}\otimes\mathcal{U}_0/(S_{reg}\otimes J+\widetilde{\mathcal{I}}_q)$,
which is induced by the projection $\mathcal{U}_0\rightarrow \mathcal{U}_0/J$.
On the other hand,
\begin{eqnarray*}
S_{reg}\otimes\mathcal{U}_0/(S_{reg}\otimes J+
\widetilde{\mathcal{I}}_q)\cong (S_{reg}\otimes\mathcal{U}_0/\widetilde{\mathcal{I}}_q)
/(S_{reg}\otimes J+\widetilde{\mathcal{I}}_q)/\widetilde{\mathcal{I}}_q,	
\end{eqnarray*}
and $(S_{reg}\otimes J+\widetilde{\mathcal{I}}_q)/\widetilde{\mathcal{I}}_q
\cong S_{reg}\otimes J/\big((S_{reg}\otimes J)\cap\widetilde{\mathcal{I}}_q\big)$.
The lemma follows.
\end{proof}

By Proposition \ref{prop:kappa iso} and \ref{lem: Q(A) structure}, we have
\begin{eqnarray}\label{eq:Q(A)}
	Q(\A)\cong S_{reg}\otimes Rees(\mathcal{U}(\mathfrak h)/\mathcal{J}_{\mathfrak{h}}),
\end{eqnarray}
where $\mathcal{J}_{\mathfrak{h}}$ is the image of $S_{reg}\otimes J/(S_{reg}\otimes J\cap\widetilde{\mathcal{I}}_q)$ via tha map $\overline{\kappa}$.

\begin{proposition}\label{prop:J_h generator}
As an $R_{reg}$-module,	$\mathcal{J}_{\mathfrak{h}}$
is generated by $\kappa\big(1\otimes(J\cap \mathcal{U}_0^2)\big)$.
\end{proposition}

\begin{proof}
Pick an arbitrary element $a\in S_{reg}\otimes J$,
it is sufficient
to prove that, there exists $w_i\in J\cap\mathcal{U}_0^2$ and 
$u_i\in S_{reg}\otimes\mathcal{U}(\mathfrak{h})$ such that 
$a=\sum_i u_i\cdot w_i$ in $S_{reg}\otimes\mathcal{U}_0/\mathcal{I}_q$.
In fact, by Theorems \ref{Prop:GaJosephAn} and \ref{Garfinkle lemma1 Dn}, 
$a=\sum_{j=1}^m b_j\cdot\tilde{w}_j$ in $S_{reg}\otimes\mathcal{U}_0$,
by some $\tilde{w}_j\in J\cap\mathcal{U}^2$, $b_j\in S_{reg}\otimes\mathcal{U}$.
Without loss of generality, suppose $m=1$, i.e., $a=b\cdot\tilde{w}$.
Assume $a\in S_{reg}\otimes\mathcal{U}^k_0$ and we perform the induction on
$k$.

\noindent (1) For $k=2$, $a=\tilde{w}$, then $\tilde{w}\in J\cap\mathcal{U}^2_0$.

\noindent (2) Assume for $k-1$, the claim holds.
Consider the case $a\in S_{reg}\otimes\mathcal{U}^{k}$.	
Suppose $a=b_1b_2\cdots b_{k-2}\cdot\tilde{w}$, $b_i\in\g$.
\begin{enumerate}	
\item[(2a)]
If $b_1\in\g_\mu$, $\mu\in\Delta^{+}$,
then $a=\dfrac{q^\mu}{q^\mu-1}[b_1,b_2\cdots b_{k-2}\tilde{w}]$
in $S_{reg}\otimes\mathcal{U}_0/\mathcal{I}_q$.
We then apply the induction assumption on
$\dfrac{q^\mu}{q^\mu-1}[b_1,b_2\cdots b_{k-2}\tilde{w}]$.

\item[(2b)]
If $b_1\in\g_{-\mu}$, then $a=\dfrac{1}{q^\mu-1}[b_2\cdots b_{k-2}\tilde{w},b_1]$, then
similar to (2a), we apply
the induction assumption.
	
\item[(2c)]	
If $b_1\in\mathfrak h$, we consider $b_1b_2\in\g_{\mu}$ or in $\g_{-\mu}$ and perform a similar argument as before.

\item[(2d)]	
Finally, we are left to consider the case where all
$b_i\in\mathfrak h$. In this case, $\tilde{w}\in J\cap\mathcal{U}_0^2$.
\end{enumerate}
Thus we proved that $a=b\cdot\tilde{w}$,
where $b\in\mathcal{U}(\mathfrak h)$, $\tilde{w}\in J\cap\mathcal{U}_0^2$.
\end{proof}

Proposition \ref{prop:J_h generator} together with \eqref{eq:Q(A)} tells us that,
to calculate $Q(\A)$, we only need to calculate $J_0\cap \mathcal{U}^2$.

We briefly come back to the study of B-algebra $B(\A)$.
\begin{lemma}\label{lemma:B-alg str}
  As an algebra, $B(\A)\cong \mathcal{U}(\mathfrak h)/(\mathcal{J}_{\mathfrak{h}}|_{q=0})$.
\end{lemma}
\begin{proof}
  By setting $q=0$ in \eqref{eq:Q(A)}, we know $B(\A)\cong \mathcal{U}(\mathfrak h)/\mathcal{J}_{\mathfrak{h}}|_{q=0}$ as vector spaces. Furthermore, it is easy to check that $\kappa$ becomes a morphism between algebras and $\mathcal{J}_{\mathfrak{h}}$ becomes an ideal when $q=0$.
\end{proof}

\subsubsection{Some calculations on dimension}
In this subsection we calculate the dimension $J_0\cap \mathcal{U}^2$. Suppose $I={\rm gr}J=\bigoplus_{k\in \N}I^k$,
where $I^k=(J\cap \mathcal{U}^k)/(J\cap \mathcal{U}^{k-1})$. Then $I$ is an ideal of the polynomial ring
${\rm gr}(\mathcal{U}(\g))= \mathrm{Sym}(\g)$, and the degree of elements in $I^k$ is $k$.
For $\mu\in Q$, set $I^k_{\mu}$ be the component of $I^k$ with weight $\mu$.
	
\begin{lemma}\label{lemma:dimJU2}With the above notations, we have
\begin{equation}\label{dimJU2}
\dim(J_0\cap \mathcal{U}^2)=\dim I^2_0.
\end{equation}
\end{lemma}

\begin{proof}
Since $J\neq \mathcal{U}(\g)$, $1\notin J$, $J\cap \mathcal{U}^0=\{0\}$. Then as a vector space,
\begin{align*}
 J\cap \mathcal{U}^1	\cong I^1=(J\cap \mathcal{U}^1)/(J\cap \mathcal{U}^0).
\end{align*}
By a theorem of Kostant (see \cite[Theorem III.2.1]{Ga}),
as an ideal of $ \mathrm{Sym}(\g)$, $I$ is generated by $I^2$,
which contains the homogenous elements with degree 2.
Therefore $I^1=\{0\}$, and $ J\cap \mathcal{U}^1	=\{0\}$.

Now we consider $ J\cap \mathcal{U}^2$, as a vector space,
\begin{align*}
 J\cap \mathcal{U}^2\cong I^2=(J\cap \mathcal{U}^2)/(J\cap \mathcal{U}^1).
\end{align*}
Since the projection $J\cap \mathcal{U}^k\rightarrow
(J\cap \mathcal{U}^k)/(J\cap \mathcal{U}^{k-1})$
is compatible with the decomposition \eqref{J decomp},
we have $J_0\cap \mathcal{U}^2\cong I^2_0$.
\end{proof}

By Konstant, we have the following decomposition of $\g$-module (see \cite{Ga} or \cite{Sh}).

\begin{theorem}[Kostant]\label{thm: Kostant}
Suppose $\g$ is a semisimple Lie algebra, and $\theta$ is the highest weight of the adjoint
representation $\g$. Then as a $\g$-module,
\begin{equation*}
	{\rm Sym}^2 \g\cong V(2\theta)\oplus L_2,
\end{equation*}
where $V(2\theta)$ is the irreducible representation of highest weight $2\theta$,
and $L_2$ is a representation with underlying space $I^2$.
\end{theorem}

For a $\g$-module $V$, denote by $V_0$ the subspace of $V$ with weight $0$,
then we have
the following.

\begin{lemma} With the notations as above, we have:
\begin{equation}\label{dimension of I}
		\dim(I^2_0)=\dim({\rm Sym}^2 \g)_0- \dim V(2\theta)_0,
\end{equation}
where
\begin{equation}\label{dimSymh}
	\dim({\rm Sym}^2 \g)_0=\frac{\dim\g-\dim\h}{2}+\dim ({\rm Sym}^2\h).
\end{equation}
\end{lemma}

\begin{proof}
Just notice that all elements in $({\rm Sym}^2 \g)_0$ is a linear combination of
$x_\mu x_{-\mu}$, $x_\mu\in \g_{\mu}$ and $h_i h_j$, $h_i, h_j\in \h$.
\end{proof}

The calculation of $\dim V(2\theta)_0$ is more difficult. We just state the result here 
and defer the proof to appendix \ref{appendix}.

\begin{lemma}\label{lemmma:dimV0}
For the ADE type Lie algebra $\g$,
	\begin{equation}\label{dimV0}
		\dim V(2\theta)_0=\frac{\dim\g-\dim\h}{2}.
	\end{equation}
\end{lemma}

Combining \eqref{dimJU2}, \eqref{dimension of I}, \eqref{dimSymh} and \eqref{dimV0}, we have
the following.

\begin{proposition}\label{Prop:I^2_0}
If $\g$ is of ADE type, then
\begin{equation}\label{eq: I_0^2}
\dim(J_0\cap \mathcal{U}^2)=\dim (L_2)_0=\dim({\rm Sym}^2\h).	
\end{equation}
\end{proposition}

\begin{remark}
\eqref{eq: I_0^2} does not hold at least in $G_2$ case. We use the notation in \cite{Hum}. In $G_2$ case, $\theta=3\alpha_1+2\alpha_2$, $\dim \g =14$, $\dim \h =2$. By \cite[Section 22.4, Table 2]{Hum}, $\dim V(2\theta)_0=5$. By \eqref{dimension of I} and \eqref{dimSymh}, $\dim(J_0\cap \mathcal{U}^2)=\dim I^2_0=4$, but $\dim({\rm Sym}^2\h)=3$.
\end{remark}

\subsubsection{Analysis on $\mathfrak{g}$-module $L_2$}

Suppose $\g$ is a semi-simple Lie algebra of ADE type, denote by
$V(\lambda)$ the irreducible $\g$ representations with
highest weight $\lambda$.
First let us recall the following well known result:

\begin{lemma}\label{lem: Casimir scaling}
Let $C$ be the Casimir element of $\g$. Then
$C$ acts on $V(\lambda)$ as a scaling
$c_\lambda=\langle 2\delta,\lambda\rangle+\langle\lambda,\lambda\rangle$,
where $\delta$ is the half of the sum of positive roots.	
\end{lemma}
\begin{proof}
  The Casimir element $C$ lies in the center of $\mathcal{U}(\mathfrak{g})$, so it act as a scaling on the $V(\lambda)$. To calculate the scaling $c_\lambda$, just perform the $C$-action on the highest weight vector of $V(\lambda)$.
\end{proof}
Let us recall a result of Kostant (see \cite{FH}), which in fact holds for any Lie algebras:

\begin{lemma}\label{lem: Casimir}
Suppose $V=V_1\oplus V_2$ is a spliting of $\mathfrak{g}$-module,
$C$ acts on $V_1$ as scaling $c_1$ and acts on $V_2$ as scaling
different from $c_1$, then as a vector space, $V_2=Im(C-c_1\mathrm{Id})$.	
\end{lemma}

Next, we have the following lemma.

\begin{lemma}\label{lem: change root}
$\sum_{\gamma\in\Delta^+}\langle\alpha,\gamma\rangle h_\gamma
=\dfrac{1}{2}\langle 2\delta+\theta,\theta\rangle h_\alpha$, for any $h_\alpha\in\mathfrak{h}$.
\end{lemma}

\begin{proof}
Consider the $\mathfrak{g}$-module $V(\theta)=\mathfrak{g}$,
for any $\alpha\in\mathfrak{h}$, we have
\begin{align*}
C\cdot h_\alpha
&=\big(\sum_{\gamma\in\Delta^+}X_\gamma
Y_\gamma+ Y_\gamma X_\gamma+\sum_{i=1}^n h_ih_i^{\vee} \big)\cdot h_\alpha\\
&=\big(2\sum_{\gamma\in\Delta^+}X_\gamma Y_\gamma\big)\cdot h_\alpha\\
&=2\sum_{\gamma\in\Delta^+}\langle\gamma,\alpha\rangle  h_\gamma.
\end{align*}
By Lemma \ref{lem: Casimir scaling} Casimir operator $C$ acts on $V(\lambda)$ as a scaling
$c_\lambda=\langle 2\delta,\lambda\rangle+\langle \lambda,\lambda\rangle$, we have
\[
C\cdot h_\alpha=(\langle 2\delta,\theta\rangle+\langle
\theta,\theta\rangle) h_\alpha=\langle 2\delta+\theta,\theta\rangle \cdot h_\alpha.\qedhere
\]
\end{proof}

Notice that,
for complex semisimple Lie algebra, there is an isomorphism
$$
\mathfrak h \rightarrow\mathfrak h^*:
h \mapsto  K(h,-),
$$
where $K(-,-)$ is the Killing form. For Lie algebra of type ADE,
the preimage of $\alpha$ under this map is $h_\alpha$, and
thus we have $\langle \alpha,\alpha'\rangle=K(h_\alpha,h_{\alpha'})$.

\begin{lemma}\label{lem: element of W}
Let $(L_2)_0$ be the weight 0 subspace of $L_2$ in Theorem \ref{thm: Kostant}. Then
$(L_2)_0$ is spanned by vectors of the following form:
\begin{eqnarray*}
h_{\beta_1}h_{\beta_2}+\sum_{\gamma\in\Delta^{+}}\langle \beta_1,
\gamma\rangle\langle \beta_2,\gamma\rangle X_\gamma Y_{\gamma},
\end{eqnarray*}
where $\beta_1,\beta_2\in\mathfrak h^*$.
\end{lemma}

\begin{proof}
For any two distinct simple roots $\alpha_i,\alpha_j\in\Pi$, let
\begin{eqnarray*}
v_{ij}:=h_ih_j+\sum_{\gamma\in\Delta^{+}}\langle \alpha_i,
\gamma\rangle\langle \alpha_j,\gamma\rangle X_\gamma Y_{\gamma}.
\end{eqnarray*}
Since these vectors are linearly independent, and
$\#\{h_ih_j\}_{1\leq i\leq j\leq n}=\dim\mathrm{Sym}^2\mathfrak h=\dim (L_2)_0$,
to show the lemma it suffices to show $v_{ij}\in (L_2)_0$.
	
In fact, we have
\begin{align*}
& \big(C-c_{2\theta}\cdot \mathrm{Id}\big)(h_ih_j)\\
&= \big(\sum_{\gamma\in\Delta^+}X_\gamma Y_\gamma+ Y_\gamma X_\gamma
+\sum_{i=1}^n h_ih_i^{\vee}-C_{2\theta}\cdot \mathrm{Id}\big)(h_ih_j)\\
&=\big(2\sum_{\gamma\in\Delta^+}X_\gamma Y_\gamma-c_{2\theta}\cdot \mathrm{Id}\big)(h_ih_j)\\
&= 2\sum_{\gamma\in\Delta^+}\big(-2\langle\alpha_i,\gamma\rangle\langle\gamma,
\alpha_j\rangle X_\gamma Y_\gamma+\langle\gamma,\alpha_i \rangle h_jh_\gamma
+\langle \gamma,\alpha_j \rangle h_i h_\gamma\big)-\big(\langle 2\delta, 2\theta\rangle
+\langle 2\theta,2\theta\rangle\big)h_ih_j.
\end{align*}
By Lemma \ref{lem: change root}, we have
\begin{align*}
&\big(C-c_{2\theta}\cdot \mathrm{Id}\big)(h_ih_j)\\
&=-4\sum_{\gamma\in\Delta^+}\langle\alpha_i,\gamma\rangle\langle\gamma,\alpha_j\rangle
X_\gamma Y_\gamma+2\langle2\delta+\theta,\theta\rangle h_ih_j-\big(\langle 2\delta,
2\theta\rangle+\langle 2\theta,2\theta\rangle\big)h_ih_j\\
&=-4\sum_{\gamma\in\Delta^+}\langle\alpha_i,\gamma\rangle\langle\gamma,
\alpha_j\rangle X_\gamma Y_\gamma-4h_ih_j\\
&=-4v_{ij}.
\end{align*}
Set $V(2\theta)=V_1$, $L_2=V_2$ in Lemma \ref{lem: Casimir}, we know $v_{ij}\in L_2$,
and the weight of $v_{ij}$ is $0$, thus $v_{ij}\in(L_2)_0$.
\end{proof}
By Proposition \ref{Prop:I^2_0} and Lemma \ref{lem: element of W} , we have the following:

\begin{corollary}\label{cor: (L_2)_0}
There is an isomorphism of vector spaces:
\begin{eqnarray*}
\Psi:\mathrm{Sym}^2\mathfrak{h}&\rightarrow &(L_2)_0,\\
h_{\beta_1}h_{\beta_2}&\mapsto & h_{\beta_1}h_{\beta_2}
+\sum_{\gamma\in\Delta^{+}}\langle \beta_1,\gamma\rangle\langle \beta_2,\gamma\rangle X_\gamma Y_{\gamma}.
\end{eqnarray*}
\end{corollary}

Considering the Killing form $K(-,-):\mathfrak h\otimes\mathfrak h\rightarrow \mathbb C$,
since it is a symmetric bilinear form, it induces a linear functional
$K:\mathrm{Sym}^2\mathfrak h\rightarrow\mathbb{C}$.
Recall that the subrepresentation $L_2$ has a further $\mathfrak g$-module decomposition
\[
L_2=W_2\oplus V(0),
\]
where $W_2=\oplus_iV(\lambda_i)$ with $c_{\lambda_i}\neq 0$
and $V(\lambda_i)$ is the highest weight representation (see Facts \ref{fact 1} and \ref{fact 2}).
We then have the following:

\begin{lemma}\label{lem: v in W_2}
Suppose $v\in (W_2)_0$, then $K\circ \Psi^{-1}(v)=0$.
\end{lemma}

\begin{proof}
Set $V_1=V(0)$, $V_2=W_2$ in Lemma \ref{lem: Casimir},
we know that $v=C\cdot u$ for some $u\in(L_2)_0$.
By Lemma \ref{lem: element of W}, without loss of generality,
assume $u=h_{\beta_1}h_{\beta_2}+\sum_{\gamma\in\Delta^{+}}\langle
\beta_1,\gamma\rangle\langle \beta_2,\gamma\rangle X_\gamma Y_{\gamma}$,
then
\begin{align*}
v&=2\langle 2\delta+\theta,\theta\rangle h_{\beta_1}
h_{\beta_2}-2\sum_{\gamma\in\Delta^{+}}\langle \beta_1,
\gamma\rangle\langle \beta_2,\gamma\rangle h_\gamma^2+
\textup{terms of }X_\gamma Y_\gamma,\\
\Psi^{-1}(v)&=2\langle 2\delta+\theta,\theta\rangle h_{\beta_1}
h_{\beta_2}-2\sum_{\gamma\in\Delta^{+}}\langle \beta_1,
\gamma\rangle\langle \beta_2,\gamma\rangle h_\gamma^2,\\
K\circ\Psi^{-1}(v)&=2\big(\langle 2\delta+\theta,\theta\rangle
\langle \beta_1,\beta_2\rangle-\sum_{\gamma\in\Delta^{+}}
\langle \beta_1,\gamma\rangle\langle \beta_2,
\gamma\rangle\langle\gamma,\gamma\rangle\big).
\end{align*}
Now to show the lemma it suffices to check that
\begin{eqnarray*}
2\sum_{\gamma\in\Delta^{+}}\langle \beta_1,\gamma\rangle\langle
\beta_2,\gamma\rangle=\langle 2\delta+\theta,\theta\rangle \langle \beta_1,\beta_2\rangle.
\end{eqnarray*}
In fact, by Lemma \ref{lem: change root}, we have
\begin{align*}
2\sum_{\gamma\in\Delta^{+}}\langle \beta_1,\gamma\rangle\langle \beta_2,\gamma\rangle
=2\big<\sum_{\gamma\in\Delta^{+}}\langle\beta_1,\gamma\rangle \gamma,\beta_2\big>
=\big<\langle 2\delta+\theta,\theta \rangle \beta_1,\beta_2\big>
=\langle 2\delta+\theta,\theta\rangle \langle \beta_1, \beta_2\rangle.
\end{align*}	
The lemma follows.
\end{proof}

\begin{corollary}
$\Psi|_{\ker K}:\ker K\rightarrow (W_2)_0$ is an isomorphism of vector spaces.
\end{corollary}

\begin{proof}
Since $K$ is surjective, $\dim\ker K=\dim \mathrm{Sym}^2\mathfrak h-1$.
On the other hand, $(L_2)_0=(W_2)_0\oplus V(0)$, and $\dim V(0)=1$, so by
Proposition \ref{Prop:I^2_0}, $\dim(W_2)_0=\dim(L_2)_0-1=\dim\ker K$.
Then by Lemma \ref{lem: v in W_2}, $\Psi^{-1}\big((W_2)_0\big)\subseteq\ker K$,
and thus $\Psi^{-1}|_{\ker K}$ is an isomorphism.
\end{proof}

\subsubsection{D-module of graded traces in the type DE cases}
Now, we fix the notation $J_0^2:=J_0\cap \mathcal{U}^2$. We consider $\mathfrak{g}$ is Lie algebra of DE type.
Recall that $\beta$ is the symmetrization map given in Example \ref{example: quantization}.

\begin{lemma}\label{lem: beta(v)}
If $v\in (W_2)_0$, then $\beta(v)\in J^2_0$.
\end{lemma}

\begin{proof}
Suppose 	$v\in (W_2)_0$. Recall that $W_2=\oplus_{i=1}^k V(\lambda_i)$, and thus
$v=v_1+\cdots+v_k$ such that $v_i\in V(\lambda_i)$.

Since $V(\lambda_i)$ is irreducible, there exists
an element $T_i\in \mathcal U(\g)$ such that $v_i=T_i\cdot v_{-\lambda_i}$.
Thus $v=\sum_{i=1}^k T_i(v_{-\lambda_i})$, and
since by Theorem \ref{Garfinkle lemma1 Dn}, $\beta(v_{-\lambda_i})\in J_0^2$, we have
\[
\beta(v)=\sum_{i=1}^k\beta\circ T_i(v_{-\lambda_i})=\sum_{i=1}^k T_i\circ \beta(v_{-\lambda_i})\in J_0^2.\qedhere
\]
\end{proof}

\begin{proposition}\label{prop:basis of J20 DE}
For $\mathfrak{g}$ is of DE type, $J_0^2$ is spanned by vectors of the following form
\begin{equation}
    h_{\alpha_i}h_{\alpha_j}+\dfrac{1}{2}\sum_{\gamma\in\Delta^{+}}\langle
\alpha_i,\gamma\rangle\langle \alpha_j,\gamma\rangle
(X_\gamma Y_{\gamma}+Y_\gamma X_{\gamma})+\dfrac{1}{4}|\Gamma|\<\alpha_i, \alpha_j\>,
\end{equation}
where $\{\alpha_1, \ldots, \alpha_n\}$ is the set of simple root of $\mathfrak{g}$ and $i, j$ run 
from $1$ to $n$.
\end{proposition}
\begin{proof}
By Corollary \ref{cor: (L_2)_0}, we can choose a basis of
$\mathrm{Sym}^2\mathfrak h$: $\phi_1,\cdots,\phi_l$ for $l=
\dim \mathrm{Sym}^2\mathfrak h$,
such that $\Psi(\phi_1),\cdots,\Psi(\phi_{l-1})$ are
a basis of $(W_2)_0$ and $\Psi(\phi_l)$ is a basis of $V(0)$ such that
$\beta(\Psi(\phi_l))=C$.
	
By Lemma \ref{lem: beta(v)}, for $i=1,\cdots,l-1$, $\beta\big(\Psi(\phi_i)\big)\in J_0^2$, which a linear combination of vectors
\begin{eqnarray*}
  h_{\alpha_i}h_{\alpha_j}+\dfrac{1}{2}\sum_{\gamma\in\Delta^{+}}\langle
\alpha_i,\gamma\rangle\langle \alpha_j,\gamma\rangle
(X_\gamma Y_{\gamma}+Y_\gamma X_{\gamma}).
\end{eqnarray*}
 By Lemma \ref{lem: v in W_2}, $K(\phi_i)=0$ ,
which means that, for $i=1,\cdots,l-1$, $\beta\big(\Psi(\phi_i)\big)$ is also a linear combination of vectors
\begin{eqnarray}\label{basis of J02}
    h_{\alpha_i}h_{\alpha_j}+\dfrac{1}{2}\sum_{\gamma\in\Delta^{+}}\langle
\alpha_i,\gamma\rangle\langle \alpha_j,\gamma\rangle
(X_\gamma Y_{\gamma}+Y_\gamma X_{\gamma})+\dfrac{1}{4}|\Gamma|\<\alpha_i, \alpha_j\>.
\end{eqnarray}

Now for $\beta(\Psi(\phi_l))=C$, it gives the following element in $J_0^2$ 
by Theorem \ref{Garfinkle lemma1 Dn}
\begin{eqnarray*}
C-c_\lambda=\sum_{\gamma\in\Delta^+}(X_\gamma Y_\gamma+
Y_\gamma X_\gamma)+\sum_{i=1}^n h_ih_i^{\vee}-c_\lambda .
\end{eqnarray*}	
Combining
the table in \S 2.1 with Theorem \ref{Garfinkle lemma1 Dn},
we get $\dfrac{c_\lambda}{n}=-\dfrac{|\Gamma|}{4}$,
where $n=rank (\mathfrak h)$, and then by a direct calculation, 
\begin{equation*}
 C-c_\lambda=\sum_{i=1}^n h_ih_i^{\vee}+\sum_{\gamma\in\Delta^+}(X_\gamma Y_\gamma+
Y_\gamma X_\gamma)+\dfrac{n}{4}|\Gamma|,
\end{equation*}
which is also a linear combination of vectors in \eqref{basis of J02}.

By Proposition \ref{Prop:I^2_0}, we know $J_0^2$ is 
spanned by  $\beta\big(\Psi(\phi_i)\big)$, 
$i=1,\cdots,l-1$ and $\beta(\Psi(\phi_l))-c_\lambda$. 
Denote by $\widetilde{V}$ the subspace of $\mathcal{U}_0^2$ 
spanned by the vectors in the form \eqref{basis of J02}, 
then  $J_0^2\subset \widetilde{V}$. Furthermore, 
$\dim J_0^2=\dim \widetilde{V}=\dim \mathrm{Sym}^2\mathfrak h$, 
so $J_0^2= \widetilde{V}$. The proof is completed.
\end{proof}

\begin{theorem}\label{thm: D-module of DE}
    For $\mathfrak{g}$ is of DE type, $Q(\A)\cong S_{reg}\otimes (\mathbb{C}\oplus \mathfrak{h})$ with the following relation
    \begin{equation}\label{eq:quantum product DE}
        h_{\alpha}\cdot (1\otimes h_{\alpha'})=-\frac{\hbar^2}{4}|\Gamma|\langle \alpha ,\alpha'\rangle
+\sum_{\gamma\in\Delta^+}\frac{\hbar}{2}\langle\alpha,\gamma\rangle\langle\alpha',
\gamma\rangle\dfrac{1+q^\gamma}{1-q^\gamma}  h_\gamma
    \end{equation}
\end{theorem}

\begin{proof}
By \eqref{eq:Q(A)}, $Q(\A)\cong S_{reg}\otimes 
Rees(\mathcal{U}(\mathfrak h)/\mathcal{J}_{\mathfrak{h}})$. 
And $\mathcal{J}_{\mathfrak{h}}$ is generated by $\kappa(1\otimes J^2_0)$ 
due to Proposition \ref{lem: Q(A) structure}. 
By Proposition \ref{prop:basis of J20 DE} and the definition of $\kappa$,  
$\kappa(1\otimes J^2_0)$ is spanned by 
\begin{align*}
& \kappa\Big( h_{\alpha_i}h_{\alpha_j}+\dfrac{1}{2}\sum_{\gamma\in\Delta^{+}}\langle
\alpha_i,\gamma\rangle\langle \alpha_j,\gamma\rangle
(X_\gamma Y_{\gamma}+Y_\gamma X_{\gamma})+\dfrac{1}{4}|\Gamma|\<\alpha_i, \alpha_j\>\Big)\\
=& h_{\alpha_i}h_{\alpha_j}+\frac{1}{4}|\Gamma|\langle \alpha_i ,\alpha_j\rangle
-\sum_{\gamma\in\Delta^+}\frac{1}{2}\langle\alpha_i,\gamma\rangle\langle\alpha_j,
\gamma\rangle\dfrac{1+q^\gamma}{1-q^\gamma}  h_\gamma.
\end{align*}
This completes the proof.
\end{proof}

\subsubsection{D-module of graded traces in the type A case}

Now we consider the case $\mathfrak{g}$ is a Lie algebra of type A.
Recall the bilinear map $(-,-)_c:\mathfrak h\otimes\mathfrak h\rightarrow \mathfrak h$ defined in Definition \ref{def: bilinear c}.
 Since it is a symmetric, it induces a linear map
 $K_c:\mathrm{Sym}^2\mathfrak h\rightarrow\mathfrak h$.
 Recall that from
 \S \ref{subsect:JosephintypeA}, the subrepresentation $L_2$ has a further $\mathfrak g$-module decomposition
\[
(L_2)_0=V(\theta+\alpha)_0\oplus V(\theta)_0\oplus V(0)_0.
\]
Then we have the following lemma:

\begin{lemma}\label{lem: v in V_2}
Suppose $v\in V(\theta+\alpha)_0\oplus V(0)_0$ in 
Proposition \ref{structure of J}, then $K_c\circ \Psi^{-1}(v)=0$.
\end{lemma}

\begin{proof}
Set $V(\theta)=V_1$, $V(\theta+\alpha)\oplus V(0)=V_2$ in Lemma \ref{lem: Casimir}, we have
\begin{eqnarray*}
v=\big(C-(\langle 2\delta,\theta\rangle+\langle \theta,\theta\rangle)\mathrm{Id}\big)(u),
\end{eqnarray*}
for some $u\in L_2$.
	
By Lemma \ref{lem: element of W},  without loss of generality,
assume \[u=h_{\beta_1}h_{\beta_2}
+\sum_{\gamma\in\Delta^{+}}\langle \beta_1,
\gamma\rangle\langle \beta_2,\gamma\rangle X_\gamma Y_{\gamma},\]
then we have:
\begin{align*}
\Psi^{-1}(v)&=2\langle 2\delta+\theta,\theta\rangle h_{\beta_1}h_{\beta_2}
-2\sum_{\gamma\in\Delta^{+}}\langle \beta_1,\gamma\rangle\langle \beta_2,\gamma\rangle h_\gamma^2,\\
K_c\circ\Psi^{-1}(v)&=2\big(\langle 2\delta+\theta,\theta\rangle K_c(\beta_1,\beta_2)
-\sum_{\gamma\in\Delta^{+}}\langle \beta_1,\gamma\rangle\langle \beta_2,\gamma\rangle(\gamma,\gamma)_c\big).
\end{align*}
We only need to check the following identity on a basis of $\mathrm{Sym}^2\mathfrak h$:
\begin{eqnarray*}
\langle\delta,\theta\rangle(\beta_1,\beta_2)_c
=\sum_{\gamma\in\Delta^{+}}\langle \beta_1,\gamma\rangle\langle \beta_2,\gamma\rangle(\gamma,\gamma)_c.
\end{eqnarray*}
In fact, choose a basis containing the following vectors $h_{\alpha_i}h_{\alpha_j}$ ($j-i\geq 2$),
$h_{\alpha_i}h_{\alpha_{i+1}}$ ($1\leq i\leq n-1$) and
$h_{\alpha_i}^2$ ($1\leq i\leq n$). It is straightforward to check 
the above identity holds on these vectors. We leave it to the readers.
\end{proof}	

Now consider a subspace $J_\theta\subseteq J_0^2$, which is defined by
\[
J_\theta:=\mathcal{U}(\g)\cdot\big(\beta(v_{-\theta})-y\big)\cap J_0^2.
\]
We have:

\begin{lemma}\label{lem: beta(v)-y}
For $v\in V(\theta)_0$, there is a unique $y_v\in
\mathcal{U}^1(\g)\cap \mathcal{U}_0(\g)\cong \mathfrak h$
such that $\beta(v)+y_v$ is in $J_\theta.$
\end{lemma}

\begin{proof}
Since $V(\theta)$ is irreducible, there exists an element
$L\in \mathcal{U}(\g)$ such that $v=L(v_{-\theta})$.
Hence $\beta(v)=\beta\circ L(v_{-\theta})=L(\beta(v_{-\theta}))$.
By Lemma \ref{Vtheta}, $L(\beta(v_{-\theta})-y))\in J_\theta$.
Denote by $y_v:=-L(y)$, then $\beta(v)+y_v\in J_\theta$.
	
Assume there is another $\tilde{y}_v$ such that $\beta(v)+\tilde{y}_v\in J_\theta$,
then $y_v-\tilde{y}_v\in J_0^1$. However, from the proof of Lemma \ref{lemma:dimJU2},
$J_0^1=J_0\cap\mathcal{U}^1\subseteq J\cap \mathcal{U}^1=\{0\}$, we get $\tilde{y}_v=y_v$.
\end{proof}

By Lemma \ref{lem: beta(v)-y}, we get a symmetric bilinear map
\begin{eqnarray*}
	K_Q:\Psi^{-1}(V(\theta)_0)\rightarrow \mathfrak h:
	\Psi^{-1}(v) \mapsto  y_v,
\end{eqnarray*}
and then we have:

\begin{lemma}\label{lem: K_Q}
Restricted on $\Psi^{-1}\big(V(\theta)_0\big)$, $K_Q=-\left(\dfrac{n+1}{2}+z\right)K_c$.	
\end{lemma}

\begin{proof}
Recall that we have obtained $\beta(v_{-\theta})-y$ in \eqref{vinU}. For $k=2,3,\cdots,n$, by applying	
$ad_{X_{12\cdots k-1}}\circ ad_{X_{k,k+1,\cdots n}}$ to \eqref{vinU}, we obtain
\begin{align*}
&-(n+1)(h_k+h_{k+1}+\cdots+h_n)\cdot(h-1+h_2+\cdots+h_{k-1})\\
&+\big(\sum_{l=1}^n Y_\theta(2l-1-n)h_l\big)(h_1+h_2+\cdots+h_{k-1})
+\textup{terms of }X_\gamma Y_\gamma\\
&-\dfrac{n-1}{2}(n+1+2z)(h_1+h_2+\cdots+h_{k-1}) \in J_0^2.
\end{align*}
Then,
\begin{align*}
&K_Q\bigg(-(n+1)(h_k+h_{k+1}+\cdots+h_n)\cdot(h-1+h_2+\cdots+h_{k-1})\\
&+\big(\sum_{l=1}^n Y_\theta(2l-1-n)h_l\big)(h_1+h_2+\cdots+h_{k-1})\bigg)\\
=&-\dfrac{n-1}{2}(n+1+2z)(h_1+h_2+\cdots+h_{k-1}).
\end{align*}
On the other hand,
\begin{align*}
&K_c\bigg(-(n+1)(h_k+h_{k+1}+\cdots+h_n)\cdot(h-1+h_2+\cdots+h_{k-1})\\
&+\big(\sum_{l=1}^n Y_\theta(2l-1-n)h_l\big)(h_1+h_2+\cdots+h_{k-1})\bigg)\\
=&(n-1)(\omega_1^*+\omega_{k-1}^*-\omega_k^*)\\
=&(n-1)(h_1+h_2+\cdots+h_{k-1}).
\end{align*}
Comparing the above two terms,
we see that they differ up to a scaler $-\left(\frac{n+1}{2}+z\right)$.

Now for $k=1$, denote $ad_{X_{12\cdots k-1}}=\mathrm{Id}$, by applying	
$ad_{X_{12\cdots n}}$ to \eqref{vinU}, we obtain the same result.

Observe that
the following $n$ vectors in $\mathrm{Sym}^2\mathfrak h$ are linearly independent
\begin{eqnarray*}
	&&v_k:=-(n+1)(h_k+h_{k+1}+\cdots+h_n)(h_1+h_2+\cdots+h_{k-1}),\ k=2,3,\cdots,n\\
	&&v_1:=(h_1+h_2+\cdots+h_n)\big(\sum_{l=1}^n (2l-1-n)h_l\big),
\end{eqnarray*}
and $\dim \Psi^{-1}(V(\theta)_0)=\dim(V(\theta)_0)=n$.
Since $V(\theta)_0\cong\mathfrak h$ and
$v_1,v_2,\cdots,v_n$ is a basis of $\Psi^{-1}(V(\theta)_0)$,
we have $K_Q=-\left(\dfrac{n+1}{2}+z\right)K_c$ on
the whole vector space $\Psi^{-1}\big(V(\theta)_0\big)$.
\end{proof}

\begin{proposition}
 For $\mathfrak{g}$ is of $A_n$ type, $J_0^2$ is spanned by vectors of the following form
  \begin{equation}
    h_{\alpha_i}h_{\alpha_j}+\dfrac{1}{2}\sum_{\gamma\in\Delta^{+}}\langle
\alpha_i,\gamma\rangle\langle \alpha_j,\gamma\rangle
(X_\gamma Y_{\gamma}+Y_\gamma X_{\gamma})-(\dfrac{n+1}{2}+z)(\alpha_i, \alpha_j)_c+\dfrac{1}{4}|\Gamma|\<\alpha_i, \alpha_j\>,
  \end{equation}
  where $\{\alpha_1, \ldots, \alpha_n\}$ is the set of simple roots of $\mathfrak{g}$ and $i, j$ run over $1$ to $n$.
\end{proposition}
\begin{proof}The proof is similar as that of Proposition \ref{prop:basis of J20 DE}.
By Corollary \ref{cor: (L_2)_0}, we can choose a basis of
$\mathrm{Sym}^2\mathfrak h$: $\phi_1,\cdots,\phi_l$, where
$l=\dim (n^2-n-2)/2$, and $\psi_1,\cdots,\psi_n$,
such that $\Psi(\phi_1),\cdots,\Psi(\phi_{l-1})$ are a basis of
$V(\theta+\alpha)_0$,
$\Psi(\phi_l)$ is a basis of $V(0)_0$ such that $\beta(\Psi(\phi))=C$,
and $\Psi(\psi_1),\cdots,\Psi(\psi_n)$
is a basis of $V(\theta)_0$.

First, by Theorem \ref{Prop:GaJosephAn} and Lemma \ref{lem: beta(v)}, 
for $i=1, 2, \ldots, l-1$, $\beta(\Psi(\phi_i))\in J_0^2$ is a linear combination of the following vectors
\begin{eqnarray*}
 h_{\alpha_i}h_{\alpha_j}+\dfrac{1}{2}\sum_{\gamma\in\Delta^{+}}\langle
\alpha_i,\gamma\rangle\langle \alpha_j,\gamma\rangle
(X_\gamma Y_{\gamma}+Y_\gamma X_{\gamma}),
\end{eqnarray*}
By Lemmas \ref{lem: v in W_2} and \ref{lem: v in V_2},  
since $ \Psi(\phi_i)\in V(\theta+\alpha)_0$, we have
then $K(\phi_i)=0$ and $K_c(\phi_i)=0$. Thus
$\beta(\Psi(\phi_i))$ is also a linear combination of the following vectors
\begin{equation}\label{basis of J02 A}
   h_{\alpha_i}h_{\alpha_j}+\dfrac{1}{2}\sum_{\gamma\in\Delta^{+}}\langle
\alpha_i,\gamma\rangle\langle \alpha_j,\gamma\rangle
(X_\gamma Y_{\gamma}+Y_\gamma X_{\gamma})-\left(\dfrac{n+1}{2}+z\right)
(\alpha_i, \alpha_j)_c+\dfrac{1}{4}|\Gamma|\<\alpha_i, \alpha_j\>.
\end{equation}

Second, by Lemma \ref{lem: beta(v)-y}, we know that
$\beta(\Psi(\psi_j))+y_{\Psi(\psi_j)}\in J_{0}^2$. More explicitly, 
it is a linear combination of the following vectors
\begin{equation*}
  h_{\alpha_i}h_{\alpha_j}+\dfrac{1}{2}\sum_{\gamma\in\Delta^{+}}\langle
\alpha_i,\gamma\rangle\langle \alpha_j,\gamma\rangle
(X_\gamma Y_{\gamma}+Y_\gamma X_{\gamma})+K_Q(h_{\alpha_i}h_{\alpha_j}).
\end{equation*}
 By Lemmas \ref{lem: v in W_2} and \ref{lem: K_Q}, 
 since $ \psi_j\in V(\theta)_0$, $K(\psi_j)=0$ and 
$K_Q( \psi_j)=-\left(\dfrac{n+1}{2}+z\right)K_c(\psi_j)$.
Thus $\beta(\Psi(\psi_j))+y_{\Psi(\psi_j)}$ is also a linear 
combination of vectors of the form \eqref{basis of J02 A}.
	
Finally, the argument on 
$\beta(\Psi(\phi_l))-c_\lambda$ is similar to that in Proposition 
\ref{prop:basis of J20 DE}. And the remaining part of the proof is by 
the same dimension counting argument as in the proof of 
Proposition \ref{prop:basis of J20 DE}.
\end{proof}

\begin{theorem}\label{thm:quantum action An}
For $\mathfrak{g}$ is of $A_n$ type, $Q(\A)\cong 
S_{reg}\otimes (\mathbb{C}\oplus \mathfrak{h})$ with the following relation 
\begin{equation}\label{eq:quantum action An}
h_{\alpha}\cdot (1\otimes h_{\alpha'})=-\frac{\hbar^2}{4}|\Gamma|\langle \alpha ,\alpha'\rangle
+(\dfrac{n+1}{2}+z)\hbar(\alpha, \alpha')_c
+\sum_{\gamma\in\Delta^+}\frac{\hbar}{2}\langle\alpha,\gamma\rangle\langle\alpha',
\gamma\rangle\dfrac{1+q^\gamma}{1-q^\gamma}  h_\gamma.
\end{equation}
\end{theorem}

\begin{proof}
Similar to the proof of Theorem \ref{thm: D-module of DE}.
\end{proof}

\subsection{Proof of Theorems \ref{maintheorem0} and \ref{maintheorem}}\label{sect:proofofmainthm}

The quantum Hikita conjecture, proposed
by Kamnitzer, McBreen and Proudfoot in \cite{KMP},
is stated as follows.

\begin{conjecture}[The quantum Hikita conjecture]
Suppose $X$ and $X^!$ are symplectic dual to each other.
Then there is an isomorphism
$M_{\mathrm{reg}}(X)\cong Q_{\mathrm{reg}}(X^!)$ of graded modules over
$R_{\mathrm{reg}}(X)\cong E_{\mathrm{reg}}(X^!)$ sending $1\in M_{\mathrm{reg}}(X)$
to $1\in Q_{\mathrm{reg}}(X^!)$.
\end{conjecture}

The Kleinian singularities
and the minimal nilpotent orbits are expected to
be symplectic dual to each other.
Recall that by
Remark \ref{rem:quantumDmoduleinADEcase}, in the ADE singularities case, the quantum D-module
is nothing but the equivariant quantum cohomology, and thus the
above conjecture is exactly Theorem \ref{maintheorem0} of the current paper.

\begin{proof}[Proof of Theorem \ref{maintheorem0}]
In the $A_n$ case,  let
\begin{equation}\label{iso:typeAcase}
t_1\mapsto-\dfrac{z\hbar}{n+1},\ t_2\mapsto
\dfrac{z+n+1}{n+1}\hbar,\ e_\alpha\mapsto h_\alpha,\ e_{\alpha'}\mapsto h_{\alpha'}.
\end{equation}
By comparing \eqref{qcohAn} with \eqref{eq:quantum action An}, we get the isomorphism.
In the DE case, let
\begin{eqnarray*}
	\hbar \mapsto 2t,\  e_\alpha\mapsto h_\alpha,\ e_{\alpha'}\mapsto h_{\alpha'}.
\end{eqnarray*}
By comparing \eqref{eq: quantum coh DE} with \eqref{eq:quantum product DE}, we get the
isomorphism.
\end{proof}

\begin{proof}[Proof of Theorem \ref{maintheorem}]
Let $q=0$. Then
the quantum D-modules
reduces to the equivariant cohomology
and the D-module of graded traces
$Q(\A)$ reduces to $B(\A)$ (see Remarks \ref{rem:redofequivquantcoh}
and \ref{re:reductionofQ(A)}).
Also, by Lemma \ref{lemma:B-alg str}, $B(\A)$ is generated in degree 2,
and thus by \cite[Remark 5.3]{KMP}, Theorem \ref{maintheorem0} reduces to Theorem \ref{maintheorem}.
\end{proof}

\begin{remark}\label{rem:JosephidealintypeA}
In the above two theorems, for type A Lie algebras,
we have assumed the Joseph ideals $J^z$ are parametrized
by the formal parameter $z$, which is also called
the {\it K\"ahler parameter}.
If we specify
$z=-(n+1)/2$ in the quantization of the minimal nilpotent orbits
in the $A_n$ Lie algebra, then for all ADE singularities,
the isomorphisms in the above theorems have a uniform form
$$
\mathrm{QH}^\bullet_{\mathbb C^\times}(\widetilde{\mathbb C^2/\Gamma})
\cong Q\big(\mathscr A[\overline{\mathcal O}_{min}]\big)
\quad\mbox{and}\quad
\mathrm{H}^\bullet_{\mathbb C^\times}(\widetilde{\mathbb C^2/\Gamma})
\cong B\big(\mathscr A[\overline{\mathcal O}_{min}]\big).
$$
\end{remark}

\section{Generalization to BCFG type singularities}\label{sect:BCFG}

In this section, we generalize the isomorphisms in Theorems \ref{maintheorem0}
and \ref{maintheorem}
to the cases of BCFG type singularities.

First recall that for the ADE type Lie algebras, the minimal nilpotent orbits are the same
as the minimal special nilpotent orbits, but for the other types of Lie
algebras, they are different.
The Lusztig-Spaltenstein duality says that for BCFG type Lie algebras,
the minimal {\it special} orbits are dual to the subregular nilpotent orbits.
Also, recall that for BC type Lie algebras, they are Langlands dual to each other,
and for simple Lie algebras of the rest types, they are Langlands self-dual.

%We shall \textcolor{red}{mention} that the normalization of the minimal
%special nilpotent orbit in one type Lie algebra is \textcolor{red}{symplectic dual}
%to the intersection of a Slodowy slice to the subregular nilpotent orbit
%with the nilpotent cone
%in its Langlands dual, where the latter
%is called a {\it simple singularity of BCFG type}
%by Slodowy \cite{S}.

\subsection{Quantization of the minimal special nilpotent orbits}\label{subsect:quantms}

In this subsection, we denote by $\mathcal O_{ms}$ the minimal special
nilpotent orbit in a Lie algebra of BCFG type.
A theorem of Panyushev in \cite{P} says that
the normalization $\tilde{\mathcal{O}}_{ms}$ of the closure of
$\mathcal{O}$ has symplectic singularities, and hence it makes
sense to study the symplectic duality for $\tilde{\mathcal O}_{ms}$.

Let $\mathfrak g$ be a simple Lie algebra. Automorphisms of the Dynkin diagram
of $\mathfrak g$ induce automorphisms of the root vectors of $\mathfrak g$
and hence of $\mathfrak g$.
Dynkin diagrams with nontrivial automorphisms are only those of
$A_{n} (n>1)$, $D_{n} (n>2)$ and $E_6$. They are given by (c.f. \cite[\S7.9]{Kac}):
\begin{enumerate}
\item[$-$] For $A_n$, the automorphism group is $\mathbb Z_2$
with the nontrivial automorphism given by $\alpha_i\mapsto
\alpha_{n+1-i}$;
\item[$-$] For $D_{4}$, there are two types
of automorphisms, the first automorphism group is $\mathfrak S_3$
which is generated by the permutations
of $\alpha_1,\alpha_3$ and $\alpha_4$, while the second automorphism group is
$\mathbb Z_2$ given by permuting
$\alpha_3$ and $\alpha_4$. For $D_{n+1}$, $n\ge 2$ and $n\ne 3$,
the automorphism group is $\mathbb Z_2$, which is
given by permuting $\alpha_{n}$ and $\alpha_{n+1}$.
\item[$-$] For $E_6$, the automorphism is $\mathbb Z_2$,
given by permuting $\alpha_1$ and $\alpha_6$,
$\alpha_3$ and $\alpha_5$, with $\alpha_2$ and $\alpha_4$ fixed.
\end{enumerate}
Taking the quotients of the Dynkin diagrams by the above
group actions, we obtain the Dynkin diagrams
of Lie algebras of the other type, which is summarized in the following table:
\begin{center}
\begin{tabular}{p{3.5cm}p{3cm}p{3cm}}
\hline
Original diagram &Automorphism & Quotient diagram\\
\hline\hline
$D_{n+1}$ & $\mathbb Z_2$ & $B_n$ \\
$A_{2n-1}$&$\mathbb Z_2$ &$ C_n $\\
$A_{2n}$&$\mathbb Z_2$ &$ A_n $\\
$E_6$&$\mathbb Z_2$& $F_4$\\
$D_4$&$\mathfrak S_3$&$G_2$ \\
 \hline
\end{tabular}
\end{center}
The above actions lift to automorphisms $\nu$ of $\mathfrak g$ defined by
$\nu(X_\alpha)=X_{\nu(\alpha)}$ and $\nu(Y_\alpha)
=Y_{\nu(\alpha)}$. The following theorem is due to Brylinski and Kostant:

\begin{theorem}[{\cite{BK}}]
The minimal nilpotent orbit $\mathcal O_{min}$ in $D_{n+1}$, $A_{2n-1}$ and $E_6$
is the double cover of the minimal special nilpotent orbit $\mathcal O_{ms}$
of type $B_n$, $C_{2n}$ and $F_4$ respectively,
while the minimal nilpotent orbit $\mathcal O_{min}$ in
$D_4$ is the 6-fold cover of minimal special nilpotent orbit $\mathcal O_{ms}$ in $G_2$
with the deck transformation $\mathfrak S_3$.
\end{theorem}

The result is summarized by the following table:
\begin{center}
\begin{tabular}{p{3cm}p{3cm}p{4cm}}
\hline
Type of $\mathcal O_{ms}$ & Covering &  Deck transformation\\
\hline\hline
$\mathcal O_{ms}(B_n)$ & $\mathcal O_{min}(D_{n+1})$ & $\mathbb Z_2$\\
$\mathcal O_{ms}(C_n)$ & $\mathcal O_{min}(A_{2n-1})$ & $\mathbb Z_2$\\
$\mathcal O_{ms}(F_4)$ & $\mathcal O_{min}(E_6)$ & $\mathbb Z_2$\\
$\mathcal O_{ms}(G_2)$ & $\mathcal O_{min}(D_4)$ & $\mathfrak S_3$\\
 \hline
\end{tabular}
\end{center}
In other words, we have the following isomorphisms:
$$\begin{array}{ll}
\mathbb C[\tilde{\mathcal{O}}_{ms}(B_n)]=
\mathbb C[\mathcal{O}_{min}(D_{n+1})]^{\mathbb Z_2},
& \mathbb C[\tilde{\mathcal{O}}_{ms}(C_{n})]=
\mathbb C[\mathcal{O}_{min}(A_{2n-1})]^{\mathbb Z_2},\\[2mm]
\mathbb C[\tilde{\mathcal{O}}_{ms}(F_4)]= \mathbb C[\mathcal{O}_{min}(E_6)]^{\mathbb Z_2},
& \mathbb C[\tilde{\mathcal{O}}_{ms}(G_2)]= \mathbb C[\mathcal{O}_{min}(D_4)]^{\mathfrak S_3}.
\end{array}
$$
In \cite{Huang}, Huang studied the quantization of the minimal special nilpotent orbits
in these Lie algebras.
For Lie algebras of DE type, the Joseph ideals
are maximal and are stable under the
actions of $\mathbb Z_2$ or $\mathfrak S_3$.
For Lie algebras of type A, the Joseph
ideals $J^z$ are not unique, but we have the following.

\begin{proposition}
Suppose $\mathfrak g$ is the Lie algebra $A_{2n-1}$. Let $J^z$ be the Joseph
ideals parameterized by $z\in\mathbb C$.
Then there is a unique Joseph ideal which is stable
under the action of $\mathbb Z_2$.
More precisely, such an ideal is $J^{-n}$ in Theorem \ref{Prop:GaJosephAn}.
\end{proposition}

\begin{proof}
For the
$A_{2n-1}$ Lie algebra, recall from Theorem \ref{Prop:GaJosephAn}
that its Joseph ideals are generated by three types of elements
$\beta(v_0)$,
\eqref{vinU} and \eqref{vtheta0An}.
It is direct to check that, under the $\mathbb Z_2$-action,
$\beta(v_0)$ and \eqref{vtheta0An} 	
are mapped to elements in $J^z$.
Applying the nontrivial element of $\mathbb Z_2$ to
\eqref{vinU}, which is now
\begin{align}\label{Z2UAn0}
&-2n\big(Y_{2\cdots 2n-1}Y_1+Y_{3\cdots 2n-1}Y_{12}+\cdots
+Y_nY_{1\cdots 2n-1}\big)\notag\\
&+Y_\theta\left(\sum_{k=1}^{2n-1} (2k-2n)h_k+(2-2n)z\right),
\end{align}
we obtain
\begin{align}\label{Z2UAn}
& 2n\big(Y_{2\cdots 2n-1}Y_1+Y_{3\cdots 2n-1}Y_{12}+\cdots
+Y_nY_{1\cdots 2n-1}-(2n-2)Y_\theta\big)\notag\\
&-Y_\theta\left(\sum_{k=1}^{2n-1} (2k-2n)h_k-(2-2n)z\right).
\end{align}
Comparing
\eqref{Z2UAn0} with \eqref{Z2UAn} we see
that \eqref{Z2UAn} is an element in $J^z$ if and only if $z=-n$.
\end{proof}

\begin{convention}[Compare with Remark \ref{rem:JosephidealintypeA}]
{\it From now on},
we take the Joseph ideal for $A_{2n-1}$
to be $J^{-n}$.
And on the dual side, the equivariant cohomology of the minimal
resolution of the $A_{2n-1}$ singularity is the $\mathbb C^\times$-equivariant
cohomology.
\end{convention}

Now we study the D-module of graded
traces $Q(\mathscr A)$ for these nilpotent orbits.
Let $\Lambda$ be the index of simple roots
of an ADE type Lie algebra $\mathfrak g$.
Denote by $\bar{\mathfrak g}$ the Lie algebra
corresponding to the Dynkin diagram obtained
from the one of $\mathfrak g$ by modulo its automorphism
described above. Denote the automorphism group by $\Phi$.
The simple roots of $\bar{\mathfrak g}$
is indexed by $\bar\Lambda:=\Lambda/\Phi$.
Then there is a projection of simple roots
\begin{eqnarray*}
	\Lambda \rightarrow \bar\Lambda:
	i \mapsto  [i].
	\end{eqnarray*}
And we have a map between Chevalley basis of $\bar{\mathfrak{g}}$ and $\mathfrak{g}$ (see \cite[Proposition 7.9]{Kac})
\begin{equation}\label{Kacconstruction}
X_{[i]}\mapsto\sum_{\nu\in\Phi}
X_{\nu(i)} ,\,
 Y_{[i]}\mapsto
\sum_{\nu\in\Phi}Y_{\nu(i)},\quad i\in\Lambda,
\end{equation}
which naturally induces a Lie algebra isomorphism $\bar{\mathfrak{g}}\cong \mathfrak{g}^\Phi$ and an embedding $\mathcal{U}(\bar{\mathfrak{g}})\hookrightarrow (\mathcal{U}(\mathfrak{g}))^{\Phi}$.

Furthermore, the above embedding will induce an isomorphism between quantization of nilpotent orbit $\A=\mathcal{U}(\mathfrak{g})/J$. The following theorem is due to Huang:

\begin{lemma}[Huang {\cite[page 318]{Huang}}]\label{lem: Huang}
The  following algebras
\begin{eqnarray*}
(\A_{D_{n+1}})^{\mathbb Z_2},
\ (\A_{A_{2n-1}})^{\mathbb Z_2},
\ (\A_{E_6})^{\mathbb Z_2},
\ (\A_{D_4})^{\mathfrak S_3}
\end{eqnarray*}
are the filtered quantizations of
$\tilde{\mathcal{O}}_{ms}(B_n), \tilde{\mathcal{O}}_{ms}(C_n),
\tilde{\mathcal{O}}_{ms}(F_4), \tilde{\mathcal{O}}_{ms}(G_2)$
respectively.
\end{lemma}

Let us fix a notation: for $\mu=\sum_{i=1}^n k_i\alpha_i$,
$\alpha_i$ is a simple root, let $q^{\mu}:=\prod_{i=1}^nq_i^{k_i}$.
Now in the $Q(\mathscr A)$ of the minimal special nilpotent
orbit in $\mathfrak g$,
specifying $q_{i}=q_{[i]}$, we obtain the following.

\begin{proposition}\label{prop:minimalspecialinBCFG}
Denote by $\tilde{\mathcal{O}}_{ms}(B_n)$,
$\tilde{\mathcal{O}}_{ms}(C_n)$,
$\tilde{\mathcal{O}}_{ms}(F_4)$,
$\tilde{\mathcal{O}}_{ms}(G_2)$ the normalization
of the minimal special nilpotent orbits in Lie algebras of BCFG type respectively.
Then we have the following isomorphisms of D-modules of graded traces:
\begin{align*}
Q(\mathscr A[\tilde{\mathcal{O}}_{ms}(B_n)])
&\cong
\left(Q(\mathscr A[\overline{\mathcal O}_{min}(D_{n+1})])|_{q_{i}\rightarrow q_{[i]}}\right)^{\mathbb Z_2},\\
Q(\mathscr A[\tilde{\mathcal{O}}_{ms}(C_n)])
&\cong \left(Q(\mathscr A[\overline{\mathcal O}_{min}(A_{2n-1})])
|_{q_{i}\rightarrow q_{[i]}}\right)^{\mathbb Z_2},\\
Q(\mathscr A[\tilde{\mathcal{O}}_{ms}(F_4)])
&\cong \left(Q(\mathscr A[\overline{\mathcal O}_{min}(E_6)])
|_{q_{i}\rightarrow q_{[i]}}\right)^{\mathbb Z_2},\\
Q(\mathscr A[\tilde{\mathcal{O}}_{ms}(G_2)])
&\cong
\left(Q(\mathscr A[\overline{\mathcal O}_{min}(D_4)])
|_{q_{i}\rightarrow q_{[i]}}\right)^{\mathfrak S_3}.
\end{align*}
\end{proposition}	

To simplify the notations, let us denote the above isomorphisms to be
\begin{eqnarray*}
Q(\mathscr A_{BCFG})
\cong
\left(Q(\mathscr A_{ADE})|_{q_{i}\rightarrow
q_{[i]}}\right)^{\Phi}.	
\end{eqnarray*}
We prove the isomorphisms
in the rest of this subsection.

Recall that $Q(\A)=\A_0/\mathcal{I}_q$, where
$$\mathcal{I}_q
=\sum_{\mu\in\mathbb N\Delta^{+}}
S_{reg}\cdot \left\{a_\mu a_{-\mu}-q^\mu a_{-\mu}a_{\mu} | a_{\mu}\in\mathscr{A}_\mu,
a_{-\mu}
\in\mathscr{A}_{-\mu}\right\}\subset S_{reg}\otimes\mathscr A_0.$$

To avoid confusion, we denote by $\mathcal I_{[q]}$ the submodule
\[
\sum_{\mu\in \mathbb{N}\Delta^{+}}
S_{reg}\cdot \left\{1\otimes ab-q^{[\mu]}\otimes ba | a\in\mathscr{A}_\mu, b
\in\mathscr{A}_{-\eta}, [\mu]=[\eta]\right\}\subset S_{reg}\otimes\mathscr A_{[0]}.\]
\begin{lemma}\label{lem: weight change}
$(S_{reg}\otimes(\mathscr A_{ADE})_0)/\mathcal{I}_q
\cong(S_{reg}\otimes(\mathscr A_{ADE})_{[0]})/\mathcal{I}_{[q]}$,
\end{lemma}
\begin{proof}
The embedding $S_{reg}\otimes(\mathscr A_{ADE})_0\hookrightarrow S_{reg}\otimes\mathscr A_{[0]}$ and $\mathcal{I}_q\hookrightarrow I_{[q]}$ naturally gives us a morphism 
$$(S_{reg}\otimes(\mathscr A_{ADE})_0)/\mathcal{I}_q
\rightarrow(S_{reg}\otimes(\mathscr A_{ADE})_{[0]})/\mathcal{I}_{[q]}.$$
The injectivity of the morphism is induced by $\mathcal{I}_q=(S_{reg}\otimes(\mathscr A_{ADE})_0)\cap \mathcal{I}_{[q]}$. In the following we prove the surjectivity.

By the weight decomposition, we have
$(\mathscr A_{ADE})_{[0]}=(\mathscr A_{ADE})_0
\oplus\mathscr A'$. If we take $a\in S_{reg}
\otimes(\mathscr A_{ADE})_{[0]}$, we have $a=a_0+a'$
for $a_0\in S_{reg}\otimes(\mathscr A_{ADE})_0$,
$a'\in\mathscr A'$. Now, denote by $(\mathscr A')^k$
the degree $k$ part of $\mathscr A'$.
We claim that $a'\in S_{reg}\otimes\mathcal I_{[q]}$ and prove this claim by
 induction on $k$.
\begin{enumerate}
\item[(1)] For $k=0$, it is easy to see the claim holds.
		
\item[(2)] Suppose for $a'\in S_{reg}\otimes(\A')^{k-1}$, the claim holds. Now consider $a'\in S_{reg}\otimes(\A')^k$.
Take $a_\mu a_{-\eta}\in S_{reg}\otimes(\A')^k$, then we have
\begin{eqnarray*}
a_\mu a_{-\eta}=\dfrac{q^{[\mu]}}
{q^{[\mu]}-1}[a_\mu,a_{-\eta}]-\dfrac{1}{q^{[\mu]}-1}
(a_\mu a_{-\eta}-q^{[\mu]} a_{-\eta}a_\mu).
\end{eqnarray*}
Since $[a_\mu,a_{-\eta}]\in S_{reg}\otimes(\A')^{k-1}$
and $a_\mu a_{-\eta}-q^{[\mu]} a_{-\eta}a_\mu\in I_{[q]}$,
the claim holds for $a'\in S_{reg}\otimes(\A')^k$.
Then we know that the morphism is surjective.
\end{enumerate}
\end{proof}

By Lemma \ref{lem: Huang}, $(\mathscr A_{BCFG})\cong
(\mathscr A_{ADE})^{\Phi}$, which is denoted by $f$. Furthermore, since $f$ is induced by $\mathcal{U}(\mathfrak{g}_{BCFG})\hookrightarrow \mathcal{U}(\mathfrak{g}_{ADE})^{\Phi}$, it preserves the weight decomposition of BCFG type. Then $S_{reg}\otimes(\mathscr A_{BCFG})_{0}\cong S_{reg}\otimes(\mathscr A_{ADE})^\Phi_{[0]}$. By Lemma \ref{lem: weight change} and  exactness of the $\Phi$-invariant functor, we have
\begin{eqnarray*}
(Q(\A_{ADE})|_{q_i\rightarrow 
q_{[i]}})^\Phi=(S_{reg}\otimes(\mathscr A_{ADE})_{[0]})/\mathcal{I}_{[q]})^\Phi=\big(S_{reg}\otimes(\mathscr A_{ADE})_{[0]})\big)^\Phi/(\mathcal{I}_{[q]})^\Phi.	
\end{eqnarray*}
Since $S_{reg}\otimes(\mathscr A_{BCFG})_{0}\cong S_{reg}\otimes(\mathscr A_{ADE})^\Phi_{[0]}$, and it is direct to check $f((\mathcal{I}_{BCFG})_q)\subseteq\mathcal{I}_{[q]}^\Phi$, we obtain the following: 
\begin{lemma}
  There is a natural surjection  
  $$p: Q(\A_{BCFG})\rightarrow (Q(\A_{ADE})|_{q_i\rightarrow 
q_{[i]}})^\Phi.$$
\end{lemma}

In the rest of this subsection we will further prove $p$ is an isomorphism with the help of gradient module.

First notice that
the PBW filtration of $\mathcal{U}$ induces a filtration on $\A$:
$$\A^0\subset \A^1\subset\cdots\subset \A^k\subset\cdots\subset \A,$$
which makes $Q(\mathscr A)$
into a filtered $R_{reg}$-module.
The degree of polynomial ring $\mathrm{Sym}\A^1_0$ induces a filtered structure on $R_{reg}=S_{reg}\otimes\mathrm{Sym}\A^1_0$:
\[
R^0\subset R^1\subset\cdots\subset R^k\subset\cdots\subset R_{reg}.
\]
This makes $Q(\A)$ into a filtered $R_{reg}$-module, and makes $\mathrm{gr}(Q(\A))$ into a graded $\mathrm{gr}(R_{reg})$-module.

\begin{lemma}\label{lem: grading of Q(A)}
There are isomorphisms
\begin{eqnarray*}
 \mathrm{gr}(Q(\A_{ADE}))\cong S_{reg,ADE}\otimes(\mathbb{C}\oplus\mathfrak{h}_{ADE})	
\end{eqnarray*}
and
\begin{eqnarray*}
 \mathrm{gr}\big((Q(\A_{ADE})|_{q_i\rightarrow q_{[i]}})^\Phi\big)\cong
 S_{reg,BCFG}\otimes(\mathbb{C}\oplus\mathfrak{h}_{BCFG}),	
\end{eqnarray*}
where the $S_{reg}-$action is free and $h_i(1\otimes h_j)=0$.
\end{lemma}

\begin{proof}
	This is a corollary of Theorem \ref{thm: D-module of DE} and Theorem \ref{thm:quantum action An}.
\end{proof}

Set $M:=\mathrm{gr}(\A_0)/I$, where $I:=\sum_{\mu\in\mathbb N\Delta^{+}}
S_{reg}\cdot \left\{\bar{a}_\mu \bar{a}_{-\mu}| \bar{a}_{\mu}\in \mathrm{gr}(\mathscr{A})_\mu, \bar{a}_{-\mu}
\in\mathrm{gr}(\mathscr{A})_{-\mu}\right\}$.

\begin{lemma}
There is a natural surjection $\pi:M\rightarrow\mathrm{gr}(Q(\A))$.
\end{lemma}

\begin{proof}
Notice that
$\mathrm{gr}(Q(\A))=\mathrm{gr}(\A_0)/\mathrm{gr}(\mathcal{I}_q)$,
and by definition of $I$, we have $I\subseteq\mathrm{gr}(\mathcal{I}_q)\subseteq \mathrm{gr}(\A_0)$.
\end{proof}

Recall that we have a surjection of $R_{reg,BCFG}$-modules:
\[
p:Q(\A_{BCFG})\rightarrow (Q(\A_{ADE})|_{q_i\rightarrow q_{[i]}})^\Phi,
\]
which induces a surjection
\[
\mathrm{gr}(p):\mathrm{gr}\big(Q(\A_{BCFG})\big)\rightarrow \mathrm{gr}\big((Q(\A_{ADE})|_{q_i\rightarrow q_{[i]}})^\Phi\big).
\]
Now, we consider the composition
\[
\mathrm{gr}(p)\circ\pi:M_{BCFG}\rightarrow\mathrm{gr}\big((Q(\A_{ADE})|_{q_i\rightarrow q_{[i]}})^\Phi\big).
\]

\begin{proposition}\label{prop:isomorphism of R module}
$\mathrm{gr}(p)\circ\pi$ is an isomorphism of $\mathrm{gr}(R_{reg,BCFG})$-modules.
\end{proposition}

\begin{proof}
We first compute $M_{BCFG}=\mathrm{gr}\big((\A_{BCFG})_0\big)/I$.
Recall that by Lemma \ref{lem: Huang},
$\A_{BCFG}\cong\A_{ADE}^\Phi$; taking the associated graded algebra, we get
\begin{align*}
\mathrm{Sym}(\g_{BCFG})/\mathrm{gr}(J_{BCFG})
&\cong\big(\mathrm{Sym}(\g_{ADE})/(\mathrm{gr}(J_{ADE})\big)^\Phi\\
&=\big(\mathrm{Sym}(\g_{ADE})\big)^\Phi/\big((\mathrm{gr}(J_{ADE})\big)^\Phi.
\end{align*}	
This isomorphism is induced by
$\mathrm{Sym}(\g_{BCFG})\hookrightarrow(\mathrm{Sym}(\g_{ADE}))^\Phi$,
and thus
\begin{equation}\label{eq:JBCFG}
\mathrm{gr}(J_{BCFG})=\mathrm{gr}(J_{ADE})^\Phi\cap\mathrm{Sym}(\g_{BCFG})
=\mathrm{gr}(J_{ADE})\cap\mathrm{Sym}(\g_{BCFG}).
\end{equation}
From the definition of $I$, we have
\begin{align}
\mathrm{gr}((\A_{BCFG})_0)/I
&=\mathrm{Sym}(\g_{BCFG})|_{\mathfrak{h}^{*}_{BCFG}}
/\mathrm{gr}((J_{BCFG})_0)|_{\mathfrak{h}^{*}_{BCFG}}\nonumber\\
&=\mathrm{Sym}(\mathfrak{h}_{BCFG})/\mathrm{gr}((J_{BCFG})_0)|_{\mathfrak{h}^{*}_{BCFG}}.\label{eq:A_BCFG/I}
\end{align}
In this quotient, we have
\begin{align*}
\mathrm{gr}((J_{BCFG})_0)|_{\mathfrak{h}^{*}_{BCFG}}
&\stackrel{\eqref{eq:JBCFG}}=\mathrm{gr}((J_{ADE})\cap\mathrm{Sym}(\g_{BCFG}))|_{\mathfrak{h}^{*}_{BCFG}}\\
&=\mathrm{gr}((J_{ADE})|_{\mathfrak{h}^{*}_{ADE}}\cap\mathrm{Sym}(\mathfrak{h}_{BCFG}),
\end{align*}
which, by Theorem \ref{Sh2}, is further equal to
\begin{align}
\mathrm{Sym}^{\geq 2}(\mathfrak{h}_{ADE})\cap\mathrm{Sym}(\mathfrak{h}_{BCFG})
=\mathrm{Sym}^{\geq 2}(\mathfrak{h}_{BCFG}).\label{eq:hADE}
\end{align}
Thus combining \eqref{eq:A_BCFG/I} and \eqref{eq:hADE} we get
\begin{equation}\label{eq:MBCFG}
M_{BCFG}
=\mathrm{Sym}(\mathfrak{h}_{BCFG})/\mathrm{Sym}^{\geq 2}(\mathfrak{h}_{BCFG})
=S_{reg}\otimes(\mathbb C\oplus\mathfrak{h}_{BCFG}).
\end{equation}

On the other hand, by Lemma \ref{lem: grading of Q(A)},
\begin{equation}\label{QAADE}
\mathrm{gr}\big((Q(\A_{ADE})|_{q_i\rightarrow q_{[i]}})^\Phi\big)=S_{reg,BCFG}\otimes(\mathbb C\oplus \mathfrak h_{BCFG}).
\end{equation}
Notice that the identity of the right-hand sides
of \eqref{eq:MBCFG} and \eqref{QAADE} is exactly given by $\mathrm{gr}(p)\circ\pi$, which proves the proposition.
\end{proof}

\begin{proof}[Proof of Proposition \ref{prop:minimalspecialinBCFG}]
With the above notations, we have the following exact sequence
\begin{eqnarray*}
0\rightarrow\ker p\rightarrow Q(\A_{BCFG})\rightarrow (Q(\A_{ADE})|_{q_i\rightarrow q_{[i]}})^\Phi\rightarrow 0,
\end{eqnarray*}
 and then the following exact sequence
\begin{eqnarray*}
0\rightarrow \mathrm{gr}(\ker p)\rightarrow \mathrm{gr}(Q(\A_{BCFG}))
\rightarrow \mathrm{gr}((Q(\A_{ADE})|_{q_i\rightarrow q_{[i]}})^\Phi)\rightarrow 0.
\end{eqnarray*}	
On the other hand, by Proposition \ref{prop:isomorphism of R module} we have the following commutative diagram
\[
\xymatrix{
                & M_{BCFG} \ar[dl]_{\pi}\ar[dr]^{\cong}              \\
 \mathrm{gr}(Q(\A_{BCFG})) \ar[rr]_{\mathrm{gr}(p)} & &
 \mathrm{gr}\big( (Q(\A_{ADE})|_{q_i\rightarrow q_{[i]}})^\Phi\big),
 }
 \]
 which implies $\pi$ is isomorphism. Thefere $ \mathrm{gr}(p)$
 is an isomorphism and $\mathrm{gr}(\ker p)=0$.
 From this we get
 $\ker p=0$, and thus
 $p$ is an isomorphism.
\end{proof}

\begin{remark}
In \cite{Lo2}, Losev showed that the moduli spaces of these
quantizations are isomorphic to $\mathrm{H}^2(\widehat{\mathcal O}_{ms},\mathbb C)$,
where
$\widehat{\mathcal O}_{ms}$ denotes the smooth loci of the $\mathbb Q$-terminalizations
of $\widetilde{\mathcal O}_{ms}$.
By the works of Fu \cite{Fu2} and Namikawa \cite{Na}, for Lie algebras of types CDEFG,
$\widehat{\mathcal O}_{ms}=\mathcal O_{ms}$, and therefore
the quantizations of these nilpotent orbits are unique since
$\mathrm H^2(\mathcal O_{ms},\mathbb C)\cong\{0\}$ (see \cite{BC} and \cite{Ju}).
Thus these quantizations are isomorphic to the ones presented in the paper.
For type A Lie algebras, since $\widehat{\mathcal O}_{ms}\cong
T^*\mathbb{P}^{n}$ by \cite{Fu1},
the quantizations of their minimal
nilpotent orbits are parameterized by $\mathrm H^2(T^*\mathbb{P}^{n},\mathbb C)\cong\mathbb C$,
and hence are also isomorphic to the ones given in \S\ref{subsect:JosephintypeA}.
For type B Lie algebras, we have only given a quantization
of the minimal special nilpotent orbits that comes from the
one in Lie algebras of type D.
However, in this case, the quantizations are not unique, since the moduli space
is $\mathrm H^2(\widehat{\mathcal O}_{ms},\mathbb C)
\cong\mathrm H^2(T^*\mathbb{Q}^{2n-1},\mathbb C)\cong\mathbb C$
by \cite{Fu1} (see also \cite[Example 1.1]{FRW}).
\end{remark}

\subsection{Equivariant cohomology of the minimal resolutions}

Now let us turn to the Slodowy slices in the BCFG type Lie algebras.
In \cite{S}, Slodowy showed that the intersections of
Slodowy slices to the subregular
nilpotent orbit with the nilpotent cone
are characterized by a pair of subgroups $\Gamma, \Gamma'$
in $\mathrm{SL}_2(\mathbb C)$, where $\Gamma$ is a normal subgroup
of $\Gamma'$. It is given by the following table:
\begin{center}
\begin{tabular}{p{3.5cm}p{3cm}p{3cm}p{2cm}}
\hline
Type of Lie algebra & Type of $\Gamma$ & Type of $\Gamma'$ & $\Gamma'/\Gamma$\\
\hline\hline
$B_n$ & $A_{2n-1}$ & $D_{n+2}$
& $\mathbb Z_2$\\
$C_{n}$ & $D_{n+1}$ & $D_{2n}$ &
 $\mathbb Z_2$\\
 $F_4$ & $E_6$ & $E_7$ &
 $\mathbb Z_2$\\
  $G_2$ & $D_4$ & $E_7$ &
 $\mathfrak S_3$\\
 \hline
\end{tabular}
\end{center}
Slodowy proved the following.

\begin{theorem}[{\cite{S}}]
Suppose $\mathfrak g$ is a Lie algebra of BCFG type
and $\Gamma, \Gamma'\in \mathrm{SL}_2(\mathbb C)$ are given in the above table.
Then for any $x$ in the subregular nilpotent orbit,
we have the following isomorphism:
\begin{eqnarray*}
S_{x}\cap\mathcal N\cong \mathbb C^2/\Gamma,
\end{eqnarray*}
under which the action of $\Gamma'/\Gamma$ on $\mathbb C^2/\Gamma$
corresponds to the action of the stabilizer, denoted by $G_{x,y}$, of
$x$ and $y$ (recall from Definition \ref{def:Slodowyslice}
that $\{x,h,y\}$ is the $\mathfrak{sl}_2(\mathbb C)$ triple) on $S_{x}\cap\mathcal N$.
\end{theorem}

More precisely, in local coordinates (recall the table in
\S\ref{subsect:Kleiniansing}),
\begin{enumerate}
\item[$-$] for the $A_{2n-1}$ singularity $x^{2n}-yz=0$,
the $\mathbb Z_2$-action is given by
$x\mapsto -x,\, y\mapsto -y,\, z\mapsto -z.$

\item[$-$] for the $D_{n+1}$ singularity $xy^2+x^n+z^2=0$,
the $\mathbb Z_2$-action is given by
$x\mapsto x,\, y\mapsto y,\, z\mapsto -z.$

\item[$-$] for the $E_6$ singularity $x^4+y^3+z^2=0$,
the $\mathbb Z_2$-action is given by
$
x\mapsto-x, \, y\mapsto y,\,z\mapsto -z.
$

\item[$-$] for the $D_4$ singularity $xy^2+x^3+z^2=0$,
the $\mathfrak S_3$-action is generated by
$
x\mapsto(-x+\sqrt{-1}y)/2, \, y\mapsto (3\sqrt{-1}x-y)/2,\,z\mapsto z
$
and $x\mapsto x,\,y\mapsto -y,\,z\mapsto -z$.
\end{enumerate}
According to Slodowy, these varieties together with the above symmetries
are called {\it simple singularities} of types $B_n$, $C_n$, $F_4$ and $G_2$
respectively.
Slodowy also showed in \cite[\S6.2]{S} that the group actions lift to the minimal
resolutions of these singularities, and the lifted actions on the irreducible
components on the exceptional fibers are exactly identical to the
ones on the associated Dynkin diagrams, described in the previous subsection.
The following definition is now reasonable.

\begin{definition}
Let $\mathcal B_n$, $\mathcal C_n$, $\mathcal F_4$
and $\mathcal G_2$ be the minimal resolutions of
$B_n$, $C_n$, $F_4$ and $G_2$ respectively.
Their equivariant cohomology algebras
are defined to be
$
\mathrm{H}^\bullet_{\mathbb Z_2\times
\mathbb C^\times}(\mathcal B_n)$,
$
\mathrm{H}^\bullet_{\mathbb Z_2\times\mathbb C^\times}(\mathcal C_n)
$,
$
\mathrm{H}^\bullet_{\mathbb Z_2\times\mathbb C^\times}(\mathcal F_4)
$
and
$
\mathrm{H}^\bullet_{\mathfrak S_3\times\mathbb C^\times}(\mathcal G_2)
$
respectively.
\end{definition}

Next, we turn to the equivariant quantum cohomology, which we denote
by $\mathrm{QH}_R^\bullet $. Let us first recall that
in \cite{BG}, Bryan and Gholampour construct,
for any irreducible and reduced root system,
a Frobenius algebra $(\mathrm {QH}^\bullet_{R}, \star)$
which generalizes Theorem \ref{thm:BryanGholam2}.
Let us go over their construction.

Let $R$ be an irreducible and reduced rank $n$
root system and $\{\alpha_1,\cdots,\alpha_n\}$ be a system
of simple roots.
Let $\mathrm H_R=\mathbb Z\oplus \mathbb Z\alpha_1
\oplus\cdots\oplus\mathbb Z\alpha_n$ and let
$\mathrm{QH}_R^\bullet=\mathrm H_R\otimes
\mathbb Z[t][\![q_1,\cdots, q_n]\!]$.
We associate a group $\Gamma\in\mathrm{SL}_2(\mathbb C)$ listed in the table
in \S\ref{subsect:Kleiniansing} as follows:
the type of $\Gamma$ is the same as the type of a simply-laced Lie algebra
which itself is ADE or
which gives the Lie algebra of BCFG type via the isomorphism \eqref{Kacconstruction}.

\begin{definition}[Bryan and Gholampour \cite{BG}]
Define a product $\star$ on $\mathrm{QH}_R^\bullet$ as follows:
\begin{eqnarray*}
e_\alpha\star e_{\alpha'}=-t^2|\Gamma|\langle e,e'\rangle
+\sum_{\gamma\in\Delta^+}t\langle\alpha,\gamma\rangle\langle\alpha',
\gamma^\vee\rangle\dfrac{1+q^\gamma}{1-q^\gamma}  e_\gamma,
\end{eqnarray*}
 where $\gamma^\vee=\dfrac{2}{\langle \gamma,\gamma\rangle}\gamma$,
%$\Gamma$ is corresponding to $\mathcal A_{2n-1}$,
%$\mathcal D_{n+1}$, $\mathcal E_6$ and $\mathcal D_4$ respectively,
$e_\gamma=c_1e_1+\cdots+c_n e_n$
if the root $\gamma=c_1\alpha_1+\cdots +c_n \alpha_n$
with $\alpha_1,\cdots, \alpha_n$ being the simple roots,
$\alpha, \alpha'$ being the positive roots corresponding to
$e_\alpha$ and $e_{\alpha'}$, and
$\langle -,-\rangle$ is the inner product in the root system.	
\end{definition}

Bryan and Gholampour showed that the product thus defined
is associative and $(\mathrm{QH}_R^\bullet,\star)$ forms a Frobenius algebra
(see \cite[Theorem 6]{BG}).
Observe that if $R$ is of ADE type, $\mathrm{QH}_R^\bullet$
is exactly the $\mathbb C^\times$-equivariant quantum cohomology
of the minimal resolution of the ADE singularity (see Theorem \ref{thm:BryanGholam2}).

Now if $R$ is of BCFG type, to specify the particular root system,
let us denote $\mathrm{QH}_R^\bullet$ by
$\mathrm{QH}_R^\bullet(\mathcal B_n)$,
$\mathrm{QH}_R^\bullet(\mathcal C_n)$,
$\mathrm{QH}_R^\bullet(\mathcal F_4)$
and
$\mathrm{QH}_R^\bullet(\mathcal G_2)$.
We first show that they are the ``equivariant quantum" cohomology
of $\mathcal B_n$, $\mathcal C_n$, $\mathcal F_4$ and $\mathcal G_2$ respectively; that is,
they are the deformations of the equivariant cohomology:

\begin{proposition}\label{quantumcohofBCFG}
There are the following isomorphisms
\begin{eqnarray*}
\begin{array}{ll}
\mathrm {QH}^\bullet_{R}
(\mathcal B_n)|_{q= 0}\cong
\mathrm{H}^\bullet_{\mathbb Z_2\times\mathbb C^\times}(\mathcal B_n),&
\mathrm{QH}^\bullet_{R}(\mathcal C_n)|_{q=0}\cong
\mathrm{H}^\bullet_{\mathbb Z_2\times\mathbb C^\times}(\mathcal C_n),\\[2mm]
\mathrm{QH}^\bullet_{R}
(\mathcal F_4)|_{q=0}\cong
\mathrm{H}^\bullet_{\mathbb Z_2\times\mathbb C^\times}(\mathcal F_4),&
\mathrm{QH}^\bullet_{R}
(\mathcal G_2)|_{q= 0}\cong
\mathrm{H}^\bullet_{{\mathfrak S_3}\times\mathbb C^\times}(\mathcal G_2)
\end{array}	
\end{eqnarray*}
of algebras over $\mathbb C[\hbar]$.
\end{proposition}

\begin{proof}
We show the first isomorphism.
In fact,
\begin{align*}
\mathrm{QH}_R^\bullet(\mathcal B_n)|_{q=0}
&\cong
\left(\mathrm{QH}^\bullet_{\mathbb C^\times}
(\mathcal A_{2n-1})|_{q_i=0}\right)^{\mathbb Z_2}\\
&\cong
\mathrm{H}_{\mathbb C^\times}(\mathcal A_{2n-1})^{\mathbb Z_2}\\
&\cong\mathrm H_{\mathbb Z_2\times
\mathbb C^\times}(\mathcal A_{2n-1})\\
&\cong
\mathrm H_{\mathbb Z_2\times\mathbb C^\times}(\mathcal B_n).
\end{align*}
The rest isomorphisms are proved similarly,
and we leave them to the interested reader.
\end{proof}

Next we relate the these equivariant quantum cohomology
with those of ADE resolutions.
Take $\mathcal B_n$ and $\mathcal A_{2n-1}$, for example.
Identify the generators of their equivariant quantum cohomology
with the simple roots in the Dynkin diagram.
Recall that $\Lambda=\{1,2,\cdots, 2n-1\}$ be the index set
for the simple roots of $A_{2n-1}$,
and let $\bar\Lambda=\Lambda/\Phi$, where
$\nu$ is the nontrivial automorphism
of the Dynkin diagram.
Specialize the quantum variables
$\{q_i\}_{i\in \Lambda}$ to variables
$\{q_{[i]}\}_{[i]\in \bar\Lambda}$ by setting
$q_i= q_{[i]}$. Then we obtain a map
$$
\left(\mathrm{QH}_{\mathbb C^\times}
(\mathcal A_{2n-1})|_{q_i\rightarrow q_{[i]}}\right)^{\mathbb Z_2}
\to \mathrm{QH}_R^\bullet(\mathcal B_n):
\frac{1}{2}
(e_{i}+e_{\nu(i)})\mapsto e_{[i]}.
$$
For the other singularities, proceed analogously
and we obtain the following.

\begin{proposition}[{\cite[\S4.3]{BG}}]\label{prop:SlodowyinBCFG}
There are the following isomorphisms
\begin{eqnarray*}
\begin{array}{ll}
\mathrm {QH}^\bullet_{R}
(\mathcal B_n)\cong
\left(\mathrm{QH}^\bullet_{\mathbb C^\times}
(\mathcal A_{2n-1})|_{q_i\rightarrow q_{[i]}}\right)^{\mathbb Z_2},&
\mathrm{QH}^\bullet_{R}
(\mathcal C_n)\cong \left(
\mathrm{QH}^\bullet_{\mathbb C^\times}
(\mathcal D_{n+1})|_{q_i\rightarrow q_{[i]}}\right)^{\mathbb Z_2},\\[2mm]
\mathrm{QH}^\bullet_{R}
(\mathcal F_4)\cong \left(
\mathrm{QH}^\bullet_{\mathbb C^\times}
(\mathcal E_6)|_{q_i\rightarrow q_{[i]}}\right)^{\mathbb Z_2},&
\mathrm{QH}^\bullet_{R}
(\mathcal G_2)\cong \left(
\mathrm{QH}^\bullet_{\mathbb C^\times}(\mathcal D_4)|_{q_i\rightarrow q_{[i]}}\right)^{\mathfrak S_3},
\end{array}	
\end{eqnarray*}
which are compatible with the quantum product $\star$.
\end{proposition}

\begin{proof}
See \cite[Theorem 6]{BG}.
\end{proof}

\begin{remark}
Notice that the Cartan matrix elements
in \cite{BG} are given by $n_{\alpha,\beta}=\dfrac{2\langle\alpha,
\beta\rangle}{\langle\alpha,\alpha\rangle}$, while in this paper we use the notation of \cite{Bo},  where
$n_{\alpha,\beta}=\dfrac{2\langle\alpha,\beta\rangle}{\langle\beta,\beta\rangle}$.
\end{remark}

\subsection{Proof of Theorems \ref{maintheorem2} and \ref{maintheorem3}}

We are now ready to prove Theorem \ref{maintheorem2}.

\begin{proof}[Proof of Theorem \ref{maintheorem2}]

Combining Theorem \ref{maintheorem0}, Propositions \ref{prop:minimalspecialinBCFG} and \ref{prop:SlodowyinBCFG},
we get the desired isomorphisms.
\end{proof}

Now we consider the extremal situation where $q=0$.
First, we have the following.

\begin{proposition}\label{BalgebraofBCFG}
There are isomorphisms of $B$-algebras:
$$
\begin{array}{ll}
B(\mathscr A[\tilde{\mathcal{O}}_{ms}(B_n)])
\cong B(\mathscr A[\overline{\mathcal O}_{min}(D_{n+1})])^{\mathbb Z_2},&
B(\mathscr A[\tilde{\mathcal{O}}_{ms}(C_n)])
\cong B(\mathscr A[\overline{\mathcal O}_{min}(A_{2n-1})])^{\mathbb Z_2},\\[2mm]
B(\mathscr A[\tilde{\mathcal{O}}_{ms}(F_4)])
\cong B(\mathscr A[\overline{\mathcal O}_{min}(E_6)])^{\mathbb Z_2},&
B(\mathscr A[\tilde{\mathcal{O}}_{ms}(G_2)])\cong
B(\mathscr A[\overline{\mathcal O}_{min}(D_4)])^{\mathfrak S_3}.
\end{array}
$$
\end{proposition}

\begin{proof}
Analogous to the proof of Proposition \ref{prop:minimalspecialinBCFG}.
\end{proof}

\begin{proof}[Proof of Theorem \ref{maintheorem3}]
In Theorem \ref{maintheorem2}, let $q=0$, then by Propositions
\ref{quantumcohofBCFG} and \ref{BalgebraofBCFG},
the specialized quantum D-module reduces to the
equivariant cohomology, and the D-module of graded traces $Q(\A(-))$ reduces
to the ring $B(\A(-))$. The theorem follows.
\end{proof}

\appendix
\section{Proof of Lemma \ref{lemmma:dimV0}}\label{appendix}
In this section we prove Lemma \ref{lemmma:dimV0}. The main tool we use is the following formula
(see \cite[Theorem 22.3]{Hum}).

\begin{lemma}[Freudenthal]
Let $V=V(\lambda)$ be an irreducible $\g$-module of highest weight $\lambda$.
Let $\Lambda$ be the set of
weights of $V$. For $\mu\in \Lambda$,
set the multiplicity $m(\mu)$ to be the dimension of the weight space $V^{\mu}$,
 then $m(\mu)$ is given recursively as follows:
\begin{equation}\label{Freudenthal}
	\big((\lambda+\delta, \lambda+\delta)-(\mu+\delta, \mu+\delta)
	\big)m(\mu)
	=2\sum_{\alpha\in \Delta^+}\sum_{i=1}^{+\infty}m(\mu+i\alpha)(\mu+i\alpha, \alpha),
\end{equation}
where $\delta=\frac{1}{2}\sum_{\alpha\in \Delta^+}\alpha$.
\end{lemma}

\begin{proof}[Proof of Lemma \ref{lemmma:dimV0}]
Notice that $\dim V(2\theta)_0$ is just the multiplicity $m(0)$ in $ V(2\theta)$.
We prove the lemma case by case.

\noindent\textbf{The $A_n$ case:}
Firstly we list some data in $A_n$ case (see \cite{Hum} or \cite{Kac}).
\begin{align*}
Q&=\left\{\sum_{i=1}^{n+1}k_i\varepsilon_i|k_i\in \Z, \sum_i k_i=0\right\},\\
\Delta &=\{\varepsilon_i-\varepsilon_j\}, \quad \Delta^+ =\{\varepsilon_i-\varepsilon_j|i<j \} ,\\
\Pi&=\{\alpha_1=\varepsilon_1-\varepsilon_2, \alpha_2=\varepsilon_2-\varepsilon_3,
\cdots , \alpha_n=\varepsilon_n-\varepsilon_{n+1}\},\\
\theta &=\varepsilon_1-\varepsilon_{n+1},\quad
\delta=\frac{1}{2}(n\varepsilon_1+(n-2)\varepsilon_2
+\cdots -(n-2)\varepsilon_n-n\varepsilon_{n+1}),\\
W &=\{\text{all permutations of the $\varepsilon_i$}\}.
\end{align*}
 Since $2\theta=2(\varepsilon_1-\varepsilon_{n+1})$ is the highest weight of $V(2\theta)$,
 $m(2\theta)=1$. Since $2\theta=2(\varepsilon_1-\varepsilon_{n+1})$, and $m(\mu)$ is
 invariant under the $W$-action (see \cite[Theorem 21.2]{Hum}), we have
 \begin{equation}\label{117}
 	m(2(\varepsilon_i-\varepsilon_j))=1.
 \end{equation}
 Now we consider $m(2\varepsilon_1-\varepsilon_n-\varepsilon_{n+1})$. By \eqref{Freudenthal}, we have
 \begin{align*}
 &\big((2\theta+\delta, 2\theta+\delta)-(2\varepsilon_1-\varepsilon_n-\varepsilon_{n+1}
 +\delta, 2\varepsilon_1-\varepsilon_n-\varepsilon_{n+1}+\delta)\big)
 m(2\varepsilon_1-\varepsilon_n-\varepsilon_{n+1})\\
=&2m(2\theta)(2\theta, \varepsilon_{n}-\varepsilon_{n+1}).
 \end{align*}
 One can check that
 \begin{align*}
 &(2\theta, \varepsilon_{n}-\varepsilon_{n+1})=2,\\
 &(2\theta+\delta, 2\theta+\delta)-(2\varepsilon_1-\varepsilon_n-\varepsilon_{n+1}
+\delta, 2\varepsilon_1-\varepsilon_n-\varepsilon_{n+1}+\delta)=4.
 \end{align*}
Therefore
 \begin{equation*}
 	m(2\varepsilon_1-\varepsilon_n-\varepsilon_{n+1})=1.
 \end{equation*}
By the $W$-invariance of $m(\mu)$ and $m(\mu)=m(-\mu)$, we have
\begin{equation}\label{118}
	m(\pm(2\varepsilon_i-\varepsilon_j-\varepsilon_k))=1.
\end{equation}
Now we consider $m(\varepsilon_1+\varepsilon_2-\varepsilon_{n}-\varepsilon_{n+1})$.
By \eqref{Freudenthal}, we have
\begin{align*}
&\big(
(2\theta+\delta, 2\theta+\delta)-(\varepsilon_1+\varepsilon_2-\varepsilon_{n}-\varepsilon_{n+1}
+\delta, \varepsilon_1+\varepsilon_2-\varepsilon_{n}-\varepsilon_{n+1}+\delta)\big)\\
&\cdot m(\varepsilon_1+\varepsilon_2-\varepsilon_{n}-\varepsilon_{n+1})\\
=&2\big(m(2\varepsilon_1-\varepsilon_{n}-\varepsilon_{n+1})(2\varepsilon_1-\varepsilon_{n}-\varepsilon_{n+1} ,
\varepsilon_1-\varepsilon_2)\\
&+m(\varepsilon_1+\varepsilon_2-2\varepsilon_{n+1})(\varepsilon_1
+\varepsilon_2-2\varepsilon_{n+1}, \varepsilon_n-\varepsilon_{n+1})\big).
\end{align*}
By \eqref{118}, we have $m(2\varepsilon_1-\varepsilon_{n}-\varepsilon_{n+1})
=m(\varepsilon_1+\varepsilon_2-2\varepsilon_{n+1})=1$. Furthermore,
\begin{align*}
&(2\varepsilon_1-\varepsilon_{n}-\varepsilon_{n+1} , \varepsilon_1-\varepsilon_2)
=(\varepsilon_1+\varepsilon_2-2\varepsilon_{n+1}, \varepsilon_n-\varepsilon_{n+1})=2,\\
&(2\theta+\delta, 2\theta+\delta)-(\varepsilon_1+\varepsilon_2-\varepsilon_{n}
-\varepsilon_{n+1}+\delta, \varepsilon_1+\varepsilon_2-\varepsilon_{n}-\varepsilon_{n+1}+\delta)=8.
\end{align*}
Thus $m(2\theta-\alpha_1-\alpha_n)=1$ by the $W$-invariance of $m(\mu)$, and we have
\begin{equation}\label{119}
	m(\varepsilon_i+\varepsilon_j-\varepsilon_k-\varepsilon_l)=1.
\end{equation}
Now we calculate $m(\theta)$. By \eqref{Freudenthal},
\begin{equation*}
\big((2\theta+\delta, 2\theta+\delta)-(\theta+\delta, \theta+\delta)\big)
m(\theta)=2\sum_{\alpha\in \Pi} m(\theta+\alpha)(\theta+\alpha, \alpha).
\end{equation*}
By \eqref{117}, \eqref{118} and \eqref{119}, $m(\theta+\alpha)=1$. Furthermore, we have
\begin{align*}
	&(2\theta+\delta, 2\theta+\delta)-(\theta+\delta, \theta+\delta)=6+2n,\\
	& \sum_{\alpha\in \Pi}(\theta+\alpha, \alpha)=(\theta, 2\delta)+2|\Pi|=2n+2\cdot
	\frac{n(n+1)}{2} =n(n+3).
\end{align*}
Then $m(\theta)=n$ and by the $W$-invariance of $m(\mu)$,
\begin{equation}\label{120}
	m(\varepsilon_i-\varepsilon_j)=n.
\end{equation}
Finally, by \eqref{Freudenthal},
\begin{align}\label{A-m(0)}
	\big((2\theta+\delta, 2\theta+\delta)-(\delta, \delta)\big)m(0)
		=2\sum_{\alpha\in \Pi} \big(m(\alpha)(\alpha , \alpha)+m(2\alpha)(2\alpha, \alpha)\big).
\end{align}
By \eqref{117} and \eqref{120}, we have $m(\alpha)=n$ and $m(2\alpha)=1$. Furthermore,
$$(2\theta+\delta, 2\theta+\delta)-(\delta, \delta)=4n+8.$$
Thus \eqref{A-m(0)} is equivalent to
$$(4n+8)m(0)=2(2n+4)|\Pi|=4(n+2)\cdot\frac{n(n+1)}{2},$$
which induces
\begin{equation*}
m(0)=\frac{n(n+1)}{2}.
\end{equation*}

\noindent\textbf{The $D_n$ case:} The data of $D_n$ is as follows:
\begin{align*}
Q&=\left\{\sum_{i=1}^{n}k_i\varepsilon_i|k_i\in \Z, \sum_i k_i\in 2\Z\right\},\\
\Delta &=\{\pm \varepsilon_i \pm \varepsilon_j\},\quad \Delta^+=\{\varepsilon_i \pm \varepsilon_j| i<j\},\\
\Pi&=\{\alpha_1=\varepsilon_1-\varepsilon_2, \alpha_2=\varepsilon_2-\varepsilon_3, \cdots , \alpha_{n-1}
=\varepsilon_{n-1}-\varepsilon_{n}, \alpha_n=\varepsilon_{n-1}+\varepsilon_n\},\\
\theta
&=\varepsilon_1+\varepsilon_{2}, \quad \delta=(n-1)\varepsilon_1+(n-2)\varepsilon_2+\cdots+ \varepsilon_{n-1}, \\
W &=\{\text{all permutations and even number of sign changes of the $\varepsilon_i$}\}.
\end{align*}
The argument is similar to $A_n$, so we just list the result and omit the details:
\begin{align*}
&m(2\theta)=m(2(\varepsilon_1+\varepsilon_2))=m(\pm2(\varepsilon_i\pm \varepsilon_j))=1,\\
&m(2\varepsilon_1+\varepsilon_2+\varepsilon_3)=m(\pm 2\varepsilon_i\pm \varepsilon_j \pm \varepsilon_k)=1,\\
&m(\varepsilon_1+\varepsilon_2+\varepsilon_3+\varepsilon_4)=
m(\pm \varepsilon_i\pm \varepsilon_j\pm \varepsilon_k\pm \varepsilon_l)=2,\\
&m(2\varepsilon_1)=m(\pm \varepsilon_i)=n-2,\\
&m(\varepsilon_1+\varepsilon_2)=m(\pm \varepsilon_i\pm \varepsilon_j)=2n-3,\\
&m(0)=n(n-1).
\end{align*}
	
\noindent\textbf{The type E case:} By \cite[\S4]{BM}, we know that
for $E_6$, $m(0)=36$; for $E_7$, $m(0)=63$ and for $E_8$, $m(0)=120$.
They are exactly $\displaystyle\frac{\dim\g-\dim\h}{2}$ in these cases.

In summary, in all the ADE cases, we have $\displaystyle m(0)=\frac{\dim\g-\dim\h}{2}$.
\end{proof}

\end{document}